\documentclass{amsart}

\usepackage{Preset}
\usepackage[makeroom]{cancel}
\usepackage{nccmath}

\usepackage{mathrsfs}
\usepackage[latin1]{inputenc}
\usepackage[T1]{fontenc}
\usepackage{tikz}
\usepackage{tikz-cd}
\usepackage[colorlinks=true,linkcolor=blue!50!black,anchorcolor=red,citecolor=blue,filecolor=black,menucolor=black,runcolor=black,urlcolor=black]{hyperref}

\usetikzlibrary{decorations.pathmorphing}
\usetikzlibrary{shapes,arrows,graphs}
\setlist[enumerate,1]{label={\textnormal{(\arabic*)}}}

\newcommand{\seq}[1]{\left\langle #1 \right\rangle}

\title[Lebesgue measure of the postcritical set]{Lebesgue measure of the postcritical set of neutral quadratic polynomials}
\author{Willie Rush Lim}
\address{Dept. of Mathematics, Brown University, RI 02912}
\email{willie\_rush\_lim@brown.edu}
\date{}

\begin{document}

\begin{abstract}
    We prove that the postcritical set of every quadratic polynomial with a neutral fixed point has zero Lebesgue measure. 
    In particular, no arithmetic condition on the rotation number is required. 
    Our proof uses sector renormalization and the pseudo-Siegel bounds of Dudko-Lyubich to derive uniform distortion estimates for the renormalization change of variables. 
    For sector renormalization towers, we also prove the same zero area statement for the associated postcritical set and the optimal Brjuno criterion for the interior of the associated Mother Hedgehog.
\end{abstract}

\maketitle

\setcounter{tocdepth}{1}
\tableofcontents

\section{Introduction}
\label{sec:intro}

\subsection{On neutral quadratic polynomials}

Consider the family of neutral quadratic polynomials
\[
    f_\theta(z) := e^{2\pi i \theta} z + z^2, \qquad \theta \in \R / \Z.
\]
It is the simplest family of global holomorphic maps that are non-linear and admit a neutral fixed point at $0$.
This family exhibits some of the most delicate behavior in one-dimensional complex dynamics.
It sits at the boundary between rigid rotation-like behavior and genuinely chaotic behavior.
Depending on the arithmetic properties of $\theta$, the neutral fixed point at 0 may be locally linearizable or non-linearizable, and the associated dynamical sets can have very different geometric properties.

The map $f_\theta$ has a unique critical point at 
\[
c_\theta := e^{-2\pi i \theta}/2.
\]
The \emph{postcritical set} of $f_\theta$ is defined as the closure of the critical orbit:
\[
    P(f_\theta) := \overline{\{ f_\theta^n(c_\theta)\}_{n\geq 1}}.
\]
The topology of $P(f_\theta)$ varies wildly depending on the arithmetic properties of $\theta$.
In complex dynamics, the postcritical set generally determines the global dynamical behavior of the system.
For instance, for $f_\theta$, it is the global measure-theoretic attractor in the following sense.
For almost every point $z$ in the Julia set of $f_\theta$, the orbit of $z$ under $f_\theta$ converges to $P(f_\theta)$.

The seminal work of Buff-Ch\'eritat \cite{BC12} asserts that the Julia set of $f_\theta$ can have positive area.
It is natural to ask whether a similar property holds for the postcritical set of $f_\theta$.
Our main result gives a complete answer to this question.

\begin{thmx}[Zero area]
    \label{main-thm-01}
    For every $\theta \in \R / \Z$, the postcritical set $P(f_\theta)$ of $f_\theta$ has zero Lebesgue measure.
\end{thmx}

This theorem is sharp in the sense that there exist examples of irrational numbers $\theta$ such that $P(f_\theta)$ has Hausdorff dimension two \cite{CDY}.
Bearing in mind the role of $P(f_\theta)$ as the measure-theoretic attractor, this theorem immediately implies:

\begin{corx}
\label{cor:non-recurrent}
    For every $\theta \in \R / \Z$, almost every point in the Julia set of $f_\theta$ is non-recurrent.
\end{corx}

By Poincar\'e recurrence theorem, we also have:

\begin{corx}
\label{cor:no-acip}
    For every $\theta \in \R / \Z$, the Julia set of $f_\theta$ admits no absolutely continuous invariant probability measure.
\end{corx}

When $\theta$ is rational, Theorem \ref{main-thm-01} is trivial because the critical orbit simply converges to the parabolic fixed point. 
When the rotation number $\theta$ is irrational, $P(f_\theta)$ is contained in the Julia set $J(f_\theta)$.
In this case, previous results established zero area under various arithmetic assumptions, particularly on the regular continued fraction expansion $\{a_n\}_{n\geq 1}$ of $\theta$.
\begin{enumerate}[leftmargin=0.4in]
    \item[(i)] \emph{Bounded type}: $\sup_n a_n < \infty$ \\
    Under this condition, the quasiconformal surgery designed by Douady and Ghys \cite{D87,G84} tells us that $P(f_\theta)$ is a quasicircle making up the boundary of the Siegel disk of $f_\theta$, and in particular, has zero area. 
    By using puzzle techniques and the circle model coming from the surgery, Petersen \cite{Pe96} further proved that $J(f_\theta)$ of $f_\theta$ is locally connected and has zero area. 
    \item[(ii)] \emph{PZ type}: $\log a_n = O(\sqrt{n})$ \\
    This condition holds for almost every irrational $\theta$. 
    Petersen and Zakeri \cite{PZ04} showed that this condition allows the activation of trans-quasiconformal surgery, a generalization of Douady-Ghys surgery, and proved that $J(f_\theta)$ is locally connected and has zero area. 
    \item[(iii)] \emph{High type}: $\inf_n a_n \gg 1$ \\
    Cheraghi \cite{Che13,Che19} proved zero area using the near-parabolic renormalization theory pioneered by Inou and Shishikura \cite{IS}.
    Unlike the first two types, this regime includes cases where $J(f_\theta)$ fails to be locally connected.
\end{enumerate}
In contrast, Theorem \ref{main-thm-01} works for all rotation numbers without imposing any arithmetic conditions.

The central difficulty in proving uniform results like Theorem \ref{main-thm-01} is that the small-scale geometry changes dramatically as the continued fraction entries vary.
To overcome this, we apply the novel breakthrough of Dudko and Lyubich \cite{DL22,DL26b} on uniform \emph{a priori bounds} for neutral renormalization.
The next two subsections further elaborate this new machinery.
Theorem \ref{main-thm-01} and its corollaries are indeed one of the first concrete applications of their machinery.

\subsection{The renormalization mechanism}

A key development in neutral dynamics has been the theory of sector renormalization. 
For a neutral holomorphic germ $g(z) =e^{2\pi i \theta}z + O(z^2)$, one can construct a fundamental sector near the fixed point at $0$ and conformally glue its sides using the dynamics; see Figure \ref{fig:sector-renormalization-local} for an illustration.
The first return map then produces a new neutral holomorphic germ $\Rsec g$. 
The procedure can be iterated to a sequence of renormalized maps $(\Rsec)^n g$.

Sector renormalization first appeared in the work of Douady and Ghys \cite{D87} in which they proved that the set 
\[
\mathcal{L} = \left\{ \theta \in \R \backslash \Q \: : \: \begin{array}{c}
\textnormal{every germ } g(z)=e^{2\pi i \theta}z+O(z^2) \\
\textnormal{is analytically linearizable near }0 
\end{array}
\right\}
\]
is invariant under the action of $\text{GL}(2,\Z)$.
An irrational number $\theta$ is called \emph{Brjuno} if the sum
\begin{equation}    
    \label{eqn:brjuno-original}
    \Brjuno(\theta) = \sum_{n=0}^\infty \frac{\log q_{n+1}}{q_n}
\end{equation}
converges, where $q_n$ is the regular $n$\textsuperscript{th} denominator continuant of $\theta$.
In \cite{Br71}, Brjuno studied the formal power series of linearization and proved that the set $\mathcal{L}$ contains the set of Brjuno numbers.
The pioneering work of Yoccoz \cite{Yoc95} carefully studied iterations of sector renormalizations and proved that $\mathcal{L}$ is indeed equal to the set of Brjuno numbers.

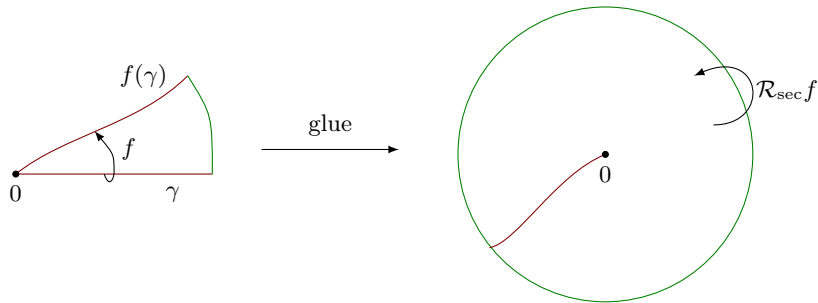
\begin{figure}
    \centering
    \begin{tikzpicture}[scale=1.3]
    \node[white] at (-4,0) {.};
    \draw[red!50!black] (-1,0) -- (-3,0) .. controls (-2.5,0.4) and (-1.75,0.5) .. (-1.25,1);
    \draw[green!50!black] (-1.25,1) .. controls (-1,0.6) .. (-1,0);
    \filldraw[color=black] (-3,0) circle (0.03);
    \node at (-3,-0.2) {\small $0$};
    \node at (-1.4,-0.2) {\small $\gamma$};
    \node at (-1.7,1) {\small $f(\gamma)$};
    \draw[-latex] (-2.1,0) .. controls (-2.05,-0.15) and (-2,-0.05) .. (-2,0.02) .. controls (-2,0.2) .. (-2.2,0.44);
    \node at (-1.85,0.25) {\small $f$};

    \draw[-latex] (-0.5,0.25) -- (0.9,0.25);
    \node at (0.2,0.5) {\small glue};
        
    \draw[green!50!black] (3,0.2) circle [radius = 1.5cm];
    \draw[red!50!black] (3,0.2) .. controls (2.5,0) and (2.1,-0.7) .. (1.82,-0.75);
    \filldraw[color=black] (3,0.2) circle (0.03);
    \node at (3,0) {\small $0$};

    \draw[-latex] (4.1,0.5) .. controls (4.7,0.5) and (4.6,1.3) .. (3.9,1);
    \node at (4.85,0.85) {\small $\Rsec f$};
\end{tikzpicture}
    \caption{A schematic picture of sector renormalization}
    \label{fig:sector-renormalization-local}
\end{figure}

Yoccoz's framework is local as the size of the fundamental sector generally depends on its rotation number.
For neutral quadratic polynomials, the construction of fundamental sectors that well-capture the critical orbit is highly desirable yet non-trivial.
In the past, this construction has been realized separately in the bounded-type regime \cite{DLS} and in the high type regime \cite{IS}.
In both cases, deep applications follow, most notably the local connectivity of the Mandelbrot set at some infinitely satellite renormalizable parameters \cite{DL23,CS15}.

The new machinery developed by Dudko and Lyubich \cite{DL22,DL26b} provides a unified framework of sector renormalization across all irrationals.
In \cite{DL22}, the authors showed that for all irrational $\theta$, the map $f_\theta$ admits a \emph{pseudo-Siegel disk} $\hat{Z}_\theta$, that is, a closed uniform quasidisk containing the neutral fixed point, the critical point, and the postcritical set, and on which $f_\theta$ is injective.
In \cite{DL26b}, they showed that these pseudo-Siegel disks are stable under sector renormalization: for all $\theta$, the renormalizations $(\Rsec)^n f_\theta$ can be constructed such that each of them admits a uniform pseudo-Siegel disk well-capturing the critical orbit.

Let us look at the set of forward orbits of sector renormalization 
\[
\towsec := \left\{ \seq{ (\Rsec)^{m+n} f_\theta }_{n \geq 0} \: : \: m \geq 0, \; \theta \in (0,1)\backslash \Q \right\}
\]
considered as a natural subspace of an infinite product space. 
\cite{DL26b} asserts that this space is pre-compact. 
For every renormalization tower $\fbold = \seq{ f_n }_{n\geq 0}$ in the closure $\overline{\towsec}$, there is a sequence of pre-renormalizations $f_0^{[n]}$ in the dynamical plane of $f_0$ that project to $f_n$ under the gluing of some appropriate sector. 
If $f_0$ is parabolic (for example, $z+z^2$), each $f_0^{[n]}$ will be a Lavaurs map of $f_0$; otherwise, $f_0^{[1]}$ is an iterate of $f_0$. 

For $\fbold = \seq{ f_n }_{n\geq 0} \in \overline{\towsec}$,
we can define the \emph{postcritical set} $P(\fbold)$ of $\fbold$ to be the closure of the forward critical orbit of the semigroup of commuting holomorphic maps generated by $f^{[n]}_0$, $n\geq 0$.
The compact hull of $P(\fbold)$ is called the \emph{Mother Hedgehog} $H(\fbold)$ of $\fbold$; it has the property that $f_0^{[n]}: H(\fbold) \to H(\fbold)$ is a self-homeomorphism for all $n \geq 0$.
The interior of $H(\fbold)$ is called the \emph{Siegel set}
\[
    Z(\fbold) := \textnormal{int}H(\fbold) = H(\fbold) \backslash P(\fbold).
\]
When the base map $f_0$ is a neutral quadratic polynomial with an irrational rotation number, $P(\fbold)$ coincides with the postcritical set of $f_0$ and $Z(\fbold)$ is the Siegel disk of $f_0$ if it exists.

Ultimately, the proof of zero area is not specific to quadratic polynomials.
Its natural setting is the space of sector renormalization towers:

\begin{thmx}[Zero area]
\label{main-thm-02}
    For any tower $\fbold$ in $\overline{\towsec}$, the postcritical set $P(\fbold)$ has zero Lebesgue measure.
\end{thmx}

Our analysis also enables us to extend Yoccoz's optimality of the Brjuno condition \cite{Yoc95} and Buff-Ch\'eritat's estimate of the size of Siegel disks \cite{BC04}. 

\begin{thmx}
    \label{side-theorem}
    For every $\fbold = \seq{ f_n  }_{n\geq 0}$ in $\overline{\towsec}$,
    \begin{enumerate}
        \item $Z(\fbold)$ is empty if and only if $\fbold$ is not eventually Brjuno;
        \item the Siegel disk of $f_0$ exists at $0$ if and only if $\fbold$ is Brjuno;
        \item if $\fbold$ is Brjuno, then the Siegel disk of $f_0$ is equal to the Siegel set $Z(\fbold)$ and its conformal radius about $0$ is $\asymp e^{-\Brjuno(\theta)}$;
        \item if $\fbold$ is eventually Brjuno but not Brjuno, then under iterated renormalization change of variables, $Z(\fbold)$ projects to the Siegel disk of $f_n$ for all sufficiently high $n \geq 0$.
    \end{enumerate}
\end{thmx}

Let us elaborate on the notation used in the theorem above.
The tower $\fbold$ is called \emph{eventually Brjuno} if there exists some $n \geq 0$ such that the rotation number of $f_n$ is a Brjuno irrational; if $n = 0$, then $\fbold$ is simply called \emph{Brjuno}.
In property (3), the constants involved in ``$\asymp$'' are independent of $\fbold$.

The next subsection provides a summary of the proof of these theorems.
In \S\ref{ss:applications}, we expand on two important future applications of this paper.

\subsection{Outline of the proof}
\label{ss:outline}

The proof of Theorem \ref{main-thm-02} (and thus \ref{main-thm-01}) and the proof of Theorem \ref{side-theorem} are intertwined. 
Ultimately, we show that the Mother Hedgehog $H(\fbold)$ with the Siegel disk or the lift of the Siegel disk of a high renormalization removed is (non-uniformly) porous at almost every point. 
There are two key ingredients.

The first one is pseudo-Siegel bounds. In Section \ref{sec:pseudo-Siegel}, we summarize the result from \cite{DL22,DL26b} on pseudo-Siegel disks and uniform bounds on sector renormalization. When $\theta$ is of bounded-type, the pseudo-Siegel disk $\hat{Z}_\theta$ of $f_\theta$ is constructed out of the Siegel disk of $f_\theta$ by filling in parabolic fjords. 
The uniform geometric control of parabolic fjords (Theorem \ref{thm:uniform-fjords}) is a particularly essential source of porosity.

The second ingredient is the distortion estimates in logarithmic coordinates, which we establish in Section \ref{sec:logarithmic}.
Pick a tower $\fbold = \seq{f_n}_{n \geq 0}$ in $\overline{\towsec}$.
Let $\theta_n \in [-\frac{1}{2},\frac{1}{2}]$ be the rotation number of $f_n$. 
There is a sector $S_n$ such that the first return map of $f_n$ back to $S_n$ projects under a gluing map $\rho_n : S_n \to \D$ to $f_{n+1}$.
We bring our dynamical plane to logarithmic coordinates by lifting under the universal covering map 
\[
\exp: \UHP \to \D^*,  \qquad \exp(z) := e^{2\pi i z}.
\]
We then pick an appropriate lift $\phi_n: \mathcal{S}_n \to \UHP$ of $\rho_n$ under $\exp$.
See Figure \ref{fig:euclidean-vs-logarithmic} for an illustration. 
The main result of Section \ref{sec:logarithmic} is the following estimate on $\phi_n$.

\begin{lemx}
    For any $n \geq 0$ and any sufficiently high $y >0$ (independent of $\theta$), whenever a point $\zeta$ in $\mathcal{S}_n$ satisfies 
    \[
    \imag(\zeta) \geq \frac{1}{2\pi} \log \frac{1}{|\theta_n|} + y,
    \]
    we have
    \[
        \Big| |\theta_n| \phi_n'(\zeta) - 1 \Big| = O(y^{-2})
        \quad \text{ and } \quad 
        \Big| |\theta_n| \phi_n(\zeta) - \zeta + \kappa_n \Big| \leq O(y^{-1}),
    \]
    where $\kappa_n = \kappa_n(f_n)$ is some complex number with 
    \[
        \kappa_n = \frac{i}{2\pi} \log\frac{1}{|\theta_n|} + O(1).
    \]
\end{lemx}
 
In the original dynamical plane of $f_n$, this lemma asserts that the renormalization change of variables $\rho_n$ is, in some sense, uniformly close to the power map $z\mapsto (|\theta_n| z)^{1/|\theta_n|}$ on a round disk about $0$ with radius comparable to $|\theta_n|$.

In Section \ref{sec:zero-area}, we apply these estimates to prove our main theorems.
An iterated application of the Key Distortion Lemma produces the Brjuno function measuring the size of the Siegel disk (part (3) of Theorem \ref{side-theorem}).
Every point $z$ in $H(\fbold)$ induces a sequence of points $z_0=z, z_1, z_2,\ldots$ where each $z_n$ lives in the dynamical plane of $f_n$ and is related to $z_{n+1}$ by repeatedly applying $f_n$ until landing on the sector $S_n$ and then applying $\rho_n$.
By virtue of the uniform expansion of the $\phi_n$'s, we show that for almost every point $z$ in $H(\fbold)$ that is not contained in the Siegel disk or a lift of the Siegel disk of a high renormalization, there exists infinitely many moments $n \in \N$ such that either $\theta_n = 0 $ or the induced point $z_n$ lies in the region where the Key Distortion Lemma is inapplicable (i.e. $|z_n| \succ |\theta_n|$).
For all such $z$ and $n$, we apply the geometric control of parabolic fjords to find a definite hole near $z_n$ in the complement of the postcritical set of the tower $\seq{f_k}_{k\geq n}$. 
By lifting such holes under the renormalization change of variables, we establish porosity of $H(\fbold)$ at such a point $z$.
This gives us the rest of Theorem \ref{side-theorem} as well as Theorem \ref{main-thm-02}.

In the high-type regime, Cheraghi's proof of the zero area theorem \cite{Che13,Che19} relies on elaborate distortion estimates in perturbed Fatou coordinates.
In this paper, we bypass the use of perturbed Fatou coordinates by using the uniform control of parabolic fjords and the Key Distortion Lemma above.
Nonetheless, Cheraghi-style estimates will come in handy for other problems; see \S\ref{sss:horseshoe}.

\subsection{Further applications}
\label{ss:applications}

Let us elaborate on two major applications of the results of this paper.
The first one is on the combinatorial rigidity of the full attractor of neutral renormalization.
The second one is on the topology of the postcritical set of neutral quadratic polynomials $f_\theta$ and more generally of towers $\fbold$.

\subsubsection{The full attractor of neutral renormalization}
\label{sss:horseshoe}

The precompactness coming from \cite{DL26b} guarantees the existence of the full renormalization attractor $\mathcal{A}$ of neutral quadratic polynomials.
Every element of $\mathcal{A}$ can be described as a bi-infinite sector renormalization tower $\underaccent{\bar}{\fbold} = \seq{ f_n }_{n\in\Z}$.
In a joint work with Dudko and Lyubich \cite{DLL}, 
we prove the combinatorial rigidity of $\mathcal{A}$: 
if two bi-infinite sector renormalization orbits $\seq{ f_n  }_{n\in\Z}$ and $\seq{ g_n  }_{n \in \Z}$ are combinatorially equivalent, then for all $n \in \Z$, $f_n$ and $g_n$ are conformally conjugate on some definite neighborhood of their Mother Hedgehogs. 

This rigidity theorem gives a complete dynamical universality of asymptotic self-similarities for neutral quadratic polynomials, extending McMullen's periodic-type result \cite{McM98}.
It establishes $\mathcal{A}$ modulo conformal equivalence as the Full Renormalization Horseshoe, bridging the two well-established renormalization regimes: the bounded type \cite{McM98,DLS} and the high type \cite{IS}.
Theorem \ref{main-thm-02} serves as one of the essential ingredients of the proof of this rigidity theorem.
Below, we will provide more details.

In \cite{DLL}, we demonstrate that every bi-infinite tower $\seq{ f_n  }_{n \in \Z}$ in $\mathcal{A}$ can be fit into a single dynamical plane as a \emph{neutral cascade} $\mathbf{F}$, that is, a commutative semigroup of global transcendental maps arising as rescaled limits of neutral quadratic polynomials.
The neutral cascade $\mathbf{F}$ captures the global geometry of $\seq{f_n}_{n\in\Z}$.
The Mother Hedgehogs associated to $f_n$ for all $n$ combine into a single Mother Hedgehog $\mathbf{H}(\mathbf{F})$ of $\mathbf{F}$, a closed unbounded $\mathbf{F}$-invariant subset of $\C$ whose boundary is the postcritical set of $\mathbf{F}$.
Theorem \ref{main-thm-02} implies that $\mathbf{H}(\mathbf{F})$ has zero area.
Moreover,

\begin{corx}
\label{cor:future-application}
    Given a neutral cascade $\mathbf{F}$ arising from a bi-infinite renormalization tower in $\mathcal{A}$, the support of any invariant line field of $\mathbf{F}$ on its Julia set must be disjoint from the grand orbit of its Mother Hedgehog.
\end{corx}

A separate hyperbolic expansion argument is supplied to rule out invariant line fields on the whole Julia set.
In \cite{DLL}, we construct a global quasiconformal conjugacy between two combinatorially equivalent cascades $\mathbf{F}$ and $\mathbf{G}$. 
The absence of line fields implies that this conjugacy must be affine.

\subsubsection{The topology of Mother Hedgehogs}
\label{sss:topology}

In a follow-up paper \cite{Lim}, we will show that the topology of the Mother Hedgehog $H(\fbold)$ for any tower $\fbold$ in $\overline{\towsec}$ must be homeomorphic to one of the five universal objects below:
\begin{enumerate}
    \item a (closed) Jordan disk if $\fbold$ is Herman;
    \item a (one-sided) hairy Jordan disk if $\fbold$ is Brjuno but not Herman;
    \item a Cantor bouquet if $\fbold$ is not eventually Brjuno;
    \item a pinched Jordan disk if $\fbold$ is eventually Herman but not Herman;
    \item a pinched hairy Jordan disk if $\fbold$ is eventually Brjuno but not Brjuno nor eventually Herman.
\end{enumerate}
The term \emph{Herman} refers to irrationals $\theta$ with the property that every real analytic circle diffeomorphism with rotation number $\theta$ is analytically linearizable.
Hairy Jordan disks and Cantor bouquets appear in the work of Aarts and Oversteegen \cite{AO93}; the latter is related to the topology of Julia sets of transcendental entire functions.
A \emph{pinched} (hairy) Jordan disk is a bouquet of countably many (hairy) Jordan disks $(D_i)_{i \in I}$ with the Hawaiian-earring-like property that every open neighborhood of the center of the bouquet contains all but finitely many $D_i$'s.

In the case of quadratic polynomials $f_\theta$ where $\theta$ is irrational, only cases (1), (2), and (3) arise.
Together, they give a positive answer to the longstanding open question on whether every quadratic Siegel disk is a Jordan disk.
This question was originally posed in the 1980s by Douady and Sullivan \cite{D83}.
In the high-type regime, 
Shishikura and Yang \cite{SY24} gave a positive answer to this question using Inou-Shishikura's near-parabolic renormalization.
Moreover, Cheraghi \cite{Che25} proved the trichotomy (1)--(3) in the high-type setting by constructing a conjugacy to a straight toy model that he constructed separately in \cite{Che23}.
In both aforementioned works, the authors applied a series of intricate distortion estimates in perturbed Fatou coordinates previously worked out by Cheraghi \cite{Che19}.
To prove (1)--(5), we will apply such distortion estimates in perturbed Fatou coordinates together with the Key Distortion Lemma from this paper to cover all possible combinatorics.

\subsection{Further discussion}

Pseudo-Siegel bounds and the results of this paper naturally generalize to higher degree unicritical polynomials $z^d + c$. 
Once some form of pseudo-Siegel bounds is developed in the multicritical setting, we expect that our results, in particular the optimality of the Brjuno condition, should hold in greater generality.

Beyond the realm of polynomial dynamics, Biswas \cite{Bi16} demonstrated the existence of neutral germs with positive area hedgehogs. 
It seems unlikely that such pathological examples can occur for polynomials.

For general quadratic polynomials, I propose the following conjecture.

\begin{conj}
\label{question-1}
    The postcritical set of any quadratic polynomial always has zero Lebesgue measure.
\end{conj}

This conjecture would imply Corollaries \ref{cor:non-recurrent} and \ref{cor:no-acip} for all quadratic polynomials.
It is known to hold for quadratic polynomials $f$ satisfying any of the following conditions:
\begin{enumerate}[label=\textnormal{(\roman*)}]
    \item $f$ is hyperbolic, for trivial reasons;
    \item $f$ has a neutral periodic point, by virtue of Theorem \ref{main-thm-01} and the straightening theorem for quadratic-like (ql) maps;
    \item all periodic points of $f$ are repelling and $f$ is at most finitely ql renormalizable, due to Lyubich and Shishikura \cite{Lyu91};
    \item $f$ is infinitely renormalizable and is robust (see \cite[Chapter 9]{McM94}), a condition that holds for bounded combinatorics (see \cite[Proposition 8.1]{McM96} and \cite{K06,DL26a}) but fails in general.
\end{enumerate}
A complete solution to this conjecture is likely within reach once some form of uniform bounds for all types of ql renormalization is developed.

Conjecture \ref{question-1} can be seen as the dynamical analog of the following conjecture, widely regarded as the probabilistic counterpart to the MLC conjecture.

\begin{conj}
\label{question-2}
    The boundary of the Mandelbrot set has zero Lebesgue measure.
\end{conj}

To my knowledge, Conjecture \ref{question-2} first appeared in Shishikura's ICM proceedings \cite{ShiICM}, in which he announced his result that the boundary of the Mandelbrot set has Hausdorff dimension two \cite{Shi98}, and sketched that the set of at most finitely renormalizable parameters on the boundary of the Mandelbrot set has zero area.
A complete proof of the latter can be found in the work of Avila and Moreira \cite{AM05}.
Lyubich's work \cite{L99} on hyperbolicity of ql renormalization with bounded combinatorics, coupled with a priori bounds \cite{K06,DL26a}, implies that the set of infinitely renormalizable parameters with bounded combinatorics also has zero area.
Hence, the remaining case is infinitely renormalizable parameters with unbounded combinatorics.
Similar to Conjecture \ref{question-1}, a complete solution to Conjecture \ref{question-2} is likely within reach once a complete renormalization theory for quadratic polynomials is developed.
The real counterpart to Conjecture \ref{question-2} has been solved by Lyubich \cite{L02}, yielding his \emph{Regular vs. Stochastic} dichotomy, a major culmination of real renormalization theory.

For higher degree multicritical polynomials, an analog of Conjecture \ref{question-2} is expected to hold (in line with Palis's conjectures \cite{Pal00}) but not Conjecture \ref{question-1}.

\subsection{Acknowledgements} 
I would like to thank Dzmitry Dudko and Mikhail Lyubich for many fruitful discussions on neutral renormalization.
I would also like to thank Davoud Cheraghi, Jeremy Kahn, Alex Kapiamba, and Jacob Mazor for insightful discussions.
The main result of this paper was first announced in Summer 2025 during the Topics in Complex Dynamics workshop at the University of Barcelona.


\section{Sector renormalization}
\label{sec:pseudo-Siegel}

This section sets up the preliminary background needed to prove our main theorems. 
It begins in \S\ref{ss:combinatorics} with a discussion on the arithmetic and combinatorial foundations of sector renormalization and the appropriate compactification $\TheCpt$ of the space of irrationals $\Irrat$ following \cite{Lim26}. 
In \S\ref{ss:pseudo-siegel}, we recall the construction of pseudo-Siegel disks \cite{DL22} and the existence of Mother Hedgehogs $H_\theta$ for neutral quadratic polynomials $f_\theta$.
In \S\ref{ss:sector-renormalization}--\ref{ss:renormalization-limits}, we recall the construction of sector renormalizations of $f_\theta$ \cite{DL26b} that are pre-compact and compatible with pseudo-Siegel disks.
In \S\ref{ss:fjord}, we emphasize on the uniform geometric control of fjords of pseudo-Siegel disks, a particular source of porosity for Mother Hedgehogs.
Finally, in \S\ref{ss:renorm-tiling}, we describe the renormalization triangulation of pseudo-Siegel disks.

\subsection{Notation}
\label{ss:notation}

For the reader's convenience, we provide a glossary of notation introduced in this section.
The keen reader may jump straight to the next subsection and come back here whenever needed.

\subsubsection{On combinatorics}
\label{sss:combinatorics}

We denote
\begin{itemize}
    \item $\Irrat = (-\frac{1}{2},\frac{1}{2}) \backslash \Q$;
    \item $\Sigma^{\N} = $ the sequence space $( \{-,+\} \times \N_{\geq 2} )^{\N}$;
    \item $\TheCpt = ( \{-,+\} \times \overline{\N}_{\geq 2} )^{\N}$, the parabolic compactification of $\Irrat$;
    \item $\gauss =$ the Gauss-like map $\Irrat \to \Irrat$ where $\gauss(\theta) \equiv -\frac{1}{\theta}$ (mod $1$);
    \item $\mathfrak{X} = $ the map $\Sigma^{\N} \to \Irrat$ described in (\ref{eqn:symbolic-irrational});
    \item $\bar{\mu} =$ the rotation number map $\TheCpt \to \T$ described in Proposition \ref{prop:rotation-number-map};
    \item $\shift = $ standard shift map on either $\Sigma^{\N}$ or $\TheCpt$.
\end{itemize}
For $\theta \in \Irrat$ and $n \geq 0$, denote
\begin{itemize}
    \item $ b_n(\theta) = \abar_n + \frac{1+\varepsilon_n \varepsilon_{n+1}}{2}$ where $\seq{(\varepsilon_n,\abar_n)}_{n\geq 1} = \mathfrak{X}^{-1}(\theta)$;
    \item $q_{[n]} = q_{[n]}(\theta)$, the sequence described in Proposition \ref{prop:q[n]}.
\end{itemize}
For $\ttheta = \seq{(\varepsilon_n,\abar_n)}_{n\geq 1} \in \TheCpt$, denote
\begin{itemize}
    \item $b_n = b_n(\ttheta) := \abar_n + \frac{1+\varepsilon_n \varepsilon_{n+1}}{2}$ for $n \geq 0$;
    \item $\Kont = \Kont_{\ttheta}$, the continuant group of $\ttheta$ (c.f. Definition \ref{def:continuant-group}) which is an ordered abelian group with generators $\qq_{[n]}=\qq_{[n]}(\ttheta)$, $n \geq 0$ with relations
    \[
        \qq_{[n]} = b_n \qq_{[n-1]} - \varepsilon_{n-1} \varepsilon_n \qq_{[n-2]} \qquad \text{ for } n \geq 2 \text{ with } \abar_n < \infty.
    \]
    \item $\Time = \Time_{\ttheta} := \{ p \in \Kont_{\ttheta} \: : \: p > 0 \}$, the time semigroup of $\ttheta$.
\end{itemize}

\subsubsection{On dynamical objects related to $\overline{\towsec}$}
\label{sss:notation-for-towsec}

The space $\overline{\towsec}$ is introduced in detail in \S\ref{ss:renormalization-limits}.
Every element $\fbold$ of $\overline{\towsec}$ is an infinite forward tower $\fbold = \seq{f_n}_{n\geq 0}$ of sector renormalizations of holomorphic maps with a neutral fixed point at $0$.
It comes with renormalization combinatorics $\ttheta = \ttheta(\fbold)$ which is an element of $\TheCpt$.
The properties of the objects listed below are described in Theorem \ref{thm:renorm-limits}.
For every $n \geq 0$, denote:
\begin{itemize}
    \item $\ttheta_n = \shift^n(\ttheta)$ and $\theta_n = \bar{\mu}(\ttheta_n)$, the rotation number of $f_n$ at $0$;
    \item the continuant group $\Kont^{n} = \Kont_{\ttheta_n}$ of $\ttheta_n$ with generators $\qq^n_{[m]} = \qq_{[m]}(\ttheta_n)$, $m \geq 0$;
    \item the time semigroup $\Time^n = \Time_{\ttheta_n}$ of $\ttheta_n$;
    \item the $n$\textsuperscript{\textnormal{th}} semigroup $\mathcal{F}_n = \{ f_n ^p \}_{p \in \Time^n}$ parametrized by $\Time^n$;
    \item the nest of pseudo-Siegel pinched disks $\hat{Z}_n^{[m]}=\hat{Z}_n^{[m]}(\fbold)$, $m \geq -1$ where each $\hat{Z}_n^{[m]}$ is almost invariant under $f_n^{[m+1]} := f_n^{\qq^n_{[m+1]}} \in \mathcal{F}_n$;
    \item the Mother Hedgehog $H_n=H_n(\fbold)$ completely invariant under $\mathcal{F}_n$;
    \item the unique bi-infinite critical orbit $\{v^n_{p} = v^n_{p}(\fbold)\}_{p \in \Kont^{n}}$ in $H_n$ where $v^n_{0}$ is the unique critical value of $f_n$;
    \item the postcritical set $P_n = P_n(\fbold) := \overline{ \big\{ v^n_{p} \big\}_{p \in \Tbold^n \cup \{0\}} }$ of $\mathcal{F}_n$, which is the boundary of $H_n$;
    \item a nest of renormalization sectors $S_n^j = S_n^j(\fbold)$, $j \geq 1$ where
    \[
        S_n := S_n^1 \supset S_n^2 \supset S_n^3 \supset \ldots;
    \]
    \item a conformal gluing map 
    \[
        \psi_n = \psi_{n,\fbold}: S_n \to \D
    \]
    for the top sector $S_n$ that projects $\hat{Z}_n^{[m-1]}$, $S_n^{m+1}$, and $f_n^{[m]}$ 
    to $\hat{Z}_{n+1}^{[m-2]}$, $S_{n+1}^{m}$, and $f_{n+1}^{[m-1]}$ respectively for every $m \geq 1$.
\end{itemize}

\subsubsection{Asymptotics}

In later sections, we will be working with either a neutral quadratic polynomial $f_\theta$ with a fixed irrational rotation number $\theta \in \Irrat$ or more generally a fixed renormalization tower $\fbold$ in $\overline{\towsec}$.
Constants that are independent of $\theta$ or $\fbold$ will be called \emph{uniform} or \emph{universal}.
We will denote 
\begin{itemize}
    \item $g = O(h)$ if there exists a universal constant $k >0$ such that $g \leq k h$;
    \item $g \asymp h$ if $g=O(h)$ and $h=O(g)$.
\end{itemize}

\subsection{The combinatorics of sector renormalization}
\label{ss:combinatorics}

In this subsection, we summarize the combinatorial and arithmetic foundations of sector renormalization as well as the appropriate compactification.
The details of the proofs are available in a separate note \cite{Lim26}.
The key notations introduced here are listed in \S\ref{sss:combinatorics}.

\subsubsection{Renormalization of irrational rotations}
\label{sss:combinatorial-sector}

Denote
\[
    \Irrat := \left( -\frac{1}{2},\frac{1}{2} \right) \backslash \Q.
\]
For $\theta \in \Irrat$, denote
\[
    \varepsilon(\theta) := \frac{\theta}{|\theta|} \in \{-1,+1\}, \quad \quad 
    \bar{a}(\theta) := \left\lfloor \frac{1}{|\theta|} \right\rfloor \in \N_{\geq 2},
\]
and let
\[
    \rotate_\theta : \bar{\D} \to \bar{\D}, \quad z \mapsto e^{2\pi i \theta} z
\]
be the rigid rotation by angle $\theta$ on the closed unit disk $\bar{\D} \subset \C$.
The sector renormalization of $\rotate_\theta$ is constructed as follows.

Let $S$ be a sector on $\bar{\D}$ bounded by two arcs $\gamma$ and $\rotate_\theta(\gamma)$ where $\gamma$ is a radial arc from $0$ to a point on the unit circle.
The power map $\psi(z) = z^{1/|\theta|}$ glues the two sides of the sector $S$ together, sending $S$ onto $\bar{\D}$.
The first return map of $\rotate_\theta$ back to $S$ is a piecewise continuous map consisting of a pair of iterates $(\rotate_\theta^{\bar{a}}, \rotate_\theta^{\bar{a}+1})$, where $\bar{a} = \bar{a}(\theta)$.
Under $\psi$, this pair projects onto a new rigid rotation $\rotate_{\gauss(\theta)}: \bar{\D} \to \bar{\D}$, called the sector renormalization of $\rotate_{\theta}$.
By elementary computation, the new rotation number $\gauss(\theta)$ is independent of $\gamma$; it is the unique irrational number in $\Theta$ with
\[
    \gauss(\theta) \equiv -\frac{1}{\theta} \quad (\textnormal{mod }1).
\]
This procedure gives us a continuous infinite-to-one surjective self-map $\gauss: \Irrat \to \Irrat$.

Consider the infinite sequence space
\[
    \The := \left\{ \seq{ (\varepsilon_n, \bar{a}_n) }_{n\geq 1} \: : \: \varepsilon_n \in \{-1,+1\}, \bar{a}_n \in \N_{\geq 2} \right\}.
\]
For every $\sigma = \seq{(\varepsilon_n, \bar{a}_n )}_{n\geq 1}$, we denote
\[
    b_n = b_n(\sigma) := \bar{a}_n + \frac{1 + \varepsilon_n \varepsilon_{n+1}}{2} \qquad \text{ for } n \geq 1
\]
and
\begin{align}
\label{eqn:symbolic-irrational}
    \mathfrak{X}( \sigma ) := \cfrac{1}{\varepsilon_1 b_1 - \cfrac{1}{\varepsilon_2 b_2 - \frac{1}{\varepsilon_3 b_3 - \ldots}}}.
\end{align}

\begin{theorem}
    \label{thm:X.theta}
    The function $\mathfrak{X}(\cdot)$ above is a well-defined homeomorphism from $\The$ onto $\Irrat$ with inverse
    \[
        \mathfrak{X}^{-1}(\theta) = \seq{ \big( \varepsilon( \gauss^{n-1}(\theta)),\;\bar{a}(\gauss^{n-1}(\theta) ) \big) }_{n\geq 1}.
    \]
    The map $\mathfrak{X}$ is a conjugacy between the standard shift map $\shift: \The \to \The$ and the map $\gauss: \Irrat \to \Irrat$.
\end{theorem}

Fix $\theta \in \Irrat$ and denote $\theta_n := \gauss^n(\theta)$.
We will now describe a particular choice of the renormalization sector $S$ for $\rotate_\theta$ that is more appropriate for iteration.

For $k \in \Z$, denote 
\[
    v_k = v_k(\theta) := \rotate_\theta^k(1).
\]
For any two distinct integers $k,l \in \Z$, we denote by $\triangle_{\theta}(k,l)$ the closed radial sector bounded by the radial line segment from $0$ to $v_k$, the radial line segment from $0$ to $v_l$, and the shortest circular arc in $\partial \D$ joining $v_k$ and $v_l$.
We will in particular consider the sector
\[
    S_\theta := \triangle_{\theta}(-\bar{a}-1,-\bar{a}),
\]
where $\bar{a}=\bar{a}(\theta)$. 
This sector contains the point $v_0 = 1$. 

The first return map of $\rotate_{\theta}$ back to $S_{\theta}$ is the piecewise linear map
\[
    \textnormal{FRM}_{\theta}(z) = \begin{cases}
        \rotate_{\theta}^{\bar{a}+1}(z) & \text{ if } z \in \triangle_{\theta}(-\bar{a}-1,-2\bar{a}-1) , \\
        \rotate_{\theta}^{\bar{a}}(z) & \text{ if } z \in \triangle_{\theta}(-2\bar{a}-1,-\bar{a}) .
    \end{cases}
\]
See Figure \ref{fig:sector-rotation}.
The power map 
\[
    \psi_{\theta}: S_{\theta} \to \bar{\D}, \quad \psi_\theta(z) = z^{1/|\theta|}
\]
sends the interior of $S_{\theta}$ conformally onto the unit disk minus the radial slit 
\[
\gamma_\theta = \{ \arg z = - 2\pi \, \gauss(\theta) \},
\]
and it glues the two radial edges of $S_{\theta}$ together onto the slit. 
It also projects $\textnormal{FRM}_{\theta}(z)$ to the renormalization $\rotate_{\gauss(\theta)}$.
Since $\gamma_{\theta}$ is disjoint from the sector $S_{\theta_1}$, we obtain a well-defined nest of closed sectors 
\[
S_{\theta}^1 \supset S_{\theta}^2 \supset S_{\theta}^3 \supset S_{\theta}^4 \supset \ldots
\]
where 
    \[
    S_{\theta}^n := \psi_{\theta_0}^{-1} \circ \psi_{\theta_1}^{-1} \circ \ldots \circ \psi_{\theta_{n-2}}^{-1}(S_{\theta_{n-1}})
    \quad \text{ for } n\geq 2.
    \]
By design, the sector $S_{\theta}^n$ always contains the real interval $[0,1] \subset \R$.
Proposition \ref{prop:q[n]} below implies that the nested intersection $\cap_n S_{\theta}^n$ is indeed equal to $[0,1]$.

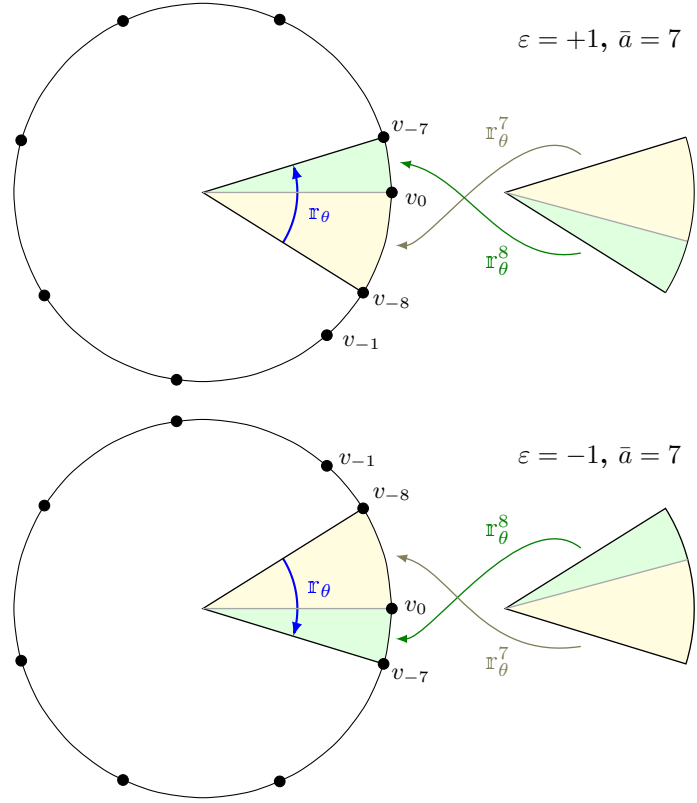
\begin{figure}    
    \centering
    \begin{tikzpicture}[scale=1]
        \filldraw[white,fill opacity=0.5,fill=green!25!white] (-3.5,0) -- (-1,0) arc (0:17:2.5cm) -- cycle;        
        \filldraw[white,fill opacity=0.5,fill=yellow!30!white] (-3.5,0) -- ({2.5*cos(-32)-3.5},{2.5*sin(-32)}) arc (-32:0:2.5cm) -- cycle;
        \draw [black,domain=0:360,smooth] plot ({2.5*cos(\x)-3.5}, {2.5*sin(\x)});
        \draw [blue,thick,domain=-32:17,-latex] plot ({1.25*cos(\x)-3.5}, {1.25*sin(\x)});
        \draw[black,line width=0.5pt] ({2.5*cos(-32)-3.5},{2.5*sin(-32)}) -- (-3.5,0) -- ({2.5*cos(17)-3.5},{2.5*sin(17)}); 
        \draw[gray!70!white,line width=0.5pt] (-3.5,0) -- (-1,0);

        \filldraw[white,fill opacity=0.5,fill=yellow!30!white] (0.5,0) -- ({2.5*cos(-15)+0.5},{2.5*sin(-15)}) arc (-15:17:2.5cm) -- cycle;   
        \filldraw[white,fill opacity=0.5,fill=green!25!white] (0.5,0) -- ({2.5*cos(-32)+0.5},{2.5*sin(-32)}) arc (-32:-15:2.5cm) -- cycle;
        \draw[black,line width=0.5pt] ({2.5*cos(17)+0.5},{2.5*sin(17)}) -- (0.5,0) -- ({2.5*cos(-32)+0.5},{2.5*sin(-32)});        
        \draw[gray!70!white,line width=0.5pt] ({2.5*cos(-15)+0.5},{2.5*sin(-15)}) -- (0.5,0);
        \draw [black,domain=-32:17] plot ({2.5*cos(\x)+0.5}, {2.5*sin(\x)});

        \filldraw (-1,0) circle (2pt);
        \filldraw ({2.5*cos(-49)-3.5},{2.5*sin(-49)}) circle (2pt);
        \filldraw ({2.5*cos(-98)-3.5},{2.5*sin(-98)}) circle (2pt);
        \filldraw ({2.5*cos(-147)-3.5},{2.5*sin(-147)}) circle (2pt);
        \filldraw ({2.5*cos(-196)-3.5},{2.5*sin(-196)}) circle (2pt);
        \filldraw ({2.5*cos(-245)-3.5},{2.5*sin(-245)}) circle (2pt);
        \filldraw ({2.5*cos(-294)-3.5},{2.5*sin(-294)}) circle (2pt);
        \filldraw ({2.5*cos(-343)-3.5},{2.5*sin(-343)}) circle (2pt);
        \filldraw ({2.5*cos(-392)-3.5},{2.5*sin(-392)}) circle (2pt);

        \node [black, font=\bfseries] at (1.75,2) {$\varepsilon =+1$, $\bar{a}= 7$};
        \node [blue, font=\bfseries] at (-1.95,-0.3) {\small $\rotate_{\theta}$};
        \draw[yellow!40!black,line width=0.5pt,-latex] (1.5,0.5).. controls (0.75,1.1) and (-0.5,-0.7) .. (-0.95,-0.7) ;
        \draw[green!50!black,line width=0.5pt,-latex] (1.5,-0.8) .. controls (0.5,-1) and (0,0.2) .. (-0.9,0.4);
        \node [green!50!black, font=\bfseries] at (0.4,-0.9) {$\rotate^8_{\theta}$};
        \node [yellow!40!black, font=\bfseries] at (0.4,0.8) {$\rotate^7_{\theta}$};
        \node [black, font=\bfseries] at (-0.67,-0.1) {\small $v_{0}$};
        \node [black, font=\bfseries] at (-1.4,-2) {\small $v_{-1}$};
        \node [black, font=\bfseries] at (-0.75,0.85) {\small $v_{-7}$};
        \node [black, font=\bfseries] at (-1,-1.45) {\small $v_{-8}$};

        \filldraw[white,fill opacity=0.5,fill=green!25!white] (-3.5,-5.5) -- (-1,-5.5) arc (0:-17:2.5cm) -- cycle;        
        \filldraw[white,fill opacity=0.5,fill=yellow!30!white] (-3.5,-5.5) -- (-1,-5.5) arc (0:32:2.5cm) -- cycle;
        \draw [black,domain=0:360,smooth] plot ({2.5*cos(\x)-3.5}, {2.5*sin(\x)-5.5});        
        \draw [blue,thick,domain=32:-17,-latex] plot ({1.25*cos(\x)-3.5}, {1.25*sin(\x)-5.5});
        \draw[black,line width=0.5pt] ({2.5*cos(32)-3.5},{2.5*sin(32)-5.5}) -- (-3.5,-5.5) -- ({2.5*cos(-17)-3.5},{2.5*sin(-17)-5.5});        
        \draw[gray!70!white,line width=0.5pt] ({2.5*cos(0)-3.5},{2.5*sin(0)-5.5}) -- (-3.5,-5.5);

        \filldraw[white,fill opacity=0.5,fill=yellow!30!white] (0.5,-5.5) -- ({2.5*cos(-17)+0.5},{2.5*sin(-17)-5.5}) arc (-17:15:2.5cm) -- cycle;        
        \filldraw[white,fill opacity=0.5,fill=green!25!white] (0.5,-5.5) -- ({2.5*cos(15)+0.5},{2.5*sin(15)-5.5}) arc (15:32:2.5cm) -- cycle;
        \draw[black,line width=0.5pt] (0.5,-5.5) -- ({2.5*cos(-17)+0.5},{2.5*sin(-17)-5.5}) arc (-17:32:2.5cm) -- cycle;        
        \draw[gray!70!white,line width=0.5pt] ({2.5*cos(15)+0.5},{2.5*sin(15)-5.5}) -- (0.5,-5.5);
        
        \filldraw (-1,-5.5) circle (2pt);
        \filldraw ({2.5*cos(49)-3.5},{2.5*sin(49)-5.5}) circle (2pt);
        \filldraw ({2.5*cos(98)-3.5},{2.5*sin(98)-5.5}) circle (2pt);
        \filldraw ({2.5*cos(147)-3.5},{2.5*sin(147)-5.5}) circle (2pt);
        \filldraw ({2.5*cos(196)-3.5},{2.5*sin(196)-5.5}) circle (2pt);
        \filldraw ({2.5*cos(245)-3.5},{2.5*sin(245)-5.5}) circle (2pt);
        \filldraw ({2.5*cos(294)-3.5},{2.5*sin(294)-5.5}) circle (2pt);
        \filldraw ({2.5*cos(343)-3.5},{2.5*sin(343)-5.5}) circle (2pt);
        \filldraw ({2.5*cos(392)-3.5},{2.5*sin(392)-5.5}) circle (2pt);

        \node [black, font=\bfseries] at (1.75,-3.5) {$\varepsilon=-1$, $\bar{a} = 7$};
        \node [blue, font=\bfseries] at (-1.95,-5.25) {$\rotate_{\theta}$};
        \draw[green!50!black,line width=0.5pt,-latex] (1.5,-4.7).. controls (0.75,-4.1) and (-0.5,-5.9) .. (-0.95,-5.9) ;
        \draw[yellow!40!black,line width=0.5pt,-latex] (1.5,-6) .. controls (0.2,-6.2) and (-0.2,-5) .. (-0.95,-4.8);
        \node [green!50!black, font=\bfseries] at (0.4,-4.45) {$\rotate^8_{\theta}$};
        \node [yellow!40!black, font=\bfseries] at (0.4,-6.2) {$\rotate^7_{\theta}$};
        \node [black, font=\bfseries] at (-0.68,-5.48) {\small $v_{0}$};
        \node [black, font=\bfseries] at (-1.45,-3.55) {\small $v_{-1}$};
        \node [black, font=\bfseries] at (-0.75,-6.4) {\small $v_{-7}$};
        \node [black, font=\bfseries] at (-1,-4) {\small $v_{-8}$};
\end{tikzpicture}
    \caption{Two examples of the first return map on the shaded sector $S_{\theta}$.}
    \label{fig:sector-rotation}
\end{figure}

\begin{proposition}[First return times]
\label{prop:q[n]}
    Let $\sigma = \seq{ (\varepsilon_n, \bar{a}_n)}_{n\geq 1} \in \The$ and $\theta = \mathfrak{X}(\sigma)$.
    Recall the notation $b_n$, $\theta_n$, and $v_n$ associated to $\sigma$ and $\theta$ introduced above.
    Consider the increasing sequence of integers $\{q_{[n]}=q_{[n]}(\sigma)\}_{n\geq 0}$ where
\[
    q_{[0]} = 1, \qquad q_{[1]} = b_1, 
\]
and recursively for $n\geq 2$,
\[
    q_{[n]} = b_n q_{[n-1]} - \varepsilon_{n-1} \varepsilon_n  q_{[n-2]}.
\]
    Then, for $n \geq 1$,
\begin{enumerate}
    \item $\rotate_{\theta}^{q_{[n]}}$ is the rigid rotation by angle $|\theta_0 \ldots\theta_{n-1}|\theta_n$;
    \item $q_{[n]}$ is the unique smallest positive integer that satisfies 
    \[
    \left| [v_0, v_{q_{[n]}]}] \right| \leq \frac{1}{2} \left| [v_0, v_{q_{[n-1]}}] \right|,
    \]
    where $|\cdot|$ denotes the normalized angular measure on $\partial \D$;
    \item if we denote 
    \[
        \check{q}_{[n]} := q_{[n]} - \varepsilon_n \varepsilon_{n+1} q_{[n-1]},
    \]
    then the sector $S_{\theta}^n$ is equal to 
    $ \triangle_{\theta} (-q_{[n]}, -\check{q}_{[n]})$, and the first return map of $\rotate_\theta$ back to $S_{\theta}^n$ is the pair of iterates
    \[
        \rotate_\theta^{q_{[n]}}: 
        \triangle_{\theta} (-q_{[n]}, -q_{[n]} - \check{q}_{[n]})
        \to  
        \triangle_{\theta}(0,-\check{q}_{[n]})
    \]
    and
    \[ 
        \rotate_\theta^{\check{q}_{[n]}}: \triangle_{\theta}(-q_{[n]} - \check{q}_{[n]}, -\check{q}_{[n]}) 
        \to 
        \triangle_{\theta}(-q_{[n]}, 0). 
    \]
\end{enumerate}
\end{proposition}

Essentially, item (3) of the proposition above says that $\rotate_\theta^{q_{[n]}}|_{S_{\theta}^n}$ is the $n$\textsuperscript{th} pre-renormalization of $\rotate_\theta$:
\[
    \rotate_{\theta_n} = \rotate_\theta^{q_{[n]}}|_{S_{\theta}^n} \Big/ \rotate_\theta^{q_{[n-1]}}.
\]

\begin{example}
    The golden mean irrationals in $\Irrat$ are $\theta_{\textnormal{gm}} = \frac{3-\sqrt{5}}{2}$ and $-\theta_{\textnormal{gm}} = \frac{\sqrt{5}-3}{2}$. We have
    \begin{align*}
        \mathfrak{X}^{-1}(\theta_{\textnormal{gm}}) &= \seq{ (+,2), (+,2), (+,2), (+,2),\ldots },\\
        \mathfrak{X}^{-1}(-\theta_{\textnormal{gm}}) &= \seq{ (-,2), (-,2), (-,2), (-,2), \ldots }.
    \end{align*}
    The corresponding modified continued fraction expansions are
    \[
        \theta_{\textnormal{gm}} = \cfrac{1}{3 - \frac{1}{3 - \frac{1}{3 - \ldots}}} \quad \text{ and } \quad 
        -\theta_{\textnormal{gm}} = \cfrac{1}{-3 - \frac{1}{-3 - \frac{1}{-3 - \ldots}}}.
    \]
    Both $\theta_{\textnormal{gm}}$ and $-\theta_{\textnormal{gm}}$ have the same first return times, namely
    \[
        q_{[0]} = 1, \quad q_{[1]} = 3, \quad q_{[2]}=8, \quad q_{[3]} = 21, \quad q_{[4]} = 55, \quad q_{[5]} = 144, \quad \ldots.
    \]
\end{example}

\subsubsection{Parabolic compactification of the irrationals}
\label{sss:compactification}

Let $\overline{\N} = \N \cup \{\infty\}$, where the addition of the point $\displaystyle \infty = \lim_{n \to \infty} n$ makes it a compactification of the discrete space $\N$.
Let 
\[
\overline{\Sigma} := \{-1,+1\} \times \overline{\N}_{\geq 2} \qquad \text{ and } \qquad
    \TheCpt := \overline{\Sigma}^{\N}.
\]
The infinite sequence space $\TheCpt$ is a Cantor set. The inclusion map $\iota: \The \xhookrightarrow[]{} \TheCpt$ has a dense image.
Through $\iota \circ \mathfrak{X}^{-1}$, the space $\TheCpt$ together with the shift map $\shift$ is a compactification of the dynamical system $(\Irrat,\gauss)$.

\begin{theorem}[Universal property of $\TheCpt$]
    The embedding 
    \[
    \bar{\iota} :=\iota \circ \mathfrak{X}^{-1}: \Irrat \to \TheCpt
    \]
    is the smallest compactification of $\Irrat$ satisfying the following properties.
    The dynamical system $(\Theta,\gauss)$ extends to a continuous self-map on $\TheCpt$, which is $(\TheCpt, \shift)$, and various natural embeddings of $\Irrat$ into some compact spaces extend continuously.
\end{theorem}

Refer to \cite[Theorem B]{Lim26} for the details of these natural embeddings.
Roughly, this formal theorem states that $\TheCpt$ is the most appropriate compactification of the irrationals that takes into account parabolic implosion.
For this reason, $\TheCpt$ is called the \emph{parabolic compactification} of $\Irrat$.

\begin{definition}
    An element $\ttheta = \seq{ (\varepsilon_n, \bar{a}_n) }_{n\geq 1}$ of $\TheCpt$ is said to be \emph{irrational} if $\bar{a}_n < \infty$ for all $n \geq 1$, \emph{enriched rational} otherwise.
\end{definition}

Here is another useful property of $\TheCpt$.

\begin{proposition}
\label{prop:rotation-number-map}
    Under $\bar{\iota}$, the embedding $\mu: \Irrat \to \T$, $ \theta \mapsto \theta \textnormal{ (mod }1)$ into the circle $\T=\R/\Z$ extends to a surjective continuous map
\[
    \bar{\mu} : \TheCpt \to \T.
\]
    For every rational number $\theta$ in $\mathbb{T}$, $\bar{\mu}^{-1}(\theta)$ is homeomorphic to $\TheCpt$ itself.
\end{proposition}

\begin{definition}
\label{def:continuant-group}
For every $\ttheta = \seq{ (\varepsilon_n, \bar{a}_n) }_{n \geq 1} \in \TheCpt$, denote
\[
    b_n = b_n(\ttheta) = 
    \begin{cases}
        \bar{a}_n + \cfrac{1+ \varepsilon_n \varepsilon_{n+1}}{2} & \text{ if } \bar{a}_n < \infty, \\
        \infty & \text{ if } \bar{a}_n = \infty.
    \end{cases}
\]
Define an additive abelian group $\Kont_{\ttheta}$, called the \emph{continuant group of $\ttheta$}, to be the group generated by the infinite sequence $\{\qq_{[n]} = \qq_{[n]}(\ttheta)\}_{n \geq 0}$ subject to the relations
\[
    \qq_{[n]} = b_n \qq_{[n-1]} - \varepsilon_{n-1} \varepsilon_n  \qq_{[n-2]} \qquad \text{ for all } n \geq 1 \text{ with } b_n < \infty.
\]
In the above, we set $\qq_{[-1]} = 0$ as the identity element.
Each $\qq_{[n]}$ will be referred to as the \emph{$n$\textsuperscript{th} return time} of $\ttheta$.
\end{definition}

In the special case where $\ttheta$ is irrational, $\qq_{[0]} \mapsto 1$ induces an isomorphism between $\Kont_{\ttheta}$ and $\Z$ that sends each $\qq_{[n]}$ to the number $q_{[n]}$ coming from Proposition \ref{prop:q[n]}.

\begin{proposition}[Time order]
\label{prop:chronological-order-01}
    For every $\ttheta = \seq{ (\varepsilon_n, \bar{a}_n) }_{n \geq 1} \in \TheCpt$, the continuant group $\Kont_{\ttheta}$ admits a unique translation-invariant total order $<$, called the chronological order of $\Time_{\ttheta}$, with the property that for all $n \geq 0$,
    \begin{enumerate}
        \item $0 < \qq_{[n]} < \qq_{[n+1]}$,
        \item if $\bar{a}_{n+1} = \infty$, then $k \qq_{[n]} < \qq_{[n+1]}$ for all $k \geq 1$.
    \end{enumerate}
\end{proposition}

\begin{definition}
    We define the \emph{time semigroup} of $\ttheta$ to be the commutative semigroup equal to the positive cone of $<$, i.e.
    \[
    \Time_{\ttheta} := \{P \in \Kont_{\ttheta} \: : \: P > 0 \},
    \]
\end{definition}

Unlike $\Kont_{\ttheta}$, the set of generators of $\Time_{\ttheta}$ is more complicated, namely
    \[
        \{\qq_{[0]}\} \cup \{ \qq_{[n+1]} - \qq_{[n]} \}_{n \geq 0} \cup \bigcup_{\bar{a}_{n+1} = \infty} \{ \qq_{[n+1]} - k \qq_{[n]} \}_{k \geq 2}.
    \]
In \S\ref{ss:renormalization-limits}, this time semigroup will be used to parametrize the semigroup of pre-renormalizations of neutral quadratic polynomials.

\subsection{Mother Hedgehogs and pseudo-Siegel disks} 
\label{ss:pseudo-siegel}

It is natural to ask whether or not sector renormalization can be applied to neutral quadratic polynomials 
\[
f_\theta(z) = e^{2\pi i \theta} z + z^2, \qquad \theta \in \Irrat
\]
in way that retains the critical orbit well.
Before we discuss renormalization of $f_\theta$, we will first summarize the theory of Mother Hedgehogs and pseudo-Siegel disks from \cite{DL22,DL26b}.
We will denote by 
\[
v_{\theta,-1} := -\frac{e^{2\pi i \theta}}{2}
\qquad \text{and} \qquad
v_{\theta,0} := f_{\theta}(v_{\theta,-1}) = -\frac{e^{4\pi i \theta}}{4}
\]
the critical point and the critical value of $f_\theta$ respectively.

\begin{definition}
    For $\theta \in \Irrat$, a subset $H$ of $\C$ is called a \emph{Mother Hedgehog} of $f_\theta$ if it is a full compact connected subset of $\C$ containing $0$ and the critical point such that $f_\theta: H \to H$ is a homeomorphism.
\end{definition}

Denote
\[
    \Irrat_{\text{bdd}} := \left\{ \theta \in \Theta \: : \: \sup_n \bar{a}(\gauss^n(\theta)) < \infty \right\}.
\]
Elements of $\Irrat_{\text{bdd}}$ are called bounded-type irrationals. 
One special subset is the set $\Irrat_{\text{EGM}}$ of eventually golden mean (EGM) irrationals, that is, the set of irrationals $\theta$ such that  $|\gauss^k(\theta)| = \theta_{\text{gm}}$ for some $k \geq 0$.

For $\theta \in \Irrat_{\text{bdd}}$, Douady-Ghys surgery \cite{D87,G84} implies that $f_\theta$ admits a Siegel disk $Z_\theta$ with quasiconformal boundary passing through the critical point. In this case, $f_\theta$ admits a unique Mother Hedgehog which is simply the closure of $Z_\theta$.

\begin{theorem}[\cite{DL22,DL26b}]
\label{thm:mother-hedgehog}
    For every $\theta \in \Irrat$, $f_\theta$ admits a unique Mother Hedgehog $H_\theta$. It has the following properties.
    \begin{enumerate}
        \item The boundary of $H_\theta$ is the postcritical set of $f_\theta$. 
        \item The critical value $v_0(f_\theta)$ is an accessible point of $H_\theta$ from infinity and $H_\theta \backslash \{v_{\theta,0}\}$ is connected.
        \item $H_\theta$ is star-like, i.e. a bouquet of quasiarcs called internal rays centered at the fixed point $0$. It can be written as $H_\theta = \bigcup_{\phi \in Q_{\theta}} I_{\phi}$ where $Q_\theta$ is a subset of $\R/ \Z$ and $f_{\theta}$ sends each internal ray $I_{\phi}$ onto $I_{\phi+\theta}$.
        \item The internal ray $I$ of $H_\theta$ landing at $v_0(f_\theta)$ is a uniform quasiarc with uniformly bounded spiraling number about $0$, that is, there is a continuous branch of the argument on $I$ such that
    \[
        \max_{z_1, z_2 \in I} |\arg(z_1)-\arg(z_2)| = O(1).
    \]
        \item $H_{\theta}$ depends uniformly continuously on $\theta \in \Irrat$ (as a subspace of $\overline{\Theta}$) in the Hausdorff topology and the normalized Riemann mapping
        \[
            X_{\theta} : \C \backslash \overline{\D} \to \C \backslash H_{\theta}, \quad 
            \quad X(1) = v_{\theta,0}
        \]
        also depends continuously on $\theta$ in the compact-open topology.
    \end{enumerate}
\end{theorem}

Observe that the Mother Hedgehog $H_\theta$ admits a unique bi-infinite critical orbit which we will denote from now on by
\[
    \ldots \; \xrightarrow[\quad]{f_\theta} \; v_{\theta,-2} \; \xrightarrow[\quad]{f_\theta} \; v_{\theta,-1} \; \xrightarrow[\quad]{f_\theta} \;
    v_{\theta,0} \; \xrightarrow[\quad]{f_\theta} \; v_{\theta,1} \; \xrightarrow[\quad]{f_\theta} \; v_{\theta,2} \; \xrightarrow[\quad]{f_\theta} \; \ldots.
\]

The proof of Theorem \ref{thm:mother-hedgehog} relies on pseudo-Siegel bounds. 
For $\theta \in \Irrat_{\text{bdd}}$, the quasiconformal dilatation of the Siegel disk $Z_{\theta}$ generally worsens as $\inf_n |\gauss^n(\theta)|$ gets closer to $0$. 
To rectify this issue, one can enlarge the Siegel disk to a \emph{pseudo-Siegel disk}, which is formally defined as follows.

\begin{definition}
    For $\theta \in \Theta$, a \emph{pseudo-Siegel disk} of $f_\theta$ is a closed quasidisk containing the postcritical set of $f_\theta$ on which $f_\theta$ is injective.
\end{definition}

Now, we will fix $\theta \in \Theta$ and outline the construction of a particular nested sequence of pseudo-Siegel disks
\[
    \ldots \subset \hat{Z}^{[2]}_\theta \subset \hat{Z}^{[1]}_\theta \subset \hat{Z}^{[0]}_\theta \subset \hat{Z}^{[-1]}_\theta = \hat{Z}_\theta
\]
associated to $f_\theta$.

Consider the sequence $\mathfrak{X}^{-1}(\theta) = \seq{ ( \varepsilon_n, \bar{a}_n ) }_{n \geq 1}$ as well as the corresponding sequence of first return times $q_{[n]} = q_{[n]}(\theta)$, $n \geq 0$ described in the previous subsection. 
We will again use the notation
    \[
        b_n := \bar{a}_n + \frac{1+ \varepsilon_n \varepsilon_{n+1}}{2} \qquad
        \text{ and } \qquad
        \check{q}_{[n]} := q_{[n]} - \varepsilon_n \varepsilon_{n+1} q_{[n-1]}.
    \]
Also, denote $f_\theta^{[n]} \equiv f_\theta^{q_{[n]}}$.

Let us fix a sufficiently large constant $\threshold \in \N$.
This will be called the \emph{combinatorial threshold} for the pseudo-Siegel disks and it is independent of $\theta$.
For every integer $n \geq 0$, there will be two cases:
\begin{itemize}
    \item $n$ is bounded: $\bar{a}_{n+1}< \threshold$,
    \item $n$ is near-parabolic: $\bar{a}_{n+1} \geq \threshold$.
\end{itemize}
We will also consider the number
\[
\threshold' := \lfloor e^{\sqrt{\log \threshold}}\rfloor,
\]
which is much smaller than $\threshold$. 
The construction of the pseudo-Siegel disks $\hat{Z}_{\theta}^{[m]}$ will technically depend on $\threshold$.

For every $n \geq 0$, set
\begin{equation}
\label{eqn:xn-yn}
    x_n = v_{\theta,-q_{[n]}}, \qquad 
    y_n = v_{\theta,-\check{q}_{[n]}}.
\end{equation}
For $n = 0$, both $x_0$ and $y_0$ are equal to the critical point $v_{\theta,-1}$ of $f_\theta$. 
For $n \geq 1$, the critical value $v_{\theta,0}$ is between $x_n$ and $y_n$ along the Carath\'eodory boundary $\partial^c H_\theta$ of $H_\theta$ and that $x_n$ and $y_n$ are identified under the iterate $f_{\theta}^{[n-1]}$.
When $\bar{a}_{n+1} \geq \threshold$, we are also interested in the pre-critical points 
\[
x'_n = v_{\theta,-q_{[n+1]} + \threshold' q_{[n]}},\qquad
y'_n = v_{\theta,-\check{q}_{[n]} - \threshold' q_{[n]}}.
\]
on $\partial H_\theta$. 
It has the property that 
\[
(f_\theta^{[n]})^{b_{n+1}-2\threshold'-1}(x'_n) = y'_n \quad \text{ and } \quad (f_\theta^{[n]})^{\threshold'}(y'_n) = y_n,
\]
and along $\partial^c H_\theta$, these points are ordered as follows: 
\[
x_n < v_{\theta,0} < x'_n < y'_n < y_n.
\]
The sub-``interval'' $\left[x'_n, y'_n\right]$ of $\partial^c H_\theta$ is almost invariant under $f_\theta^{[n]}$.

Whenever $n$ is a near-parabolic level, we define the \emph{principal level $n$ dam} $d^{[n]}_{\theta,0}$ to be the unique hyperbolic geodesic of $\RS \backslash H_{\theta}$ with endpoints $x'_n$ and $y'_n$ that does not separate $v_{\theta,0}$ from infinity.
In general, for $j \in \{0, 1,2,\ldots,q_{[n]}\}$, we define the $j$\textnormal{th} \emph{level $n$ dam} $d_{\theta,j}^{[n]}$ to be the unique connected component of $f_\theta^{-j}(d^{[n]}_{\theta,0})$ that has endpoints on $\partial H_{\theta}$.
Observe that $d^{[n]}_{\theta,q_{[n]}}$ is combinatorially very close to $d^{[n]}_{\theta,0}$.

For every near-parabolic level $n$ and every $j \in \{0,1,\ldots,q_{[n]}-1\}$, the closure in $\C \backslash H_{\theta}$ of the disk enclosed by $d^{[n]}_{\theta,j}$ is called a \emph{parabolic fjord} of level $n$, which we will denote by $V^{[n]}_{\theta,j}$.
For every $m \geq 0$, we then define
    \[
        \hat{Z}_\theta^{[m-1]} := H_\theta \cup \bigcup_{\substack{ n\geq m, \\ n \textnormal{ is near-parabolic}}} \bigcup_{j=0}^{q_{[n]}-1} V^{[n]}_{\theta,j}.
    \]
From the construction, we have that for all $m \geq 0$,
\[
    \hat{Z}_\theta^{[m-1]} \supseteq \hat{Z}_\theta^{[m]}
\]
and equality holds if and only if $n$ is a bounded level.

\begin{figure}
    \centering
\begin{tikzpicture}
    \node[anchor=south west,inner sep=0] (image) at (0,0) {\includegraphics[width=1\linewidth]{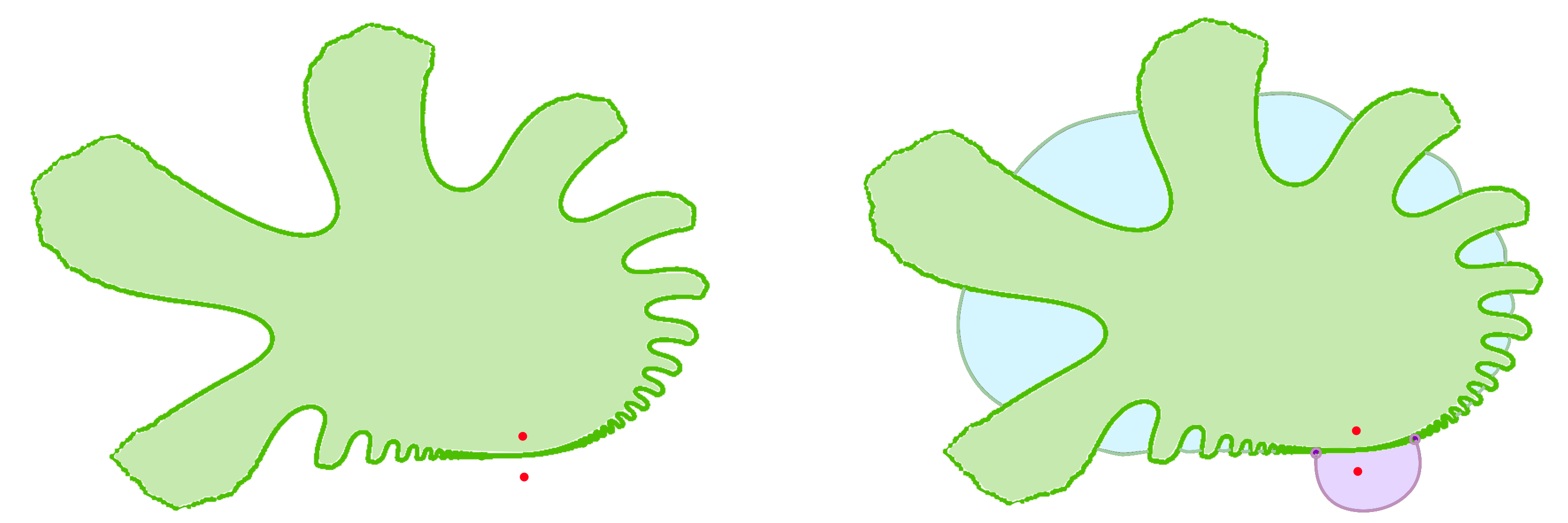}};
    \begin{scope}[
        x={(image.south east)},
        y={(image.north west)}
    ]
    
    \node [green!20!black, font=\bfseries] at (0.3,0.4) {$Z_\theta$};
    \node [green!20!black, font=\bfseries] at (0.83,0.4) {$\hat{Z}_\theta$};
    \end{scope}
\end{tikzpicture}
    \caption{The Siegel disk and a pseudo-Siegel disk of $f_\theta$ when $\mathfrak{X}^{-1}(\theta) = \seq{ (+,100),(-,100),(+,2),(+,2), \ldots }$. 
    The level $0$ dam bounds the parabolic fjord in purple and level $1$ dams bound the fjords in blue. 
    The alpha and beta fixed points are marked in red.}
    \label{fig:PS-disk}
\end{figure}

\begin{theorem}[Pseudo-Siegel bounds \cite{DL22, DL26b}]
\label{thm:pseudo-siegel}
    For sufficiently high $\threshold \in \N$ and for all $\theta \in \Theta$ and $m \geq -1$, the set $\hat{Z}_{\theta}^{[m]}$, $m \geq -1$ constructed above is a pseudo-Siegel disk of $f_\theta$. 
    They also satisfy the following properties.
    \begin{enumerate}
        \item Almost invariance: There exists a universal constant $C>0$ (independent of $\threshold$ and $\theta$) such that for all $m \geq 0$ and $j \in \{1,\ldots, q_{[m]}(\theta)\}$, the hyperbolic distance between every point on $d_j^{[m]}$ and the unique geodesic with the same endpoints as $d_j^{[m]}$ is bounded above by $C$.
        \item Uniform quasiconformality: There exists a constant $ K(\threshold)>0$ (independent of $\theta$) such that $K(\threshold) \to \infty$ as $\threshold \to \infty$ and $\hat{Z}_{\theta} = \hat{Z}_{\theta}^{[-1]}$ is $K(\threshold)$-quasiconformal.
        \item Uniform continuity: For every $m \geq -1$, the normalized Riemann mapping 
        \[
            X_{m,\theta}: \C \backslash \overline{\D} \to \RS \backslash \hat{Z}_{\theta}^{[m]}, \quad X_{m,\theta}(1) = v_{\theta,0}
        \]
        depends uniformly continuously on $\theta \in \Theta$ (as a subspace of $\TheCpt$) in the $C^0$ topology. 
    \end{enumerate}
\end{theorem}

The proof of pseudo-Siegel bounds is achieved in the near-degenerate regime. In \cite{DL22}, geodesic pseudo-Siegel disks are constructed (a slight variation of $\hat{Z}_{\theta}$ where all dams $d^{[m]}_{\theta,j}$ are instead hyperbolic geodesics) for EGM irrationals and are shown to be almost invariant and uniformly quasiconformal. 
As a corollary, the existence of the Mother Hedgehog $H_{\theta}$ across all $\theta$ was shown as it is equal to the non-escaping set of $f_{\theta}: \hat{Z}_{\theta}^{[-1]} \to \C$. 
The uniform continuity of pseudo-Siegel disks across all bounded-type rotation numbers was proven in the follow up work \cite{DL26b} and this results in Theorem \ref{thm:mother-hedgehog}. 
In the next subsection, we will describe in detail the uniform bounds on sector renormalization obtained in \cite{DL26b}.

Throughout this paper, the combinatorial threshold $\threshold$ will be assumed to be a fixed sufficiently high number and the dependence on $\threshold$ will be suppressed unless otherwise stated.

\subsection{Sector renormalization}
\label{ss:sector-renormalization}

\begin{definition}
\label{def:sector}
    We define a \emph{sector} $S$ to be an open Jordan domain in $\RS$ such that its boundary $\partial S$ is a concatenation of three marked arcs $\alpha, \alpha', \omega$, where $\alpha$ and $\alpha'$ are called the \emph{sides} of $S$ and $\omega$ is called the \emph{summit} of $S$. The \emph{vertex} of $S$ is the common endpoint of $\alpha$ and $\alpha'$.
    A \emph{conformal gluing map} for a sector $S$ is a univalent map $\rho: S \to \D$ with the following properties.
    \begin{itemize}
        \item There is a ray $\gamma$ in $\D$, called the \emph{slit} of $\rho$, that emanates from $0$ to a point on the boundary of $\D$ such that the image of $\rho$ is $\D \backslash \gamma$;
        \item The map $\rho$ extends continuously to a map from $\partial S$ onto $\partial \D \cup \gamma$. It sends the vertex of $S$ to $0$, maps each of the two sides of $S$ homeomorphically onto $\overline{\gamma}$, and maps the summit of $S$ onto $\partial \D$.
    \end{itemize}
\end{definition} 

Now, we are ready to describe sector renormalizations of $f_\theta$, $\theta \in \Irrat$.
The theorem below states that they can be constructed to still admit uniform pseudo-Siegel disks and Mother Hedgehogs.

\begin{theorem}[{Sectorial bounds \cite{DL26b}}]
\label{thm:sectorial-bounds}
    Pick any $\theta \in \Irrat$ and denote $f=f_{\theta}$, $H=H_{\theta}$, $\hat{Z}^{[m]} = \hat{Z}^{[m]}_\theta$ for all $m \geq -1$, and $\theta_n = \gauss^n(\theta)$ for all $n \geq 0$.
    There exists a nested sequence of sectors 
    \[
    S_1 \supset S_2 \supset S_3 \supset \ldots
    \]
    in the dynamical plane of $f$ such that the following properties hold for every $n \geq 1$.
\begin{enumerate}
    \item For any tip $z$ of $H$, denote by $I_z$ the internal ray in $H$ with endpoints $0$ and $z$.
    Consider the pre-critical points $x_n, y_n$ coming from (\ref{eqn:xn-yn})
    and let $z_n \in H$ be such that $\{z_n, f^{[n-1]}(z_n)\} = \{x_n,y_n\}$.
    The first side of $S_n$ is the concatenation of $I_{z_n}$ with an arc $\Gamma_n$ starting at $z_n$ which has diameter 
    \[
        \diam(\Gamma_n) \asymp |x_n - y_n|.
    \] 
    The second side of $S_n$ is the image of the first side under $f^{[n-1]}$, and $f^{[n-1]}(\Gamma_n)$ is contained in $(f^{[n]})^{-1}(H) \backslash H$.
    The summit of $S_n$ connects the endpoints of $\Gamma_n$ and $f^{[n-1]}(\Gamma_n)$.
    \item Gluing $I_{z_n} \cup \Gamma_n$ with its image under $f^{[n-1]}$ gives us a conformal gluing map $\psi_n: (S_n,0) \to (\D,0)$ of $S_n$. 
    Under $\psi_n$, the action of $f^{[n]}$ on $S_n$ (modulo $f^{[n-1]}$) projects to a holomorphic map 
    \[
        \Rsec^n f := f_n: \Dom(f_n) \to \D,
    \]
    on some open domain $\Dom(f_n) \subset \D$ which has a neutral fixed point at $0$ with rotation number $\theta_n$.
    \item For $m \geq -1$, the map $\psi_n$ projects $\hat{Z}^{[n+m]} \cap \overline{S_n}$ to a pseudo-Siegel disk $\hat{Z}^{[m]}_n$ of $f_n$ which satisfies the following properties:
    \begin{enumerate}[label=\textnormal{(\alph*)}]
        \item $\hat{Z}^{[m]}_n$ is almost invariant under the iterate $f^{[m+1]}_n := f_n^{q_{[m+1]}(\theta_n)}$;
        \item $\hat{Z}^{[-1]}_n$ and $\hat{Z}^{[0]}_n$ are $K$-quasidisks where $K\geq 1$ depends only on the combinatorial threshold $\threshold$;
        \item there exist universal constants $R$ and $R'$ (independent of $\theta$, $n$, $m$) such that $0<R<R'<1$ and
            \[
            \D(0,R) \subset \hat{Z}^{[-1]}_n \subset \D(0,R').
            \]
    \end{enumerate}
    \item The map $\psi_n$ also projects $H \cap \overline{S_n}$ to a set $H_n$, called the Mother Hedgehog of $f_n$, which satisfies \textnormal{(1)--(4)} in Theorem \ref{thm:mother-hedgehog}. 
    The intersection of $H_n$ and $\psi_n(\partial S_n)$ is the internal ray of $H_n$ ending at the critical point of $f_n$.
    \item The arc $I_{z_n} \cup \Gamma_n$ can be chosen to depend uniformly continuously on $\theta \in \Irrat$ (as a subspace of $\TheCpt$). 
    Consequently, the gluing map $\psi_n$, the resulting map $f_n$, the pseudo-Siegel disks $\hat{Z}^{[m]}_n$, and the Mother Hedgehog $H_n$ of $f_n$ also depend uniformly continuously on $\theta$.
\end{enumerate}
\end{theorem}

For $n \geq 1$, define $S_n^1$ to be the sector $\psi_n(S_{n+1})$ in the dynamical plane of $f_n$. Gluing the two sides of $S_n^1$ via $f_n$ projects the first return map of $f_n$ onto $S_n^1$ to $f_{n+1}: \Dom(f_{n+1}) \to \D$. See Figure \ref{fig:sec-renorm}.

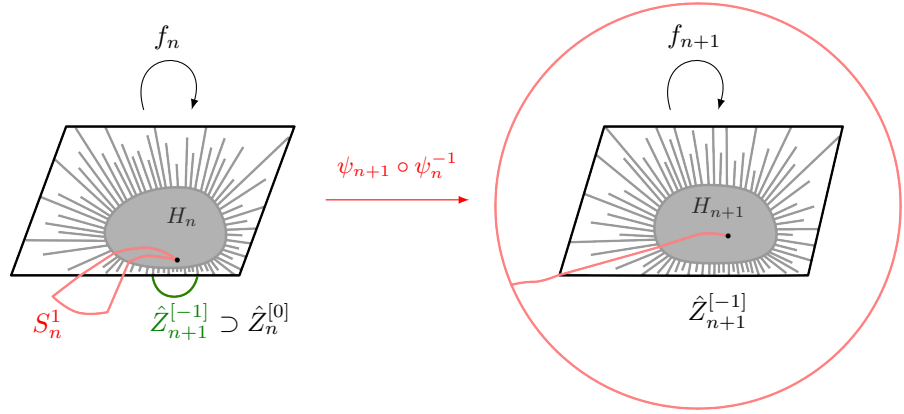
\begin{figure}
    \centering
    
    \begin{tikzpicture}
    \node[anchor=south west, inner sep=0] (image) at (0,0) {\includegraphics[width=1\linewidth]{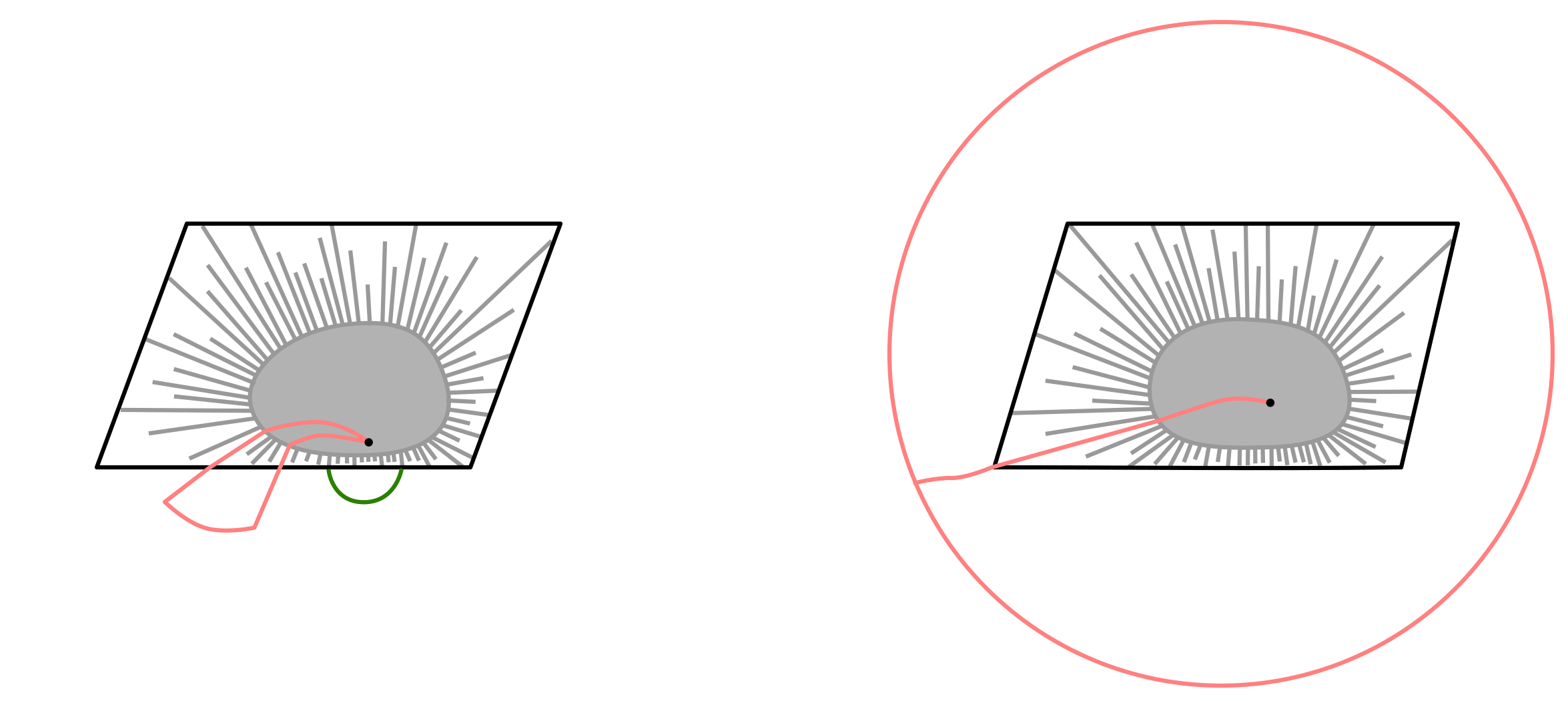}};
    \begin{scope}[
        x={(image.south east)},
        y={(image.north west)}
    ]
        \node [gray!30!black] at (0.24,0.47) {\scalebox{0.85}{$H_n$}};
        \node at (0.28,0.225) {\textcolor{green!50!black}{$\hat{Z}^{[-1]}_{n+1}$} $\supset\hat{Z}^{[0]}_n$ };
        \node [gray!30!black] at (0.8,0.49) {\scalebox{0.85}{$H_{n+1}$}};
        \node [black] at (0.8,0.25) {$\hat{Z}^{[-1]}_{n+1}$};
        
        \node [red] at (0.1,0.22) {$S_n^1$};
        \node [black] at (0.225,0.89) {$f_n$};
        \draw[-latex] (0.2,0.72) .. controls (0.18,0.87) and (0.27,0.87) .. (0.25,0.72);
        \node [black] at (0.775,0.89) {$f_{n+1}$};
        \draw[-latex] (0.75,0.72) .. controls (0.73,0.87) and (0.82,0.87) .. (0.80,0.72);
        \draw[red,-latex] (0.39,0.51) -- (0.54,0.51);
        \node [red] at (0.465,0.59) {\small $\psi_{n+1} \circ \psi_n^{-1}$};
    \end{scope}
\end{tikzpicture}
    
    \caption{Sector renormalization preserves pseudo-Siegel disks and Mother Hedgehogs.}
    \label{fig:sec-renorm}
\end{figure}

\subsection{Renormalization towers}
\label{ss:renormalization-limits}

Denote the collection of sector renormalizations of neutral quadratic polynomials by
\begin{align*}
    \orbsec := \{ (\Rsec^n f_\theta \: : \: \theta \in \Theta, n \geq 0 \}.
\end{align*}
It is forward invariant under sector renormalization.
Equip $\orbsec$ with the topology of uniform convergence on compact subsets. 
Define the infinite sequence space
\[
    \towsec = \left\{ \seq{\Rsec^n f  }_{n\geq 0} \: : \: f \in \orbsec \right\},
\]
equipped with the topology where for every $k \geq 0$, the projection map $\towsec  \to \orbsec, \seq{ \Rsec^n f  }_{n\geq 0} \mapsto \Rsec^k f$ is continuous.

It was demonstrated in \cite{DL26b} that renormalization sectors have uniformly bounded geometry.
This leads to an extension of Theorem \ref{thm:sectorial-bounds} below.

\begin{theorem}[\cite{DL26b}]
\label{thm:renorm-limits}
    The spaces $\orbsec$ and $\towsec$ are pre-compact in the following sense. 
    For any infinite sequence of maps $g_1$, $g_2$, $ g_3$, $\ldots$ in $\orbsec$, there is a subsequence $N_k \to \infty$ such that for every $n \geq 0$, the following properties hold.
    \begin{enumerate}[label = \textnormal{(\arabic*)}]
        \item In $\TheCpt$, the rotation number of $g_{N_k}$ converges to an element $\ttheta = \seq{ (\varepsilon_m, \bar{a}_m) }_{m \geq 1}$ of $\overline{\Theta}$ as $k \to \infty$.
        \item The sequence of maps $f_{n,k} := \Rsec^n g_{N_k}: \Dom(f_{n,k}) \to \D$ converges to a holomorphic map $f_n : \Dom(f_n) \to \D$ uniformly on compact subsets of some open domain $\Dom(f_n) \subset \D$.
        \item The map $f_n$ admits a simple critical point $v_{-1}(g_n)$ and a neutral fixed point at $0$ with rotation number $\bar{\mu}(\ttheta_n)$ where $\ttheta_n = \shift^n(\ttheta)$. 
        If $f_n'(0) = 1$ or equivalently $\bar{a}_{n+1} = \infty$, then $0$ is a simple parabolic fixed point of $f_n$.
        \item The domain $\Dom(f_n)$ contains a nested sequence of full connected compact sets 
        \[
            \ldots \subset \hat{Z}^{[1]}_n \subset \hat{Z}^{[0]}_n \subset \hat{Z}^{[-1]}_n = \hat{Z}_n
        \]
        called the \emph{pseudo-Siegel pinched disks} of $f_n$. 
        These have the following properties.
        \begin{enumerate}[label = \textnormal{(\alph*)}]
            \item Each $\hat{Z}^{[m]}_n$ is the Hausdorff limit of the level $[m]$ pseudo-Siegel disk of $f_{n,k}$ as $k \to \infty$.
            \item $\hat{Z}^{[-1]}_n$ and $\hat{Z}^{[0]}_n$ are closed $K(\threshold)$-quasidisks.
            \item There exist universal constants $R$ and $R'$ with $0<R<R'<1$ such that
            \[
            \D(0,R) \subset \hat{Z}^{[-1]}_n \subset \D(0,R').
            \]
        \end{enumerate} 

        \item Denote $f_n^{[0]} = f_n$. Consider the time semigroup $\Time^n= \Time_{\ttheta_n}$ of $\ttheta_n$ and the corresponding return times $\qq_{n,[m]} = \qq_{[m]}(\ttheta_n) \in \Time^n$, $m \geq 0$. 
        The semigroup generated by the iterates of $f_{n,k}$ converges to a semigroup 
        \[
            \mathcal{F}_n = \{f_n^p \}_{p \in \Time^n}
        \]
        of commuting holomorphic maps parametrized by $\Time^n$ in the following sense.
        \begin{enumerate}[label = \textnormal{(\alph*)}]
            \item For every $m \geq 0$, the pre-renormalizations $f_{n,k}^{[m]}: \Dom(f_{n,k}^{[m]}) \to \D$ converge to a holomorphic map $f_n^{[m]} = f_n^{\qq_{n,[m]}}: \Dom(f_n^{[m]}) \to \D$ uniformly on compact subsets as $k \to \infty$. 
            In $\D \backslash \{0\}$, $\Dom(f_n^{[m]}) \backslash \{0\}$ is an open neighborhood of $\hat{Z}^{[m-1]}_n \backslash \{0\}$.
            Moreover, the map $f_n^{[m]}$ is injective on $\hat{Z}^{[m-1]}_n$, and $\hat{Z}^{[m-1]}_n$ is almost invariant under $f_n^{[m]}$. 
            \item If $\bar{a}_{n+m+1} < \infty$, then $f_n^{[m+1]}$ is an iterate of $f_n^{[m]}$. 
            Else, $f_n^{[m+1]}$ is a ``Lavaurs map'' of $f_n^{[m]}$.
        \end{enumerate}
        
        \item The compact set
        \[
            H_n := \bigcap_{ m \geq -1} \hat{Z}^{[m]}_n
        \]
        satisfies the following properties.
        \begin{enumerate}[label = \textnormal{(\alph*)}]
            \item $H_n$ contains $0$ and the critical point $v_{-1}(g_n)$.
            \item $H_n$ is the Hausdorff limit of the Mother Hedgehog of $f_{n,k}$, and the corresponding Riemann mapping $\C \backslash \overline{\D} \to \C \backslash H(f_{n,k})$ sending $1$ to the critical value $v_0(f_{n,k})$ converges uniformly on compact subsets to the Riemann mapping of the complement of $H_n$.
            \item Every map in $\mathcal{F}_n$ restricts to a self-homeomorphism of $H_n$.
            \item $H_n$ satisfies the properties stated in Theorem \ref{thm:mother-hedgehog} (2)--(4) for $f_n$.
            \item The boundary of $H_n$ is equal to the \emph{postcritical set} $P_n$, that is, the closure of the orbit of the critical point $v_{-1}(f_n)$ of $f_n$ under the semigroup $\mathcal{F}_n$.
        \end{enumerate}

        \item 
        As $k \to \infty$, the nest of renormalization sectors $S^m(f_{n,k})$, $m \geq 1$ of $f_{n,k}$ described in Theorem \ref{thm:sectorial-bounds} converges to a nest of sectors 
        \[
        S_n^1 \supset S_n^2 \supset S_n^3 \supset \ldots
        \]
        in the Carath\'eodory topology of open disks with basepoint $0$.
        These sectors satisfy an analog of Theorem \ref{thm:sectorial-bounds} (1)--(4) for $f_n$. 
        In particular, for any $m \geq -1$ and $j \geq 1$, the conformal gluing map for the sector $S_n^j$ projects $f_n^{[m+j+1]}$ and $\hat{Z}^{[j+m]}_n$ into $f_{n+j}^{[m+1]}$ and $\hat{Z}^{[m]}_{n+j}$.
    \end{enumerate}
\end{theorem}

Most of the notations introduced in this theorem are recorded in \S\ref{sss:notation-for-towsec}.

Note that the structure of the pseudo-Siegel pinched disks $\hat{Z}_n^{[m]}$ are similar to those for quadratic polynomials as described previously.
Namely, each $\hat{Z}_n^{[m]}$ is the union of the Mother Hedgehog $H_n$ together with the level $s$ principal parabolic fjord and its preimages up to time $\qq^n_{[s]}$ for all $s \geq m$ with $\bar{a}_{n+s+1} \geq \threshold$.

As we include the limiting maps $f_0$ and the limiting renormalization orbits $\seq{f_n}_{n\geq 0}$, we obtain compactifications $\overline{\orbsec}$ and $\overline{\towsec}$ of $\orbsec$ and $\towsec$ respectively. 
The space $\overline{\orbsec}$ is equipped with the topology of uniform convergence on compact subsets.
The space $\overline{\towsec}$ is equipped with the topology such that
\begin{itemize}
    \item the projection map
    \[
        \overline{\towsec} \to \overline{\orbsec}, \qquad \seq{f_n}_{n\geq 0} \mapsto f_0
    \]
    is continuous;
    \item the first pre-renormalization $\seq{f_n}_{n\geq 0} \mapsto f_0^{[1]}$ is continuous with respect to the topology of uniform convergence on compact subsets;
    \item the shift map (the renormalization operator)
    \[
        \Rsec: \overline{\towsec} \to \overline{\towsec}, \quad \seq{f_0,f_1,f_2,\ldots} \mapsto \seq{f_1,f_2,f_3,\ldots}
    \]
    is continuous.
\end{itemize} 

Every forward tower $\fbold = \seq{f_n}_{n\geq 0}$ in $\overline{\towsec}$ comes with an associated rotation number 
\[
    \ttheta = \ttheta(\fbold) = \seq{ (\varepsilon_n, \bar{a}_n)}_{n\geq 1} \in \overline{\Theta}.
\]
The rotation number as a map 
$\ttheta: \overline{\towsec} \to \TheCpt$
is continuous\footnote{We will not need this fact. The proof will be provided in \cite{DLL}.} and satisfies the equation
\[
    \shift \circ \ttheta = \ttheta \circ \Rsec.
\]

\subsection{The top regularized fjord}
\label{ss:fjord}

The next theorem states that parabolic fjords have uniform size. 
Consider the combinatorial threshold $\threshold \in \N$ in the construction of pseudo-Siegel disks in the previous subsection. 

\begin{definition}
    We say that a map $f \in \overline{\orbsec}$ is \emph{near-parabolic} if the rotation number $\theta$ of $f$ at the fixed point $0$ satisfies $|\theta| < 1/\threshold$, (\emph{simply}) \emph{parabolic} if $\theta = 0$.
\end{definition}

\begin{definition}
    For any orbit $\fbold = \seq{f_n}_{n\geq 0} \in \overline{\towsec}$ with $f_n$ being near-parabolic for some $n \geq 0$, the set
\[
    \hat{V}_n = \hat{V}_n(\fbold) := \overline{ \hat{Z}_n^{[-1]}(\fbold) \backslash \hat{Z}_n^{[0]}(\fbold) },
\]
    which is non-empty, will be called the $n$\textsuperscript{th} \emph{top regularized fjord} of $\fbold$.
\end{definition}

\begin{theorem}[Uniformization of fjords]
\label{thm:uniform-fjords}
    There exist uniform constants $K>1$ and $C>0$ such that the following holds. 
    Consider $\fbold=\seq{ f_n  }_{n\geq 0} \in \overline{\towsec}$ with combinatorics $\thetabold = \seq{ (\varepsilon_n, \bar{a}_n)  }_{n\geq 1}$.
    Suppose $f_0$ is near-parabolic with rotation number $\theta = \bar{\mu}(\thetabold)$.
    \begin{enumerate}
        \item The top regularized fjord $\hat{V}_0(\fbold)$ contains a unique fixed point $\beta$ of $f_0$ satisfying
        \[
            C^{-1} |\theta| \leq |\beta| \leq C |\theta|.
        \]
        The fixed point $\beta$ is repelling if $\theta \neq 0$, parabolic if $\theta =0$.
        \item There exists a $K$-quasiconformal map $\chi_{\fbold} : \C \to \C$ such that
        \begin{align*}
            \chi_{\fbold} \big( \hat{Z}_0^{[-1]}(\fbold) \big) &= \overline{\D}, \qquad
            & \chi_{\fbold}(0) &= i\theta, \\
            \chi_{\fbold}(
            \hat{V}_0 \big( \fbold) \cap \hat{Z}_0^{[0]}(\fbold) \big) &= [-1,1] \subset \R,\qquad& \chi_{\fbold} (\beta) &= -i\theta.
        \end{align*}
    \end{enumerate}
\end{theorem}

See Figure \ref{fig:fjord} for an illustration.

\begin{figure}
    \centering
    \begin{tikzpicture}[scale=1.2]
        \draw[line width=0.5pt,-latex] (-0.75,0) -- (0.75,0);
        \node [black, font=\bfseries] at (0,0.25) {$\chi_{\fbold}$};

        \draw[color=black, thick] (-5,-1) -- (-1.75,-1) .. controls (-1,-1) .. (-1,0) .. controls (-1,2) .. (-2.5,2) -- (-3.5,2) .. controls (-5,2) .. (-5,0.5) -- (-5,-1);
        \draw[color=black, thick] (-3.25,-1) .. controls (-3.25,-2) and (-1.75,-2) .. (-1.75,-1);
        \node [color=black, font=\bfseries] at (-4,1.5) {\small $\hat{Z}^{[0]}(\fbold)$};
        \node [color=black, font=\bfseries] at (-1.2,-1.3) {\small $\hat{Z}^{[-1]}(\fbold)$};
        
        \filldraw[color=black] (-2.5,-0.8) circle (0.03);
        \node [black, font=\bfseries] at (-2.5,-0.55) {\footnotesize $0$};
        \filldraw[color=black] (-2.5,-1.2) circle (0.03);
        \node [black, font=\bfseries] at (-2.5,-1.45) {\footnotesize $\beta$};
        
        \draw[color=black, fill=white, thick](3,0) circle (2);
        
        \draw[color=black, thick] (1,0) -- (5,0);
        
        \filldraw[color=black] (3,0.301) circle (0.03);
        \node [black, font=\bfseries] at (3,0.6) {\footnotesize $i\lambda_\theta$};
        \filldraw[color=black] (3,-0.301) circle (0.03);
        \node [black, font=\bfseries] at (3,-0.6) {\footnotesize $-i\lambda_\theta$};

        \node [black, font=\bfseries] at (-4.5,-1.5) {$f_0$};
        \draw[line width=0.5pt,-latex] (-4.3,-1.2) .. controls (-4.5,-1.8) and (-3.5,-1.8) .. (-3.7,-1.15);
        
        \node [black, font=\bfseries] at (1.8,0.75) {\footnotesize $\rho_\theta$};
        \draw[line width=0.5pt,-latex] (1.5,0.2) .. controls (1.5,0.7) and (2,0.7) .. (2,0.2);
\end{tikzpicture}
    \caption{Uniformization of the top level fjord}
    \label{fig:fjord}
\end{figure}
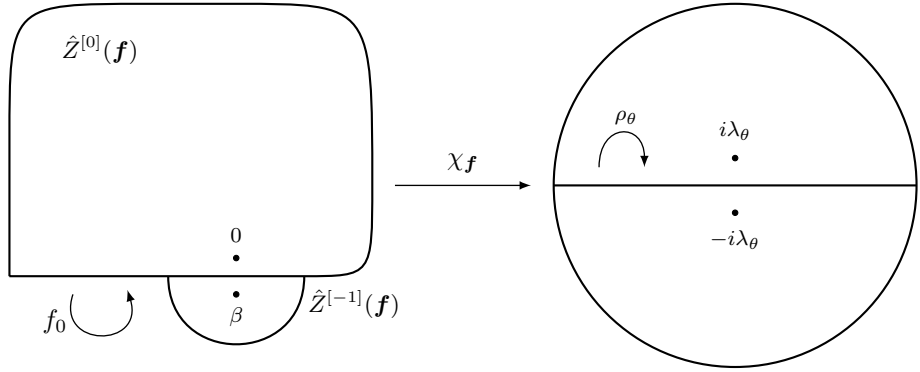

\begin{proof}
    Assume $\theta \neq 0$ and denote $b_1 = b_1(\theta)$.
    The existence of the other fixed point $\beta$ was proven in \cite{DL26b}. 
    By Theorem \ref{thm:renorm-limits}, there exists a universal constant $C>0$ such that whenever $f_0$ is near-parabolic, then $|(f_0)''(0)| > C$. 
    (Otherwise, there would be a degenerate parabolic bifurcation.) 
    By solving for the fixed point $\beta$ near $0$, we obtain the estimate in (1).

    For any closed Jordan disk $D$, any point $x$ in the interior of $D$, and any interval $E \subset \partial D$, we denote by $\omega(D,x,E)$ the harmonic measure of $E$ relative to the pointed domain $(D,x)$, that is, the probability of a Brownian motion in the interior of $D$ starting at $x$ to land at $E$.
    
    Consider the quasiarc
    \[
    J := \partial \hat{V}_0 \cap \partial \hat{Z}^{[0]}_0.
    \]
    The set of pre-critical points 
    \[
    \{v^0_{-j} \: : \: \threshold' + 1 \leq j \leq b_1 - \threshold' \}
    \]
    splits $J$ 
    into consecutive intervals $J_1, J_2, \ldots, J_{b_1 - 2\threshold'-1}$ where $f_0 (J_j) = J_{j+1}$.
    According to \cite{DL22, DL26b}, both $\hat{Z}_0^{[0]}$ and $\hat{Z}_0^{[-1]}$ are uniformly quasiconformal, and for all $j$,
    \begin{equation}
        \label{eqn:welding-1}
        \omega \big(\hat{Z}_0^{[0]}, 0, J_j \big) \asymp \omega \big( \hat{Z}_0^{[0]},0, \partial \hat{Z}_0^{[0]} \backslash J \big), \qquad
        \omega \big( V_0, \beta, J_j \big) \asymp \omega \big( V_0 ,0, \partial V_0 \backslash J \big) 
    \end{equation}

    Let $\rho_\theta$ be the unique real M\"obius transformation such that 
    \begin{enumerate}
    \item[(i)] $\rho_\theta^{b_1-2\threshold'-1}(-1)=1$,
    \item[(ii)] $\rho_\theta$ is elliptic having a fixed point $i \lambda_\theta$ in $\UHP$ with multiplier $e^{2\pi i/b_1}$.
    \end{enumerate}
    By elementary calculations, we have 
    \begin{equation}
    \label{eqn:lambda-fixed-point}
        \lambda_\theta 
        = \tan \frac{\pi( 2\threshold'+1)}{2b_1} \sim \pi \theta \left( \threshold'+\frac{1}{2} \right).
    \end{equation}
    The orbit of $-1$ under $\rho_\theta$ splits the interval $[-1,1]$ into consecutive intervals $E_1$, $E_2$, $\ldots$, $E_{b_1-2\threshold'-1}$ where $\rho(E_j) = E_{j+1}$.
    Since $\rho_\theta$ is elliptic, for any $j,k \in \{1,\ldots,b_1-2\threshold'-1\}$,
    \begin{equation}
        \label{eqn:welding-2}
        \omega ( \D \cap \UHP, i \lambda_n, E_j) \asymp \omega ( \D \cap \UHP, i \lambda_n, E_k) \asymp \omega ( \D \cap \UHP, i \lambda_n, \partial \D \cap \UHP).
    \end{equation}
    Then, (\ref{eqn:welding-1}) and (\ref{eqn:welding-2}) allow us to construct a simultaneous uniformization of $\hat{Z}_0^{[0]}$ and $\hat{Z}_0^{[-1]}$ in the following sense.
    There exists a uniformly quasiconformal map $\chi_{\fbold}: \C \to \C$ such that 
    \begin{align*}
    \chi_{\fbold}(\hat{Z}_0^{[-1]}) &= \overline{\D}, \qquad
            & \chi_{\fbold}(0) &= i\lambda_\theta, \\
            \chi_{\fbold}(\hat{Z}_0^{[-1]} \cap \hat{Z}_0^{[0]}) &= [-1,1] \subset \R,\qquad& \chi_{\fbold}(\beta) &= -i\lambda_\theta,
    \end{align*}
    and
    \[
        \chi_{\fbold}(J_j) = E_j \qquad \text{for all } j \in \{1,2,\ldots, b_1-2\threshold'-1\}.
    \]
    Up to post-composition with some other uniformly quasiconformal map, we can replace $\lambda_\theta$ with $\theta$ by virtue of (\ref{eqn:lambda-fixed-point}).

    The limiting parabolic case when $\theta = 0$ follows from the observation that
    as $\theta \to 0$, then $\rho_\theta$ converges to the real M\"obius transformation
    \[
        \rho_0(z) = \frac{z}{1- z/(\threshold'+\frac{1}{2})}
    \]
    which admits a simple parabolic fixed point at $\displaystyle \lim_{\theta \to 0} i\lambda_\theta = 0$.
\end{proof}

\subsection{Renormalization triangulation}
\label{ss:renorm-tiling}

Let us fix $\fbold = \seq{f_n}_{n\geq 0}$ in $\overline{\towsec}$ with combinatorics $\ttheta = \ttheta(\fbold) = \seq{ (\varepsilon_n, \bar{a}_n)}_{n\geq 1}$. 
We will adapt the notation associated to $\fbold$ summarized in \S\ref{sss:notation-for-towsec}.
Below, we will describe a dynamical triangulation 
$\hat{\Delta}^{[m]}_n = \hat{\Delta}^{[m]}_n(\fbold)$ 
of the pseudo-Siegel pinched disks $\hat{Z}^{[m]}_n$ for every $n \geq 0$ and $m \geq -1$.
 
For every $m \geq 1$, denote 
\[
    \check{\qq}^n_{[m]} := \qq^n_{[m]} - \varepsilon_{n+m} \varepsilon_{n+m+1} \qq^n_{[m-1]}.
\]
For $p \in \Kont^{n}$, we will denote by $I^n_p$ the internal ray of the $n$\textsuperscript{th} Mother Hedgehog $H_n$ starting at $0$ and ending at $v^n_p$.
Recall that for every $m \geq 1$, the boundary of the renormalization sector $S_n^m$ contains two internal rays $I^n_{-\qq^n_{[m]}}$ and $I^n_{-\check{\qq}^n_{[m]}}$.
Denote by
\[
\psi_{n,m} : S_n^m \to \D
\]
the conformal gluing map for $S_n^m$ projecting $f_n^{[m]}$ to $f_{n+m}$.
For $m=1$, we simply denote $\psi_n=\psi_{n,1}: S_n \to \D$.

Firstly, the triangulation $\hat{\Delta}^{[-1]}_n$ consists of two triangles, which are the closure of the connected components of $\hat{Z}^{[-1]}_n$ with the internal rays $I^n_{-1}$ and $I^n_0$ removed.
In general, for $m \geq 1$, we define $\hat{\Delta}^{[m-1]}_n$ out of $\hat{\Delta}^{[-1]}_{n+m}$ by pulling back and spreading around as follows.

The lift of $\hat{\Delta}^{[-1]}_{n+m}$ under $\psi_{n,m}$ is a triangulation of $\hat{Z}_n^{[m-1]} \cap S_n^m$ consisting of two elements $A_n^m$ and $\check{A}_n^m$.
The two triangles are separated by $I^n_0$.
Without loss of generality, we will assume that $A_n^m$ is bounded by $I^n_0$ and $I^n_{-\qq^n_{[m]}}$, and $\check{A}_n^m$ is bounded by $I^n_0$ and $I^n_{-\check{\qq}_{[m]}}$.
For $p \in \Time^n$, denote by $A_n^m(-p)$ (resp. $\check{A}_n^m(-p)$) the closure of the lift of the interior of $A_n^m$ (resp. $\check{A}_n^m$) under $f_n^p$ that contains the fixed point $0$ on its boundary.

\begin{lemma}
\label{lem:triangulation-first-return}
    The first return map of the semigroup $\mathcal{F}_n$ back to $\hat{Z}_n^{[m-1]} \cap S_n^m$ is the pair of maps
    \[
        f_n^{\qq^n_{[m]}}: \check{A}_n^m (-\qq^n_{[m]}) \to \check{A}_n^m (0), \quad 
        f_n^{\check{\qq}^n_{[m]}}: A_n^m(-\check{\qq}^n_{[m]}) \to A_n^m(0).
    \]
    In particular, for any point $z$ in $\hat{Z}_n^{[m-1]}$, there exists some time $p \in \Time^n \cup \{0\}$ such that $f_n^p(z)$ is in $\hat{Z}_n^{[m-1]} \cap S_n^m$; the smallest of such $p$ satisfies either $p< \qq^n_{[m]}$ or $p < \check{\qq}^n_{[m]}$.
\end{lemma}

\begin{proof}
    The case when $m=1$ follows from the way the pseudo-Siegel disk $\hat{Z}^{[0]}_n$ is defined by principal dams and their pullbacks. 
    The case when $m \geq 2$ follows from induction in a similar way.
\end{proof}

This lemma implies that we have a well-defined triangulation $\hat{\Delta}^{[m-1]}_n$ of $\hat{Z}_n^{[m-1]}$ given by
\[
\hat{\Delta}^{[m-1]}_n = \left\{ A_n^m(-p) \right\}_{0 \leq p < \check{\qq}^n_{[m]}} \cup 
\left\{ \check{A}_n^m(-p) \right\}_{0 \leq p < \qq^n_{[m]}}.
\]
See Figure \ref{fig:triangulation} for an example when $m=1$.
By construction, we have the following property.

\begin{figure}
        \centering
       \begin{tikzpicture}
    \node[anchor=south west,inner sep=0] (image) at (0,0) {\includegraphics[width=0.98\linewidth]{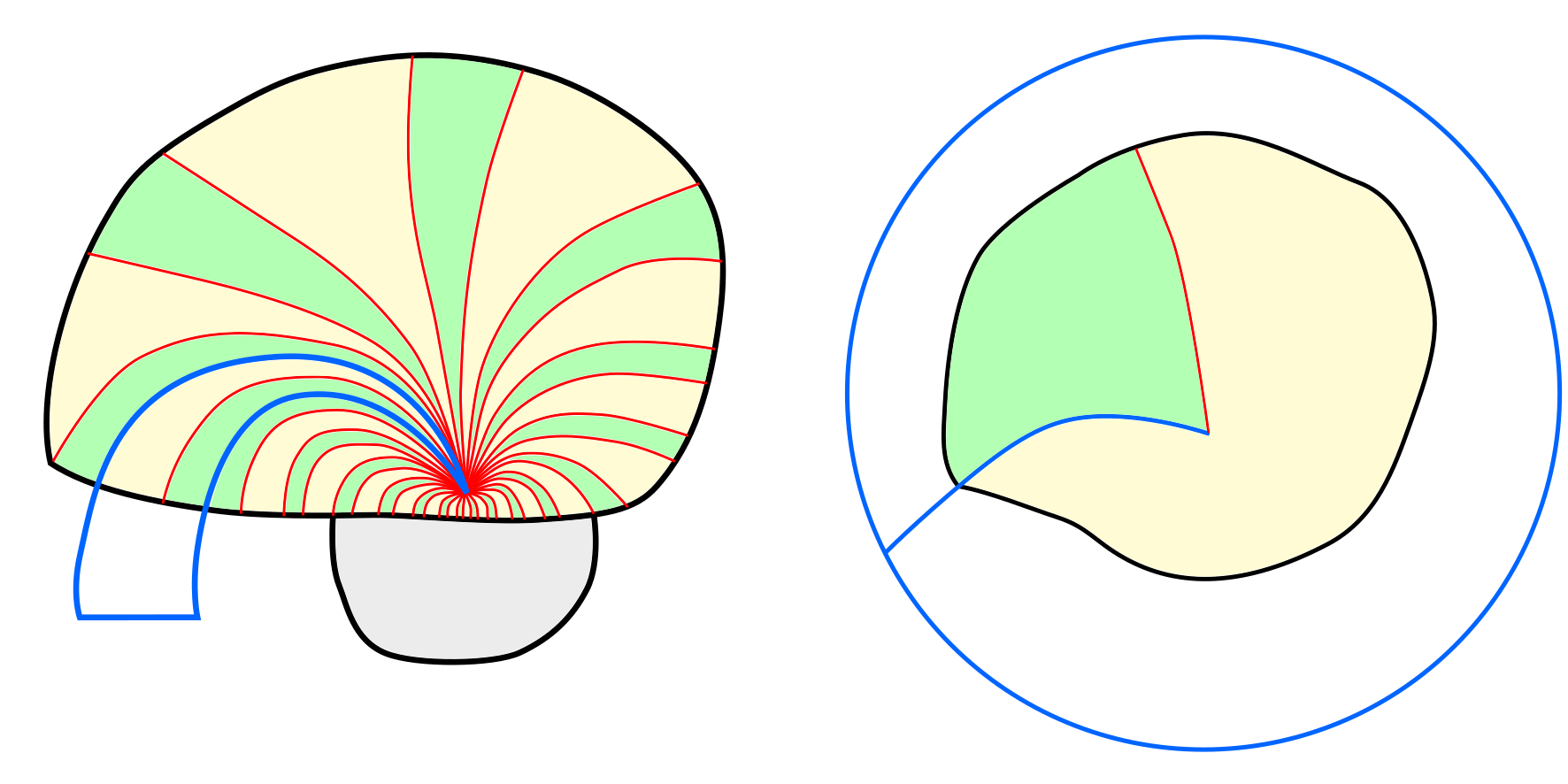}};
    \begin{scope}[
        x={(image.south east)},
        y={(image.north west)}
    ]
        \node [black] at (0.295,0.375) {\small $\bullet$};
        \draw[-latex] (0.32,0.965) .. controls (0.26,1.04) and (0.15,1) .. (0.095,0.865);
        \draw[-latex] (0.075,0.82) .. controls (-0.01,0.77) and (-0.03,0.55) .. (0.01,0.45);
        
        \node [black] at (0.035,0.4) {\small $\bullet$};
        \node [black] at (0.028,0.36) {\small $v_{n,-1}$};
        \node [black] at (0.105,0.355) {\small $\bullet$};
        \node [black] at (0.095,0.31) {\small $v_{n,0}$};
        \node [black] at (0.767,0.448) {$\bullet$};
        \node [black] at (0.77,0.41) {$0$};
        \node [black] at (0.45, 0.88) {\small $\hat{Z}_n^{[0]}$};
        \node [black] at (0.42, 0.24) {\small $\hat{Z}_n^{[-1]}$};
        \node [black] at (0.78, 0.2) {\small $\hat{Z}_{n+1}^{[-1]}$};
        \node [black] at (0.075,0.055) {\small $f_n$};
        \draw[-latex] (0.05,0.16) .. controls (0.055,0.09) and (0.105,0.09) .. (0.12,0.17);
        \node [black] at (0.69,0.64) {\small $f_n$};
        \draw[-latex] (0.69,0.475) .. controls (0.70,0.57) and (0.72,0.6) .. (0.755,0.62);
        \node [blue] at (0.35,0.03) {\small $\psi_n$};
        \draw[blue,-latex] (0.145,0.21) .. controls (0.25,0.03) and (0.45,0.04) .. (0.58,0.17);
        \node [blue] at (0.02,0.24) {\small $S_n$};
    \end{scope}
\end{tikzpicture}
    \caption{The triangulation $\hat{\Delta}_n^{[0]}$ of $\hat{Z}^{[0]}_n$ on the left is obtained by pulling back the triangulation $\hat{\Delta}_{n+1}^{[-1]}$ of $\hat{Z}^{[-1]}_{n+1}$ on the right and spreading it around.}
    \label{fig:triangulation}
\end{figure}
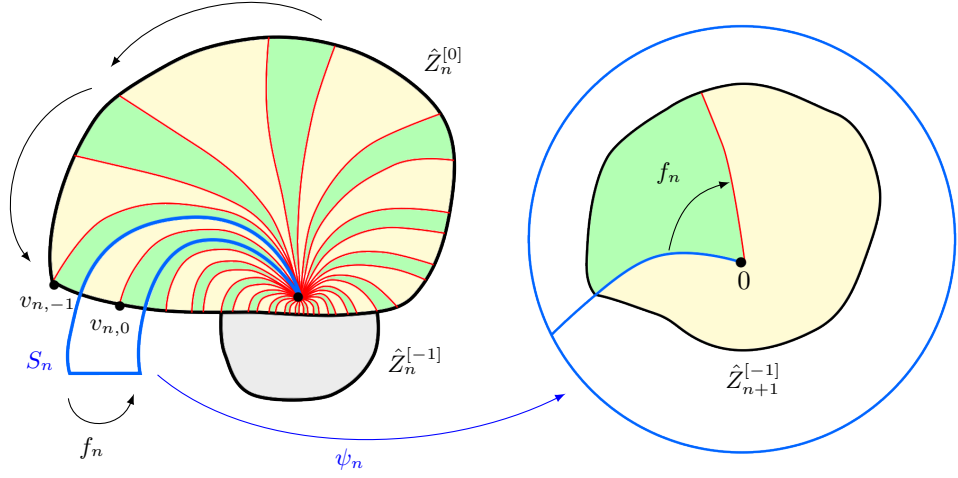

\begin{lemma}
\label{lem:triangulation-invariance-01}
    The triangulations $\hat{\Delta}_n^{[m]}$ of $\hat{Z}_n^{[m]}$ are invariant under sector renormalization, that is, for $m, n \geq 0$,
    \[
        \psi_{n,1}(\hat{\Delta}_n^{[m]} \cap S_n^1) = \hat{\Delta}_{n+1}^{[m-1]}.
    \]
\end{lemma}


\section{Estimates in logarithmic coordinates}
\label{sec:logarithmic}

\begin{figure}
    \centering
    
    \begin{tikzpicture}
    \node[inner sep=0] (image) at (6,15) {\includegraphics[width=0.98\linewidth]{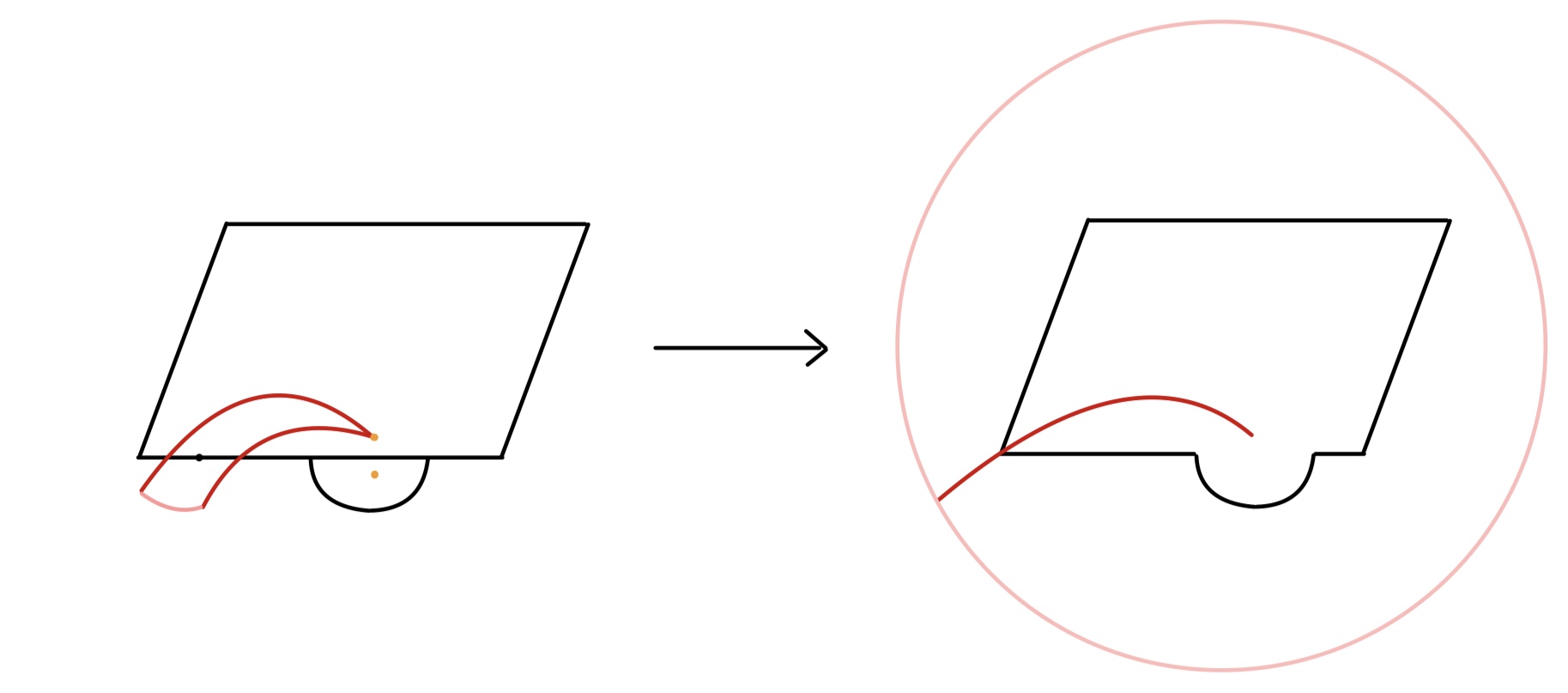}};
    \node[anchor=south west, inner sep=0] (image) at (0,0) {\includegraphics[width=0.95\linewidth]{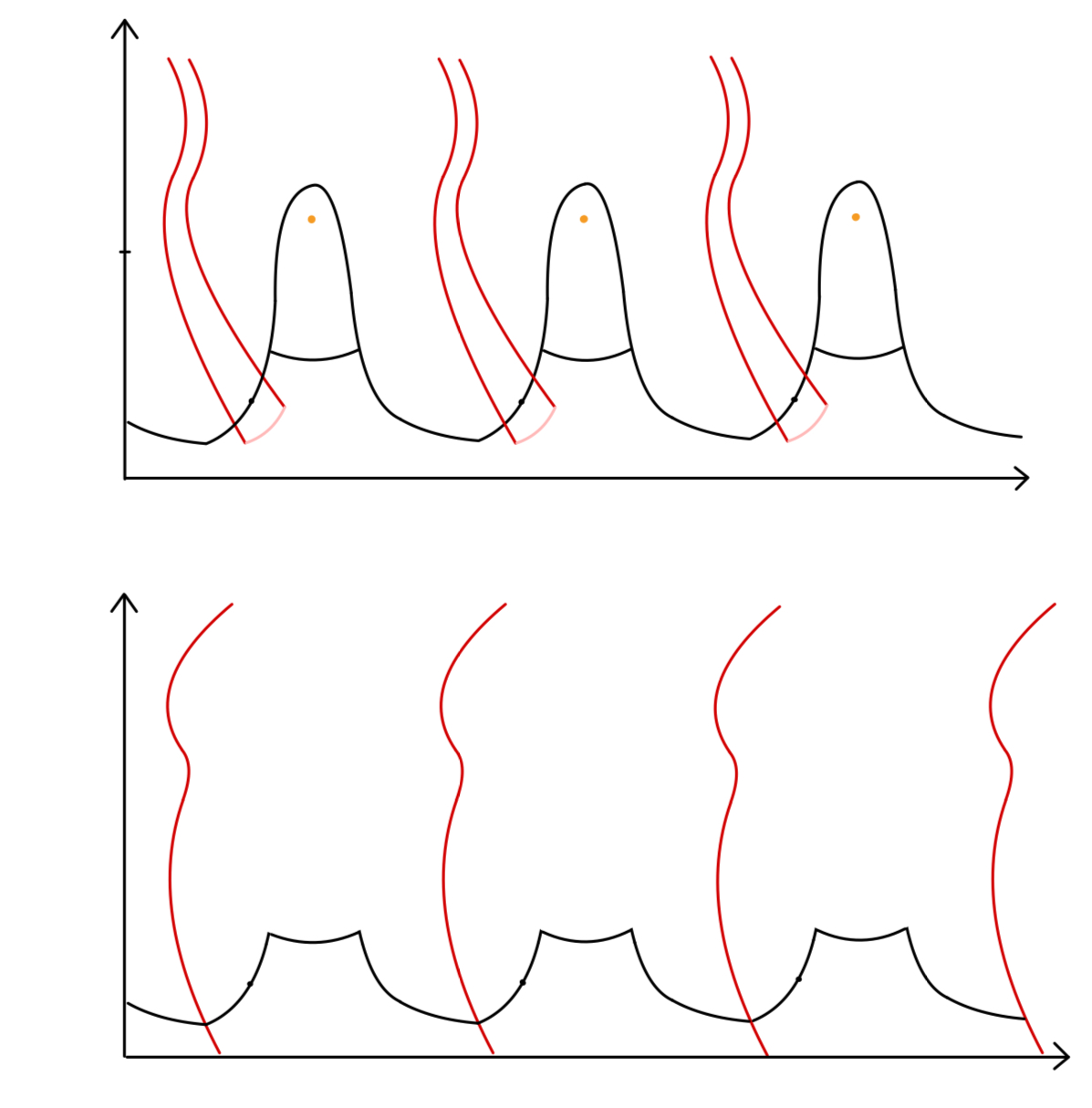}};
    \begin{scope}[
        x={(image.south east)},
        y={(image.north west)}
    ]
        \node [black, font=\bfseries] at (0.32,1.27) {$\hat{Z}^{[0]}_n$};
        \node [black, font=\bfseries] at (0.3,1.12) {$\hat{Z}^{[-1]}_n$};
        \node [black, font=\bfseries] at (0.8,1.25) {$\hat{Z}^{[-1]}_{n+1}$};
        \node [black, font=\bfseries] at (0.04,0.77) {$\frac{1}{2\pi}\log\frac{1}{|\theta_n|}$};
        
        \node [red, font=\bfseries] at (0.19,1.22) {$S_n$};
        \node [red, font=\bfseries] at (0.22,0.87) {$\mathcal{S}_n$};
        \node [red, font=\bfseries] at (0.485,0.87) {$\mathcal{S}_n + 1$};
        \node [red, font=\bfseries] at (0.735,0.87) {$\mathcal{S}_n + 2$};
        \node [red, font=\bfseries] at (0.3,0.3) {$\Sigma_{n+1}$};
        \node [red, font=\bfseries] at (0.545,0.3) {$\Sigma_{n+1} + 1$};
        \node [red, font=\bfseries] at (0.8,0.3) {$\Sigma_{n+1} + 2$};
        
        \node [black, font=\bfseries] at (0.25,0.515) {$\phi_n$};
        \draw[black,line width=0.5pt,-latex] (0.245,0.59) .. controls (0.29,0.51) .. (0.28,0.45);        
        \node [black, font=\bfseries] at (0.37,0.515) {$\phi_n$};
        \draw[black,line width=0.5pt,-latex] (0.49,0.59) .. controls (0.43,0.52) .. (0.315,0.45);
        \node [black, font=\bfseries] at (0.615,0.515) {$\phi_n$};
        \draw[black,line width=0.5pt,-latex] (0.73,0.59) .. controls (0.58,0.52) .. (0.37,0.45);
        
        \node [black, font=\bfseries] at (0.415,1.005) {$\ftilde_n$};
        \draw[line width=0.5pt,-latex] (0.405,0.92) .. controls (0.355,1) and (0.45,1) .. (0.435,0.92);
        \node [black, font=\bfseries] at (0.565,0.455) {$\ftilde_{n+1}$};
        \draw[line width=0.5pt,-latex] (0.555,0.37) .. controls (0.505,0.45) and (0.60,0.45) .. (0.585,0.37);
        \node [black, font=\bfseries] at (0.225,1.39) {$f_n$};
        \draw[line width=0.5pt,-latex] (0.2,1.32) .. controls (0.18,1.38) and (0.27,1.38) .. (0.25,1.32);
        \node [black, font=\bfseries] at (0.775,1.39) {$f_{n+1}$};
        \draw[line width=0.5pt,-latex] (0.75,1.32) .. controls (0.73,1.38) and (0.82,1.38) .. (0.80,1.32);

        \node [black, font=\bfseries] at (0.115,1.01) {\small $\imag(\zeta)$};
        \node [black, font=\bfseries] at (0.115,0.49) {\small $\imag(\zeta)$};
        \node [black, font=\bfseries] at (0.89,0.535) {\small $\real(\zeta)$};
        \node [black, font=\bfseries] at (0.91,0.01) {\small $\real(\zeta)$};

        \filldraw[color=white] (0.45,1.17) circle (0.1);
        \draw[line width=0.5pt,-latex] (0.39,1.22) -- (0.54,1.22);
    \end{scope}
\end{tikzpicture}
    
    \caption{Sector renormalization in logarithmic coordinates}
    \label{fig:euclidean-vs-logarithmic}
\end{figure}
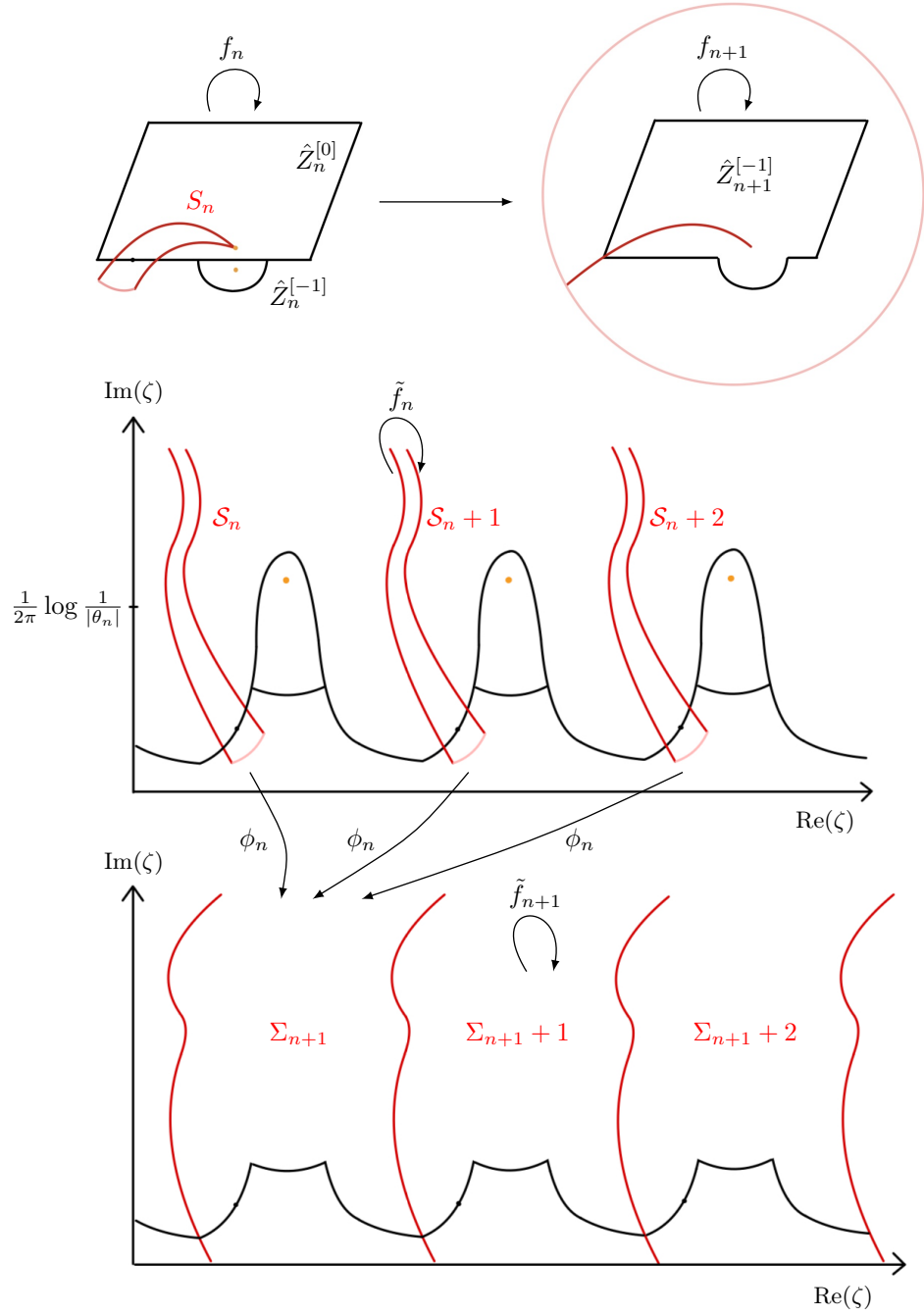

Let us fix $\fbold = \seq{ f_n  }_{n\geq 0} \in \overline{\towsec}$.
For $n \geq 0$, let $\theta_n \in \left[ -\frac{1}{2},\frac{1}{2} \right]$ be the rotation number of the neutral fixed point $0$ of $f_n$. 
In the dynamical plane of $f_n$, consider the nest of pseudo-Siegel disks $\hat{Z}_n^{[m]}$, the nest of sectors $S_n^m$, and the sequence of conformal gluing maps 
\[
\psi_{n,m}: S_n^m \to \D
\]
projecting $f_n^{[m]}$ to $f_{n+m}$.

In this section, we will lift $\fbold$ under the exponential map 
\[
\exp(z) := e^{2\pi i z}
\]
and obtain estimates of the renormalization change of variables in logarithmic coordinates.

\subsection{Passing to logarithmic coordinates}

\label{ss:log-coordinates}

Let us lift $f_n: \Dom(f_n)\to \D$ via the exponential map $\exp$ to the unique map $\ftilde_n: \Dom(\ftilde_n)\to \UHP$ on an open unbounded connected subset of the upper half plane $\UHP$ that satisfies
\begin{equation}
\label{eqn:limit_fbold}
    \lim_{\imag(z) \to +\infty} \ftilde_n(z) - z = \theta_n.
\end{equation}
Let $T(z):=z+1$ be the unit translation. We have 
\[
    \exp \circ \ftilde_n = f_n \circ \exp \qquad \text{ and } \qquad \ftilde_n \circ T = T \circ \ftilde_n.
\]
For every $n\geq 0$ and $m \geq -1$, denote the lift of the pseudo-Siegel disk $\hat{Z}^{[m]}_n$ by
\[
    \hat{\mathcal{Z}}^{[m]}_n := \exp^{-1}\left(\hat{Z}^{[m]}_n\right).
\]
Each $\hat{\mathcal{Z}}^{[m]}_n$ is invariant under $T$ and $\hat{\mathcal{Z}}^{[-1]}_n$ is a uniformly quasiconformally equivalent to the upper half plane. 

Let us pick the unique lift $\vtilde_0^n$ under $\exp$ of the critical value $v_0^n$ of $f_n$ normalized such that
\[
    0 \leq \text{Re}(\vtilde_0^n) \leq 1.
\]
In the dynamical plane of $f_n$, for every $m \geq 1$, the sector $S_n^m$ lifts via $\exp$ to a pairwise disjoint union of open half-strips $T^j(\mathcal{S}_n^m), j \in \Z$, where $\mathcal{S}_n^m$ is the unique connected component containing $\vtilde_0^n$.
For brevity, we will denote 
\[
    S_n := S_n^1 \qquad \text{and} \qquad \mathcal{S}_n 
    := \mathcal{S}_n^1.
\]
For $n\geq 0$, the sides of $S^{n}$ with endpoint $0$ are identified and sent by the conformal gluing map $\psi_{n,1}$ to an arc $\alpha_{n+1}$ in the dynamical plane of $f_{n+1}$ connecting $0$ and $\partial \D$. 
Then, the open disk $\D \backslash \alpha_{n+1}$ lifts via $\exp$ to a pairwise disjoint union of open half-strips $T^j(\Sigma_{n+1}), j \in \Z$ whose closure is $\overline{\UHP}$. 
We will again choose $\Sigma_{n+1}$ to be the unique sector containing $\vtilde_0^{n+1}$. 
The conformal gluing map $\psi_{n,1}$ lifts to an infinite degree covering map 
\[
    \phi_n : \bigcup_{j \in \Z} T^j(\mathcal{S}_n) \to \Sigma_{n+1}
\]
with deck group generated by $T$. See Figure \ref{fig:euclidean-vs-logarithmic} for an illustration.

Given $\zeta \in \C$, if $\theta_n \neq 0$, denote
\[
    \imag(\zeta;\theta_n) := \imag (\zeta) - \frac{1}{2\pi} \log \frac{1}{|\theta_n|}.
\]
In this section, we prove the uniform fundamental estimates independent of $\fbold$. 

\begin{proposition}
\label{prop:almost-translation}
    There exist uniform constants $C, C'>0$ such that for all $n \geq 0$ and $\zeta \in \UHP$ with $\imag(\zeta) > C'$, we have
    \[
        \lvert \ftilde_n(\zeta) - \zeta - \theta_n \rvert \leq C e^{-2\pi \imag(\zeta)} 
        \quad \text{ and } \quad 
        \lvert \ftilde_n'(\zeta) - 1 \rvert \leq C e^{-2\pi \imag(\zeta)}.
    \]
\end{proposition}

In other words, $\ftilde_n$ is close to the translation by $\theta_n$, and the error decays exponentially in the imaginary value of $\zeta$.

\begin{proposition}
\label{prop:renormalization-change-of-variables}
    There exist uniform constants $C, C'>0$ such that for all $n \geq 0$ with $\theta_n \neq 0$ and for every point $\zeta \in \mathcal{S}_n$ with $\imag(\zeta ; \theta_n) > C'$, we have
    \[
        \Big| |\theta_n| \phi_n'(\zeta) - 1 \Big| \leq \frac{C}{\imag(\zeta;\theta_n)^2}
        \quad \text{ and } \quad 
        \Big| |\theta_n| \phi_n(\zeta) - \zeta + \kappa_n \Big| \leq \frac{C}{\imag(\zeta;\theta_n)},
    \]
    where $\kappa_n=\kappa_n(\tilde{f}_n)$ is some complex number with 
    \[
        \kappa_n = \frac{i}{2\pi} \log\frac{1}{|\theta_n|} + O(1).
    \]
\end{proposition}

In other words, the renormalization change of variables $\phi_n$ is close to the affine map $\zeta \mapsto \cfrac{\zeta - \kappa_n}{|\theta_n|}$ and the error decays at least linearly in $\imag(\zeta;\theta_n)$. 

Proposition \ref{prop:almost-translation} and the estimate for $\phi'_n$ in Proposition \ref{prop:renormalization-change-of-variables} essentially follow from controlling the radius of univalence of $\ftilde_n$ and $\phi_n$ respectively.
The estimate for $\phi_n$ and $\kappa$ requires one additional ingredient, namely Theorem \ref{thm:uniform-fjords}.

\begin{remark}
    With more refined analysis, we believe it should be possible to show that the error decays exponentially (see \cite[\S5]{Che13} for example), but we will not need it here.
\end{remark}

Without loss of generality, the proof of the two propositions above will be carried out only for the base level $n=0$.
We will simplify the notation by dropping the index $n=0$ and writing $f = f_0$, $\ftilde = \ftilde_0$, $\theta = \theta_0$, $\phi = \phi_0$,
 and $\mathcal{S} = \mathcal{S}_0$. 
We will also assume without loss of generality that $0 \leq \theta \leq \frac{1}{2}$; otherwise, we can conjugate $\ftilde$ with the reflection about the imaginary axis $x+iy \mapsto -x+iy$. 

\subsection{Distortion estimates for $\ftilde$}

By Theorem \ref{thm:renorm-limits} (4)(c), there exists a uniform constant $R \in (0,1)$ such that $f$ is univalent on the round disk $\D(0,R)$. 

\begin{lemma}
    For every univalent function $g: \D \to \C$ fixing $0$,
    \[
        \left| z \frac{g'(z)}{g(z)} - 1 \right| \leq \frac{2 |z|}{1-|z|} 
        \qquad \text{ for all } z \in \D.
    \]
\end{lemma}

\begin{proof}
    By pre-composition with the linear map $z \mapsto Rz$, we can assume $R=1$. Let us fix $z \in \D$. Consider the disk automorphism $h(w) = \frac{z-w}{1-\bar{z}w}$ and the univalent map
    \[
        G(w) = \frac{ g\circ h(z) - g(z) }{g'(z) h'(0)}, \quad z \in \D.
    \]
    It has the property that
    \[
        G(0) = 0, \quad G'(0) = 1, \quad G(z) = \frac{g(z)}{(1-|z|^2)g'(z)}.
    \]
    Conjugating with the inversion gives us $G(z^{-1})^{-1} = z + b_0 + \sum_{n=1} b_n z^{-n}$. A theorem of Bieberbach asserts that $|b_0| \leq 2$, and the Area Theorem asserts that $\sum_{n=1}^\infty n |b_n| \leq 1$. Therefore, $|G(z)^{-1} - z^{-1}| \leq 2 + |z|$ and so
    \[
        \left| \frac{z}{G(z)} -1 \right| \leq 2|z| + |z|^2.
    \]
    By writing $G(z)$ in terms of $g$, we obtain the desired inequality:
    \[
        \left|  z \frac{g'(z)}{g(z)} - 1 \right| 
        = \frac{1}{1-|z|^2} \left| \frac{z}{G(z)}-1 + |z|^2 \right|
        \leq \frac{2 |z|}{1-|z|}. \qedhere
    \]
\end{proof}

Let us now prove the first estimate.

\begin{proof}[Proof of Proposition \ref{prop:almost-translation}]
    Set $C' := \frac{1}{2\pi} \log\frac{2}{R}$.
    Pick a point $\zeta$ and let $z=\exp(\zeta)$.
    Assume $\imag(\zeta) \geq C'$ or equivalently $|z|\leq R/2$.
    Since $\exp \circ \ftilde = f \circ \exp$, we have 
    \[
    \ftilde'(\zeta) = \frac{z \,f'(z)}{f(z)}.
    \]
    Then, by the previous lemma,
    \[
        \left| \ftilde'(\zeta) - 1 \right| \leq \frac{2 |z|}{R - |z|} \leq \frac{4}{R} |z| = \frac{4}{R} e^{- 2 \pi \imag(\zeta)}.
    \]
    By taking the line integral of $\ftilde'-1$ along the vertical path $\{ \zeta + it \: : \: t \geq 0 \}$ and applying (\ref{eqn:limit_fbold}), we obtain 
    \[
        \lvert \ftilde(\zeta) - \zeta - \theta \rvert
        \leq \int_{\imag(\zeta)}^\infty \frac{4}{R} e^{-2\pi y} dy 
        \leq \frac{2}{\pi R} e^{- 2 \pi \imag(\zeta)}.
    \]
    This completes the proof with $C=\frac{4}{R}$.
\end{proof}

\subsection{Univalent extension of $\phi$}

Throughout the rest of the section, we will assume that 
\[
\theta \neq 0.
\]
We will introduce two uniform constants $C_0>0$ and $K_0 >0$, where
\begin{itemize}
    \item whenever $\imag(\zeta; \theta) \geq C_0$, we have
    \begin{equation}
    \label{eqn:f-in-log-coord}
        \lvert \ftilde(\zeta) - \zeta - \theta \rvert \leq  \frac{\theta}{11} 
        \quad \text{ and } \quad 
        \lvert \ftilde'(\zeta) - 1 \rvert \leq   \frac{\theta}{11}.
    \end{equation}
    \item $\mathcal{S}$ and $\Sigma_1$ are contained in the strip $\{\zeta \in \C \: : \: |\real (\zeta)| \leq K_0 \}$.
\end{itemize}
The first item follows from Proposition \ref{prop:almost-translation}, and the second is because internal rays have uniformly bounded spiraling number about the fixed point (by Theorem \ref{thm:renorm-limits} (6)(d)).

The next lemma is on a natural univalent extension of $\phi$.

\begin{lemma}
\label{lem:univalent-extension}
    There exists a uniform constant $C_1 \in (C_0,\infty)$ such that
    \begin{enumerate}
        \item the change of coordinates $\phi: \mathcal{S} \to \C$ extends to a univalent map on $\mathcal{S} \cup \mathcal{V}$ where 
    \[
        \mathcal{V} := \left\{ \imag(\zeta;\theta) \geq 0.1 \; | \real ( \zeta) | + C_1 \right\};
    \]
        \item whenever $\zeta$ and $\ftilde(\zeta)$ are in $\mathcal{S} \cup \mathcal{V}$, we have
    \[
        \phi \circ \ftilde(\zeta) = \phi(\zeta) + 1.
    \]
    \end{enumerate}
\end{lemma}

\begin{proof}
    Let $\mathbf{k} := \lfloor 1.1 K_0 /\theta \rfloor$, then
    \[
        0.1 K_0 \leq \frac{\mathbf{k}\theta}{11} < 0.1 K_0 + \frac{1}{22}.
    \]
    For any point $\zeta$, if $\imag(\zeta;\theta) \geq C_0 + 0.1 K_0 + 1/22$, then
    \begin{equation}
    \label{eqn:shifting-many}
        |\ftilde^k(\zeta)-\zeta-k\theta| \leq \frac{k\theta}{11} \quad \text{ for all } k \in \{1,2,\ldots,\mathbf{k} \}.
    \end{equation}

    Pick a point $\zeta_0^-$ on the left side of $\mathcal{S}$ and a proper arc $\alpha_0$ in $\mathcal{S}$ such that
    \begin{itemize}
        \item $\alpha_0$ starts from $\zeta_0^-$ and ends at $\zeta_0^+ = \ftilde(\zeta_0^-)$;
        \item $\alpha_0$ is contained in the strip
        \[
            \left\{
            \zeta \in \C \: : \: 0.1 K_0 + \frac{1}{22}\leq \imag (\zeta;\theta) - C_0 \leq C' 
            \right\}
        \]
        for some uniform constant $C'>0$;
        \item let $\gamma_0^-$ be the sub-ray of the left side of $\mathcal{S}$ emanating from $\zeta_0^-$ and let $\gamma_0^+=\ftilde(\gamma_0^-)$, then $\gamma_0^- \cup \alpha \cup \gamma_0^+$ bounds an infinite sector $\mathcal{U}_0 \subset \mathcal{S}$ contained in the half-strip 
        \[
            \left\{ 
            \zeta \in \C \: : \: |\real(\zeta)| \leq K_0 , \; \imag (\zeta;\theta) \geq C_0 + 0.1 K_0 + \frac{1}{22}
            \right\}.
        \]
    \end{itemize}
    Such objects exist because the sides of $\mathcal{S}$ are uniformly quasiconformal and have uniformly bounded spiralling numbers.

    For $k \in \N$, denote $\gamma_k^- = \ftilde^{-k}(\gamma_0^-)$, $\gamma_k^+ = \ftilde^k(\gamma_0^+)$ and $\alpha_k = \ftilde^k(\alpha_0)$.
    By (\ref{eqn:shifting-many}), for $k \in \{1,\ldots, \mathbf{k}\}$, 
    there exist a pair of infinite sectors $\mathcal{U}_{-k}$ and $\mathcal{U}_{k}$ such that 
    \begin{itemize}
        \item $\mathcal{U}_{-k}$ has summit $\alpha_{-k}$ and sides $\gamma_k^-$ and $\gamma_{k+1}^-$, 
        \item $\mathcal{U}_{k}$ has summit $\alpha_{k}$ and sides $\gamma_k^+$ and $\gamma_{k+1}^+$,
        \item the maps $\ftilde^k: \mathcal{U}_{-k} \to \mathcal{U}_0$ and $\ftilde^k :\mathcal{U}_0 \to \mathcal{U}_k$ are conformal isomorphisms.
    \end{itemize}
    Denote
    \[
        C_1 := C' + C_0 + 0.1K_0 + \frac{1}{22} 
        \qquad \text{and} \qquad
        \mathcal{U} := 
        \bigcup_{j=-\mathbf{k}}^{\mathbf{k}} \mathcal{U}_j.
    \]
    The choice of $\mathbf{k}$ ensures that
    \begin{align*}
        \bigcup_{k=-\mathbf{k}}^{\mathbf{k}} \alpha_k 
        &\subset 
        \left\{ \zeta \in \C \: : \: C_0 \leq \imag (\zeta;\theta) \leq C_1 \right\}, \\    \gamma_{\mathbf{k}}^- 
        &\subset \left\{ \zeta \in \C \: : \: 
        \real(\zeta) \leq -\frac{6\theta}{11}, \;
        \imag (\zeta;\theta) \geq C_0 \right\}, \\
        \gamma_{\mathbf{k}}^+
        &\subset \left\{ \zeta \in \C \: : \: 
        \real(\zeta) \geq \frac{6\theta}{11}, \;
        \imag (\zeta;\theta) \geq C_0 \right\}.
    \end{align*}
    In particular, $\mathcal{U}$ contains the closed half-strip
    \[
        \Pi := \left\{ \zeta \in \C \: : \: 
        |\real(\zeta)| \leq \frac{6\theta}{11}, \, \imag (\zeta;\theta) \geq C_1 \right\}.
    \]

    Consider the domain $\mathcal{V}$. 
    The removal of $\Pi$ splits $\mathcal{V}$ into two components, $\mathcal{V}^-$ on the left and $\mathcal{V}^+$ on the right.
    From (\ref{eqn:f-in-log-coord}), we have for all $\zeta \in \mathcal{V}$,
    \begin{align*}
        \real(\ftilde(\zeta)-\zeta) &\geq \frac{10}{11}\theta, \\
        \imag (\ftilde(\zeta)) + 0.1 \, \real (\ftilde(\zeta)) &\geq \imag (\zeta) + 0.1 \, \real (\zeta), \\
        \imag (\ftilde(\zeta)) - 0.1 \, \real (\ftilde(\zeta)) &\leq \imag (\zeta) - 0.1 \, \real(\zeta).
    \end{align*}
    These inequalities ensure $\ftilde$-invariance in the following sense:
    \[
        \ftilde (\mathcal{V}_-) \subset \mathcal{V}_- \cup \Pi \qquad \text{and} \qquad \mathcal{V}_+ \subset \ftilde(\mathcal{V}_+ \cup \Pi).
    \]
    
    Since $\Pi$ is contained in $\mathcal{U}$, every point $\zeta$ in $\mathcal{V} \backslash \mathcal{U}_0$ admits a unique non-zero integer $n_\zeta$ with the smallest $|n_\zeta|$ such that $\ftilde^{n_\zeta}(\zeta)$ is contained in $\mathcal{U}_0 \subset \mathcal{S}$. 
    If $\zeta$ is in $\mathcal{S}$, then we will just set $n_\zeta = 0$.
    So for $\zeta \in \mathcal{V}$, we set
    \[
        \phi(\zeta) := \phi(\ftilde^{n_{\zeta}}(\zeta)) - n_{\zeta}.
    \]
    The extension defined above is clearly holomorphic and continuous in $\mathcal{V} \cup \mathcal{S}$. 
    
    Let us show that $\phi$ is univalent on $\mathcal{V} \cup \mathcal{S}$. 
    Suppose $\phi(\zeta) = \phi(\zeta')$ for some distinct points $\zeta$ and $\zeta'$ in $\mathcal{V} \cup \mathcal{S}$. 
    Then,
    \[
        \phi(\ftilde^{n_{\zeta}}(\zeta))-\phi(\ftilde^{n_{\zeta'}}(\zeta')) = n_{\zeta}-n_{\zeta'}.
    \]
    Since $\phi(\ftilde^{n_{\zeta}}(\zeta))-\phi(\ftilde^{n_{\zeta'}}(\zeta'))$ is an integer and $\ftilde^t(\zeta) \neq \ftilde^{t'} (\zeta')$, then $\zeta$ and $\zeta'$ can only be on the boundary of the sector $\mathcal{S}$, $|n_{\zeta}-n_{\zeta'}|=1$, and either $\ftilde(\zeta)=\zeta'$ or $\zeta = \ftilde(\zeta')$.
\end{proof}

Let us very briefly reintroduce the index $n$; what truly matters in the next few pages is the base level $n=0$ and the next level $n=1$.
For $n \geq 0$, consider the Mother Hedgehog 
\[
\mathcal{H}_n := \exp^{-1}(H_n)
\]
of $\ftilde_n$.
Recall that $\mathcal{S}_n$ contains the critical value $\vtilde_0^n = \vtilde_0(\ftilde_n)$ of $\ftilde_n$.
For $j,k\in \Z$, denote
\[
    \vtilde_{j,k}^n := T^k \circ (\ftilde_n|_{\mathcal{H}_n})^j(\vtilde_0^n)
\]
and let $\mathcal{I}^n_{j,k}$ be the internal ray of $\mathcal{H}_n$ landing at $\vtilde_{j,k}^n$.

In the dynamical plane of $\ftilde$,
we denote by $\mathcal{U}$ the infinite sector whose sides are the internal rays $\mathcal{I}^0_{-1,0}$ and $\mathcal{I}^0_{-1,1}$ and whose summit is the arc along $\partial \hat{\mathcal{Z}}^{[0]}$ joining the critical points $\vtilde^0_{-1,0}$ and $\vtilde^0_{-1,1}$.
In the dynamical plane of $\ftilde_1$, let us consider the infinite sector 
\[
\mathcal{E} = \Sigma_1 \cap \hat{\mathcal{Z}}^{[-1]}_1.
\]
The internal ray $\mathcal{I}_{0,0}^1$ splits $\mathcal{E}$ into two infinite sectors, namely $E^-$ on the left and $E^+$ on the right.
Define 
\[
    \tilde{\ftilde}_1(z) := \begin{cases}
        \ftilde_1(z) \quad & \text{ if } \varepsilon_2 = -1, \\
        \ftilde_1\circ T^{-1}(z) \quad & \text{ if } \varepsilon_2 = +1.
    \end{cases}
\]

\begin{lemma}
\label{lem:phi-PS-extension}
    The map $\phi$ uniquely extends to a univalent map from $\mathcal{U}$ onto the infinite sector
\[
    \mathcal{Y} = T^{-1}(E^+) \cup E^- \cup E^+ \cup \bigcup_{j=1}^{\bar{a}_{1}-1} \Big( T^{j-1} \tilde{\ftilde}_{1}^{-1}(E^+) \cup T^j \tilde{\ftilde}_{1}^{-1}(E^-) \Big)
\]
    with sides $\mathcal{I}^1_{0,-1}$ and $\mathcal{I}^1_{-1,b_1}$.
\end{lemma}

See Figure \ref{fig:siegel-extension}.
Note that the summit of $\mathcal{Y}$ is very close to the interval on the boundary of $\hat{\mathcal{Z}}^{[-1]}_1$ with endpoints $\vtilde^1_{0,-1}$ and $\vtilde^1_{-1,b_1}$.

\begin{figure}
    \centering
    
    \begin{tikzpicture}
    \node[anchor=south west, inner sep=0] (image) at (0,0) {\includegraphics[width=0.99\linewidth]{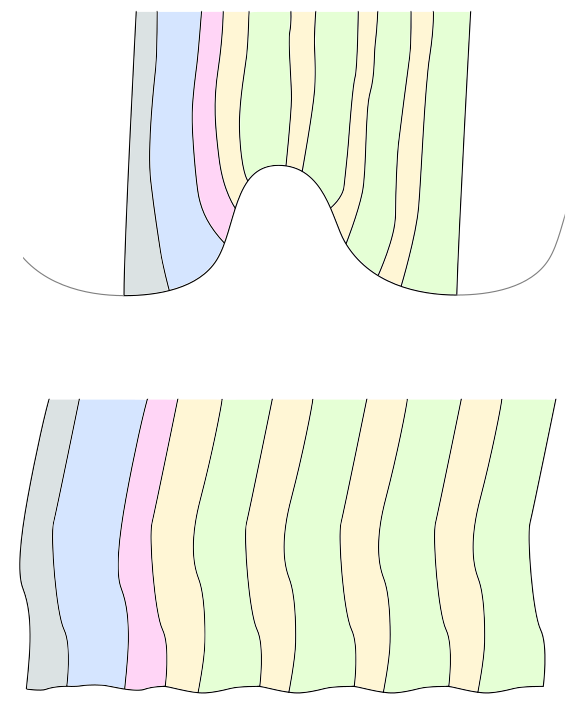}};
    \begin{scope}[
        x={(image.south east)},
        y={(image.north west)}
    ]   
        \node [black, font=\bfseries] at (0.395,0.663) {\small $\bullet$};
        \node [black, font=\bfseries] at (0.423,0.663) {\small $\vtilde^0_{0}$};
        \node [black, font=\bfseries] at (0.22,0.587) {\small $\bullet$};
        \node [black, font=\bfseries] at (0.22,0.565) {\small $\vtilde^0_{-1,0}$};
        \node [black, font=\bfseries] at (0.804,0.587) {\small $\bullet$};
        \node [black, font=\bfseries] at (0.804,0.565) {\small $\vtilde^0_{-1,1}$};
        
        \node [black, font=\bfseries] at (0.22,0.04) {\small $\bullet$};
        \node [black, font=\bfseries] at (0.1975,0.065) {\small $\vtilde^1_{0}$};
        \node [black, font=\bfseries] at (0.048,0.04) {\small $\bullet$};
        \node [black, font=\bfseries] at (0.048,0.015) {\small $\vtilde^1_{0,-1}$};
        \node [black, font=\bfseries] at (0.955,0.042) {\small $\bullet$};
        \node [black, font=\bfseries] at (0.96,0.017) {\small $\vtilde^1_{-1,b_1}$};
        
        \node [black, font=\bfseries] at (0.52,0.53) {$\phi$};
        \draw[black,line width=0.75pt,-latex] (0.5,0.6) -- (0.5,0.46);

        \node [black!30!blue, font=\bfseries] at (0.3,0.85) {\small $D^-$};
        \node [purple, font=\bfseries] at (0.36,0.85) {\small $D^+$};
        \node [black, font=\bfseries] at (0.2,0.92) {$\mathcal{U}$};
        \node [black, font=\bfseries] at (0.045,0.4) {$\mathcal{Y}$};
        
        \node [black, font=\bfseries] at (0.32,1.025) {\small $\ftilde$};
        \draw[black,line width=0.5pt,-latex] (0.26,0.99) .. controls (0.28,1.01) and (0.345,1.01) .. (0.365,0.99);
        \node [red!50!black, font=\bfseries] at (0.402,1.035) {\small $T^{-1}\ftilde^{\bar{a}_1}$};
        \draw[red!50!black,line width=0.5pt,-latex] (0.42,0.99) .. controls (0.41,1.02) and (0.39,1.02) .. (0.38,0.99);
        \node [black, font=\bfseries] at (0.485,1.025) {\small $\ftilde$};
        \draw[black,line width=0.5pt,-latex] (0.435,0.99) .. controls (0.455,1.01) and (0.525,1.01) .. (0.545,0.99);
        \node [black, font=\bfseries] at (0.595,1.025) {\small $\ftilde$};
        \draw[black,line width=0.5pt,-latex] (0.555,0.99) .. controls (0.575,1.01) and (0.635,1.01) .. (0.655,0.99);
        \node [black, font=\bfseries] at (0.715,1.025) {\small $\ftilde$};
        \draw[black,line width=0.5pt,-latex] (0.665,0.99) .. controls (0.685,1.01) and (0.745,1.01) .. (0.765,0.99);

        \node [gray!50!black, font=\bfseries] at (0.95,0.59) {\small $\partial \hat{\mathcal{Z}}^{[0]}$};
        \node [black!30!blue, font=\bfseries] at (0.15,0.25) {\small $E^-$};
        \node [purple, font=\bfseries] at (0.24,0.25) {\small $E^+$};

        \node [black, font=\bfseries] at (0.16,-0.015) {\small $T$};
        \draw[black,line width=0.5pt,-latex] (0.085,0.02) .. controls (0.115,0) and (0.205,0) .. (0.235,0.02);
        \node [red!50!black, font=\bfseries] at (0.29,-0.015) {\small $\tilde{\ftilde}_1$};
        \draw[red!50!black,line width=0.5pt,-latex] (0.32,0.02) .. controls (0.305,0) and (0.27,0) .. (0.26,0.02);
        \node [black, font=\bfseries] at (0.435,-0.015) {\small $T$};
        \draw[black,line width=0.5pt,-latex] (0.36,0.02) .. controls (0.39,0) and (0.48,0) .. (0.51,0.02);
        \node [black, font=\bfseries] at (0.595,-0.015) {\small $T$};
        \draw[black,line width=0.5pt,-latex] (0.52,0.02) .. controls (0.55,0) and (0.64,0) .. (0.67,0.02);
        \node [black, font=\bfseries] at (0.765,-0.015) {\small $T$};
        \draw[black,line width=0.5pt,-latex] (0.69,0.02) .. controls (0.72,0) and (0.81,0) .. (0.84,0.02);
    \end{scope}
\end{tikzpicture}

    \caption{The renormalization change of variables $\phi: \mathcal{S} \to \Sigma_1$ extends to a univalent map $\phi: \mathcal{U} \to \mathcal{Y}$. In the dynamical plane of $\ftilde$, part of $\mathcal{S}$ in $\hat{\mathcal{Z}}^{[0]}$ is split into two sub-sectors $D^-$ and $D^+$ and is spread around to form $\mathcal{U}$. In the dynamical plane of $\ftilde_1$, part of $\Sigma_1$ in $\hat{\mathcal{Z}}^{[-1]}_1$ is split into two sub-sectors $E^-$ and $E^+$ and is spread around to form $\mathcal{Y}$.}
    \label{fig:siegel-extension}
\end{figure}
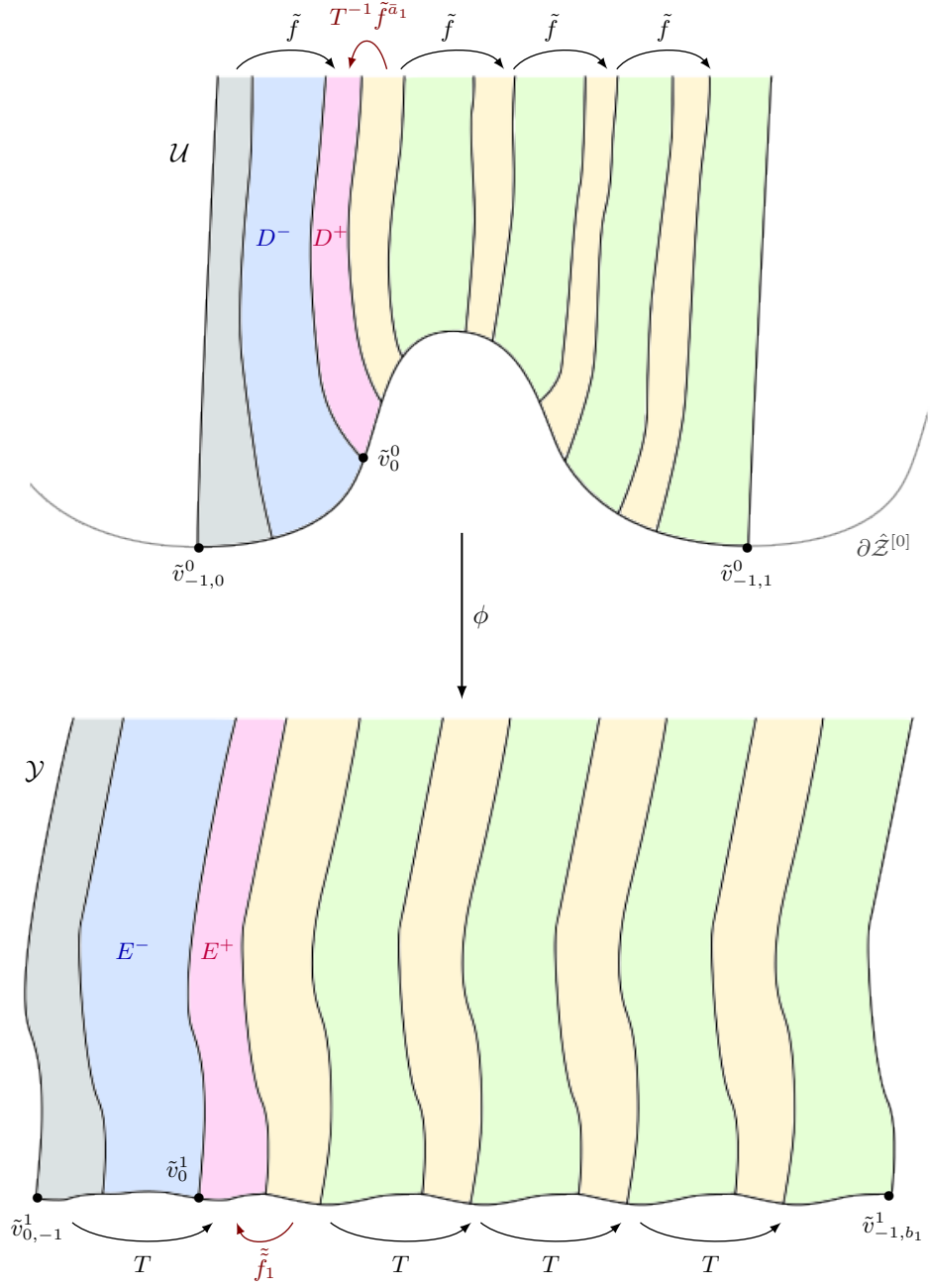

\begin{proof}
    Let $D^-$ be the infinite sector with sides $\mathcal{I}^0_{-\bar{a}_1-1,1}$ and $\mathcal{I}^0_{0,0}$ and summit being a subarc of $\partial E$, and let $D^+$ be the infinite sector with sides $\mathcal{I}^0_{0,0}$ and $\mathcal{I}^0_{-\bar{a}_1,1}$ and summit being a subarc of $\partial E$.
    Then, $\mathcal{U}$ can be partitioned into
    \[
        \ftilde^{-1}(D^+), \; D^-, \; D^+, \; T\ftilde^{-\bar{a}_1}(D^+), \;  T\ftilde^{-\bar{a}_1+1}(D^-), \ldots, \;
        \; T\ftilde^{-2}(D^+),
        \; T\ftilde^{-1}(D^-),
    \]
    written from left to right.
    Any two elements in the partition above are either disjoint or intersect along a single common side.
    This partition is essentially a particular lift under $\exp$ of the triangulation $\hat{\Delta}_0^{[0]}$ described in \S\ref{ss:renorm-tiling}.

    Similarly, observe that 
    \[
        T^{-1}(E^+), \; E^-, \; E^+, \;
        \tilde{\ftilde}_1^{-1}(E^+), \; T \tilde{\ftilde}_1^{-1}(E_-), \;
        \ldots, \;
        T^{\bar{a}_1-2} \tilde{\ftilde}_1^{-1}(E^+), \; 
        T^{\bar{a}_1-1} \tilde{\ftilde}_1^{-1}(E_-)
    \]
    form a collection of infinite sectors in the dynamical plane of $\ftilde_1$ with the property that
    the sides of every sector is an internal ray of $\ftilde_1$, and that
    any two sectors are either disjoint or intersect along a common side.
    The union of these sectors is $\mathcal{Y}$.

    The map $\phi$ is well-defined on $D^-$ and $D^+$ because they are contained in $\mathcal{S}$.
    Since $\phi$ sends each $\hat{\mathcal{Z}}^{[m]} \cap \mathcal{S}$, and $\mathcal{I}_{0,0}^0$ onto $\hat{\mathcal{Z}}^{[m-1]} \cap \Sigma$ and $\mathcal{I}_{0,0}^1$ respectively, then $\phi(D^-) = E^-$ and $\phi(D^+) = E^+$.

    On $\ftilde^{-1}(D^+)$, we define $\phi$ by
    \[
        \phi(z) := T^{-1}\phi(\ftilde(z)),
    \]
    By equivariance of $\ftilde$ along the sides of $\mathcal{S}$, $\phi$ must be continuous on the common side $\mathcal{I}^0_{-\bar{a}_1-1,1} = \ftilde^{-1}(D^+) \cap D^-$.
    The image of $\ftilde^{-1}(D^+)$ under $\phi$ is equal to $T^{-1}(E^+)$.
    
    On $T\ftilde^{-\bar{a}_1}(D^+)$, we define $\phi$ by
    \[
        \phi(z) = \tilde{\ftilde}_1^{-1} \circ \phi \circ T^{-1}\ftilde^{\bar{a}_1}(z).
    \]
    This is again well-defined and continuous on the common side $\mathcal{I}^0_{-\bar{a}_1,1} = D^+ \cap T \ftilde^{-\bar{a}_1}(D^+)$ because $\phi$ projects $T^{-1} \ftilde^{b_1}$ to the map $\ftilde_1$.
    More generally, for $j \in \{1,\ldots,\bar{a}_1-1\}$, let us set
    \[
        \phi(z) = 
            T^j \tilde{\ftilde}_1^{-1} \circ \phi \circ T^{-1}\ftilde^{\bar{a}_1+j}(z), \qquad \text{ for } z \in T\ftilde^{-\bar{a}_1+j}(D^- \cup D^+).
    \]
    For the same equivariance reasons, $\phi$ is a well-defined univalent map map on $T\ftilde^{-\bar{a}_1+j}(D^\pm)$, the image is equal to $T^j \tilde{\ftilde}_1^{-1}(E^\pm)$, and altogether $\phi$ glues to a continuous univalent map from $\mathcal{U}$ to $\mathcal{Y}$.
\end{proof}

\subsection{The estimate on $\phi'$} 

In order to estimate the renormalization change of variables $\phi$, we will adapt an approach inspired by \cite[Chapter 1 \S3.4]{Yoc95}.

\begin{lemma}
\label{lem:de-branges}
    For any univalent function $g: (\D,0) \to (\C,0)$ and $z \in \D^*$, we have
    \[
        \left| g(z) - g(-z) - 2z g'(0) \right| \leq 2 |g'(0)| \frac{|z|^3(3-|z|^2)}{(1-|z|^2)^2}.
    \]
\end{lemma}

\begin{proof}
    By postcomposing with the linear map $z \mapsto g'(0) z$, assume without loss of generality that $g'(0)=1$. By de Branges' Theorem, $g$ has power series expansion $g(z) = z + \sum_{n\geq 2} a_n z^n$ where $|a_n| \leq n$ for all $n \geq 2$. Therefore, for all $z \in \D^*$,
    \begin{align*}
        |g(z)-g(-z)-2z| &\leq 2\left| \sum_{n \geq 1} a_{2n+1} z^{2n+1} \right| \\
        &\leq 2 \sum_{n\geq 1} (2n+1) |z|^{2n+1} = 2 |z|^3 \frac{3-|z|^2}{(1-|z|^2)^2}. \qedhere
    \end{align*}
\end{proof}

\begin{lemma}
\label{lem:yoc-3}
    There exist uniform constants $C_2, C_3 >0$ such that for every point $\zeta$ with $|\real(\zeta)| \leq K_0$ and $\imag(\zeta;\theta) \geq C_2$, we have
    \[
        | \theta \phi'(\zeta) - 1 | \leq 
        \frac{C_3}{\imag(\zeta;\theta)^2}.
    \]
\end{lemma}

\begin{proof}
    Let $C_1 >0$ be the constant from Lemma \ref{lem:univalent-extension}.
    Fix a point $\zeta$ in $\mathcal{V} \cap \mathcal{S}$ such that the quantity
    \[
        y := \imag (\zeta;\theta) - C_1 - 0.1 K_0 \geq 1
    \]
    is sufficiently large. 
    Since $|\real (\zeta)| \leq K_0$, then by straightforward computation, the distance $r$ between $\zeta$ and the boundary of $\mathcal{V}$ satisfies
    \begin{equation}
    \label{eqn:r-and-y}
        r \geq \cos(\tan^{-1} 0.1) \, y 
        \geq 0.9 \, y.
    \end{equation}
    By Lemma \ref{lem:univalent-extension}, $\phi$ is univalent on the round disk $\D(\zeta,r)$.

    To estimate the derivative $\phi'(\zeta)$, we will be applying Lemma \ref{lem:de-branges}. But to do so, we need a point $\zeta'$ such that 
    \begin{equation}
    \label{eqn:average}
        \zeta = \frac{\zeta' + \ftilde(\zeta')}{2}
    \end{equation}
    and that the distance between $\zeta$ and $\zeta'$ is close to $\theta/2$. We claim that such a point exists and is unique. Indeed, consider the function $h(w) = \frac{w + \ftilde(w)}{2}$ where $w$ is any point on the half plane 
    \[
        \left\{ w \in \C \: : \: \imag (w;\theta) \geq  C_0 \right\}.
    \]
    For such a point $w$, we have
    \[
        | h'(w) - 1 | \leq \frac{\theta}{22}
        \quad \text{ and } \quad
        \imag \left(h(w)-w\right) \leq \frac{\theta}{22}.
    \]
    Therefore, $h$ is injective and its image contains the half-plane
    \[
        \left\{ w \in \C \: : \: \imag (w;\theta) \geq C_1 + \frac{\theta}{22} \right\}
    \]
    and in particular the point $\zeta$. Therefore, the desired point $\zeta'$ exists and is unique. 
    
    By Proposition \ref{prop:almost-translation}, there exists some uniform constant $C_4>0$ such that
    \begin{equation}
    \label{eqn:zeta-and-zeta-and-theta}
        \left|\zeta' - \zeta - \frac{\theta}{2} \right| \leq C_4 e^{-2\pi y} \theta.
    \end{equation}
    Assume $y$ is sufficiently high (e.g. $y \geq \frac{1}{2\pi}\log(10 C_4)$) so that the right hand side of the inequality above is bounded above by $\theta/10$. Thus,
    \begin{equation}
    \label{eqn:zeta-and-zeta}
        \left|\zeta' - \zeta\right| \leq \frac{3\theta}{5}.
    \end{equation}
    Also, by (\ref{eqn:r-and-y}) and the inequality $0<\theta\leq\frac{1}{2}$, the point $w_*:= \frac{\zeta' - \zeta}{r}$ satisfies
    \begin{equation}
    \label{eqn:w-star}
        \left| w_* \right| \leq \frac{1}{3y} \leq \frac{1}{3}.
    \end{equation}

    Consider the function $g(w) = \phi(\zeta + w r)$, which is univalent on $\D$.  By applying Lemma \ref{lem:de-branges} to the point $w=w_*$, we have
    \[
        \big| \phi(\ftilde(\zeta')) - \phi(\zeta') - 2 r w_* \phi'(\zeta) \big|
        \leq 2 |r w_*| \cdot |w_*|^2 \frac{ 3-|w_*|^2 }{(1-|w_*|^2)^2} \cdot|\phi'(\zeta)|.
    \]
    Recall that $rw_* = \zeta'-\zeta$. By the functional equation for $\phi$ and inequalities (\ref{eqn:zeta-and-zeta}) and (\ref{eqn:w-star}),
    \begin{align*}
        \big| 1- 2 (\zeta'-\zeta) \phi'(\zeta) \big|
        &\leq 2\cdot \frac{3\theta}{5} \cdot \frac{1}{9y^2} \cdot \frac{117}{32} \cdot |\phi'(\zeta)| \\
        &\leq \frac{1}{2y^2} |\theta \phi'(\zeta)|.
    \end{align*}
    Then, by (\ref{eqn:zeta-and-zeta-and-theta}) and the general inequality $e^x \geq x^2$, we have
    \begin{align*}
        \big| 1-\theta \phi'(\zeta) \big| 
        &\leq 2 C_4 e^{-2\pi y} |\theta \phi'(\zeta)| + \frac{1}{2y^2} |\theta \phi'(\zeta)| \\
        &\leq \frac{C_5}{y^2} |\theta \phi'(\zeta)|,
    \end{align*}
    where $C_5 = \frac{C_4}{2 \pi^2} + \frac{1}{2}$. By the triangle inequality, we have
    \[
        \big| 1-\theta \phi'(\zeta) \big| \leq \frac{C_5}{y^2 - C_5}.
    \]
    Then, assuming $y$ is sufficiently high, we have 
    \[
        |1-\theta \phi'(\zeta)| \leq \frac{2 C_5}{ \imag(\zeta;\theta)^2}. \qedhere
    \]
\end{proof}

\subsection{The estimate on $\phi$}

Consider the constants $C_2, C_3>0$ from Lemma \ref{lem:yoc-3}.
We will need the following lemma to prove Proposition \ref{prop:renormalization-change-of-variables}.

\begin{lemma}
\label{lem:new-estimate}
    There exists a point $\zeta_* \in \mathcal{S}$ such that
    \[
        0 \leq \imag (\zeta_*;\theta) - C_2 = O(1) \quad \text{ and } \quad
        \imag \,\phi(\zeta_*) \asymp \frac{1}{\theta}.
    \]
\end{lemma}

\begin{proof}
    By compactness considerations, it suffices to prove this lemma in the near-parabolic situation when $\theta \ll 1$.
    Recall the infinite sectors $\mathcal{U}$ and $\mathcal{Y}$ from Lemma \ref{lem:phi-PS-extension}.

    Recall that the summit of $\mathcal{U}$ is a fundamental interval of $\hat{\mathcal{Z}}^{[0]}$.
    By Theorem \ref{thm:uniform-fjords}, there exists a uniform constant $K_1>0$ such that
    \[
        \left| \max_{\zeta \in \partial \hat{\mathcal{Z}}^{[0]}} \imag(\zeta;\theta) \right| \leq K_1.
    \]
    Fix a constant $K_2 \geq \max\{C_1,K_1\} + 1$.
    Let $\gamma$ be the unique connected component of $\{\imag(\zeta; \theta) = K_2\} \backslash \partial \mathcal{U}$ that connects the two sides of $\mathcal{S}$.
    Let $\mathbf{R}$ be the bounded connected component of $\text{int}\mathcal{U} \backslash \gamma$, equipped with the structure of a conformal rectangle whose horizontal sides are $\gamma$ and the summit of $\mathcal{U}$.
    By the uniform quasiconformality of the critical internal ray, we have $\text{mod}(\mathbf{R}) \asymp 1$. 

    We will now estimate the height of $\phi(\gamma)$.
    Consider the image $\phi(\mathbf{R})$, which is a conformal rectangle in the dynamical plane of $\ftilde_1$ with top horizontal side $\phi(\gamma)$. 
    The left side $\mathcal{I}_{0,-1}^1$ of $\mathcal{Y}$ is contained in the strip $\{-K_0-1 \leq \real(\zeta) \leq K_0-1\}$ whereas the right side $\mathcal{I}_{-1,b_1}^1$ of $\mathcal{Y}$ is contained in the strip $\{b_1-1-K_0 \leq \real(\zeta) \leq b_1-1+K_0\}$.
    Since $\partial \hat{\mathcal{Z}}^{[-1]}$ is contained in the strip $\{0 \leq \imag(\zeta) \leq C_0\}$, then so is the summit of $\mathcal{Y}$.
    Therefore, since the distance between the opposite vertical sides of $\mathbf{R}$ is $\asymp \frac{1}{\theta}$ and since $\text{mod}(\mathbf{R}) \asymp 1$, then the distance between the opposite horizontal sides of $\mathbf{R}$ is also $\asymp \frac{1}{\theta}$.
    In particular, there exists a point $\zeta_* \in \gamma$ such that $\imag(\phi(\zeta_*)) \asymp \frac{1}{\theta}$.
\end{proof}

\begin{proof}[Proof of Proposition \ref{prop:renormalization-change-of-variables}]
    The derivative estimate comes from Lemma \ref{lem:yoc-3}. Pick two points $\zeta$ and $\zeta'$ in $\mathcal{S}$ with $\imag(\zeta';\theta) \geq \imag(\zeta;\theta) \geq C_2$. 
    Consider a straight vertical path $\gamma_1$ from $\zeta$ to $\real(\zeta) + i\,\imag(\zeta')$ and a straight horizontal path $\gamma_2$ from $\real(\zeta) + i \,\imag(\zeta')$ to $\zeta'$. 
    Observe that $\gamma_2$ has Euclidean length at most $2 K_0$. By Lemma \ref{lem:yoc-3},
    \begin{align}
    \label{eqn:zeta1-zeta2}
        \big| (\zeta - \theta \phi(\zeta)) - (\zeta' - \theta \phi(\zeta')) \big| 
        & \leq \int_{\gamma_1 \cup \gamma_2} \frac{C_3}{\imag(w;\theta)^2} \, |dw| \nonumber \\
        & \leq \frac{C_3}{\imag(\zeta;\theta)} - \frac{C_3}{\imag(\zeta';\theta)} + \frac{2 C_3 K_0}{\imag(\zeta' ; \theta)^2}.
    \end{align}
    Let us apply this inequality to the point $\zeta = \zeta_*$ from Lemma \ref{lem:new-estimate}. We then have
    \begin{align*}
        \imag\left( \zeta' - \theta \phi(\zeta') \right) = \frac{1}{2\pi}\log\frac{1}{\theta} + O(1). 
    \end{align*}
    We always have
    \[
        \big| \real \left(\zeta' - \theta \phi(\zeta') \right)\big| \leq \frac{3}{2}K_0. 
    \]
    Hence, the value of $\zeta' - \theta \phi(\zeta')$ always lies in a compact subset of the plane. As we take $\imag(\zeta') \to \infty$ and freeze the real value of $\zeta'$, then $\zeta' - \theta \phi(\zeta')$ converges in subsequence to some limit $\kappa \in \C$ with
    \[
        \real(\kappa) = O(1)
        \qquad \text{ and } \qquad
        \imag(\kappa;\theta) = O(1).
    \]
    Then, let us unfreeze $\zeta$. By (\ref{eqn:zeta1-zeta2}), we have
    \begin{align*}
        \big| \zeta - \theta \phi(\zeta) - \kappa \big| 
        &\leq \frac{C_3}{\imag(\zeta;\theta)}.
    \end{align*}
    This last estimate implies that $\zeta - \theta \phi(\zeta)$ always converges to $\kappa$ as $\imag(\zeta) \to \infty$. In particular, the limit $\kappa$ is independent of the choice of $\zeta'$.
\end{proof}

\section{Zero area}
\label{sec:zero-area}

Let us fix $\fbold = \seq{ f_n  }_{n\geq 0} \in \overline{\towsec}$ and let $\thetabold = \thetabold(\fbold) \in \overline{\Theta}$ be the associated combinatorics.
For $n\geq 0$, denote $\thetabold_n := \shift^n(\thetabold)$ the $n$\textsuperscript{th} shift of $\thetabold$ and $\theta_n = \boldsymbol{\mathfrak{X}}(\thetabold_n) \in (-\frac{1}{2},\frac{1}{2}]$ the rotation number of $f_n$. 

In this section, we will apply the estimates from Section \ref{sec:logarithmic} to prove Theorems \ref{main-thm-02} (and thus \ref{main-thm-01}) and \ref{side-theorem}.

\subsection{Going down the renormalization tower}
\label{ss:setup}
We will again adapt the notation $\ftilde_n$, $\phi_n: \mathcal{S}_n \to \Sigma_{n+1}$, etc introduced in the previous section.
For $n \geq 0$, denote the lift of the postcritical set $P_n = P_n(\fbold)$, the Mother Hedgehog $H_n=H_n(\fbold)$, and the pseudo-Siegel pinched disks $\hat{Z}^{[m]}_n = \hat{Z}^{[m]}_n(\fbold)$, $m \geq -1$ by 
\[
\mathcal{P}_n := \exp^{-1}(P_n), \qquad
\mathcal{H}_n := \exp^{-1}(H_n), \qquad
\hat{\mathcal{Z}}^{[m]}_n := \exp^{-1} \big( \hat{Z}^{[m]}_n \big)
\]
respectively. 
Again, we have
\[
    \hat{\mathcal{Z}}^{[-1]}_n \supset \hat{\mathcal{Z}}^{[0]}_n \supset
    \hat{\mathcal{Z}}^{[1]}_n \supset \ldots \supset \bigcap_{m=-1}^\infty \hat{\mathcal{Z}}^{[m]}_n = \mathcal{H}_n 
    \qquad \text{ and } \qquad
    \partial \mathcal{H}_n = \mathcal{P}_n.
\]

Consider the time semigroup $\Time^n = \Time_{\thetabold_n}$ associated to $\thetabold_n$.
Of particular importance are the smallest element $\qq^n_{[0]} \in \Time^n$, which we will just identify with the integer $1$, and the first return time $\qq_{[1]}^n \in \Time^n$.
If $\theta_n$ is irrational, then $\qq^n_{[1]}$ is some finite integer; otherwise, there are infinitely many elements of $\Time^n$ that are less than $\qq^n_{[1]}$, namely
\[
    1<2<3<\ldots < \ldots <\qq^n_{[1]}-3 < \qq^n_{[1]}-2 < \qq^n_{[1]}-1.
\]

To include parabolic dynamics in our discussion, let us make the convention that $\log\frac{1}{0}=+\infty$ and so for any point $\zeta \in \C$, $\imag(\zeta;0) = -\infty$ is less than every real number. 

Let us fix $\boldsymbol{\epsilon}>0$ to be a sufficiently small constant. 
To be concrete, we can take it to be the solution to the equation
\begin{equation}
    \label{eqn:epsilon}
    e^{2\boldsymbol{\epsilon}} = 1.2 (1-\boldsymbol{\epsilon})^2.
\end{equation}
This choice will be used later in the proof of Proposition \ref{prop:brjuno}. 

From Propositions \ref{prop:almost-translation} and \ref{prop:renormalization-change-of-variables}, there exist absolute constants $C_{\boldsymbol{\epsilon}}>0$ and $\Kbold_0 > 0$ such that for all $n \in \N$ and $\zeta \in \C$ satisfying $\imag(\zeta;\theta_n) \geq C_{\boldsymbol{\epsilon}}$, we have
\begin{enumerate}[label=\textnormal{(\Roman*)}]
    \item \label{item-1} $\lvert \ftilde_n'(\zeta) - 1 \rvert \leq {\boldsymbol{\epsilon}} |\theta_n|$,
    \item \label{item-2} $\lvert \ftilde_n(\zeta) - \zeta - \theta_n \rvert \leq {\boldsymbol{\epsilon}} |\theta_n|$,
\end{enumerate}
and if $\zeta$ is in $\mathcal{S}_n$ too,
\begin{enumerate}[label=\textnormal{(\Roman*)}]
  \setcounter{enumi}{2}
    \item \label{item-3} $\Big| |\theta_n| \phi_n'(\zeta) - 1 \Big| \leq {\boldsymbol{\epsilon}}$,
    \item \label{item-4} $\Big| |\theta_n| \phi_n(\zeta) - \zeta + \kappa_n \Big| \leq {\boldsymbol{\epsilon}}$ where $\kappa_n \in \C$ satisfies $|\kappa_n - \frac{i}{2\pi} \log\frac{1}{|\theta_n|}| \leq \Kbold_0$.
\end{enumerate}
Set
\[
    \boldsymbol{C} := C_{\boldsymbol{\epsilon}} + {\boldsymbol{\epsilon}}.
\]

\begin{lemma}[Landing points and landing times]
\label{ref:landing-points-and-times}
    For every pair of natural numbers $m, n \geq 0$ and every point $\zeta$ in $\hat{\mathcal{Z}}^{[m]}_n$, 
    there exist a sequence of points $\{ \zeta_{k} \}_{0 \leq k \leq m+1}$ in $\UHP$ and times $\{t_k \in \Time^{n+k}  \cup \{0\} \}_{0 \leq k \leq m}$ with the following properties:
    \begin{enumerate}
        \item $\zeta_0 = \zeta$;
        \item for $k \in \{0,1,\ldots, m\}$, $t=t_k$ is the smallest element of $\Time^{n+k} \cup \{0\}$ such that $\ftilde_{n+k}^t(\zeta_k)$ is contained in $\mathcal{S}_{n+k} + \Z$;
        \item for $k \in \{0,1,\ldots, m\}$, $\zeta_{k+1} := \phi_{k} \big( \ftilde_{n+k}^{t_k}(\zeta_k) \big)$. 
    \end{enumerate}
\end{lemma}

These $\zeta_{k}$'s will be called the \emph{landing points} of $\zeta$, and the $t_k$'s will be called the \emph{landing times} of $\zeta$.
From the first return dynamics, if $\theta_n \neq 0$ or equivalently $\bar{a}_{n+1} < \infty$, then the landing time $t_n$ is always bounded above by $\bar{a}_{n+1}$.

\begin{proof}
    Let $z_0 = \exp(\zeta)$, a point in $\hat{Z}_n^{[m]}$.
    According to Lemma \ref{lem:triangulation-first-return}, for every $k \in \{0,\ldots,m\}$, there exists some first time $t_k \in \Time^{n+k} \cup \{0\}$ such that $t=t_k$ is the first time $f_{n+k}^t(z_k)$ is in $S_{n+k,1} \cap \hat{Z}_{n+k}^{[m-k]}$, then $z_{k+1} = \psi_{n+k,1}(f_{n+k}^t(z_k))$ is set.
    Then, we lift $z_1,\ldots, z_{m+1}$ under $\exp$ to obtain the desired corresponding points $\zeta_1, \ldots, \zeta_{m+1}$.
\end{proof}

For $m,n \geq 0$, define the set
\[
    A^m_n := \left\{ \zeta \in \hat{\mathcal{Z}}^{[m]}_n 
    \: : \: 
    \begin{array}{c}
         \text{ the landing points of } \zeta \text{ satisfy}\\
         \imag(\zeta_k;\theta_{n+k}) \geq \boldsymbol{C} \text{ for all } k \in \{0,\ldots,m\}
    \end{array}
    \right\}.
\]
By Proposition \ref{prop:almost-translation}, $\zeta$ being in $A^m_n$ implies that
\[
    \imag( \ftilde_{n+k}^t(\zeta_k); \theta_{n+k}) \geq C_{\boldsymbol{\epsilon}} \quad \text{ for all } k= 0,1,\ldots,m \text{ and } t = 0,1,\ldots, \qq^{n+k}_{[1]},
\]
and inequalities \ref{item-1}--\ref{item-4} become applicable to each $\ftilde_{n+k}^t(\zeta_k)$.

By design, we have the following elementary property.

\begin{lemma}
\label{lem:empty-An}
    For all $m,n\geq 0$, the set $A^m_n$ is empty if and only if $\theta_{n+k} = 0$ for some integer $k$ with $0 \leq k \leq m$. 
\end{lemma}

In the next subsection, we will study the nested intersection
\[
    A^\infty_n := \bigcap_{n\geq 0} A^m_n \subset \mathcal{H}_n.
\]

\subsection{The Brjuno condition}

For $m,n \geq 0$, define the finite sum
\[
    \brjuno_m(\theta_n) := \log\frac{1}{|\theta_n|} + \sum_{j=1}^m |\theta_{n}\theta_{n+1} \ldots \theta_{n+j-1}| \log\frac{1}{|\theta_{n+j}|}.
\]
Each $\brjuno_m(\theta_n)$ is the $m$\textsuperscript{th} partial sum of the (standard) Brjuno function
    \[
           \brjuno(\theta_n) := \log\frac{1}{|\theta_{n}|} 
           + |\theta_{n}|\log\frac{1}{|\theta_{n+1}|} + |\theta_{n} \theta_{n+1}| \log\frac{1}{|\theta_{n+2}|} + \ldots.
    \]
It satisfies the equation
\begin{align}
\label{eqn:yoccoz-relation}
    \brjuno_{m+1}(\theta_n) = \log \frac{1}{|\theta_n|} + |\theta_n| \, \brjuno_{m} (\theta_{n+1}) \qquad \text{for } m, n \geq 0.
\end{align} 

\begin{definition}
    We say that $\fbold$ is \emph{eventually Brjuno} if there exists some $N \geq 0$ such that $\theta_N$ is a Brjuno irrational, that is, $\brjuno(\theta_N) < \infty$. If $N=0$, we simply say that $\fbold$ is \emph{Brjuno.}
\end{definition}

If $\fbold$ is eventually Brjuno but not Brjuno, then there exists a finite number of indices $n \geq 0$ such that $\theta_n$ is rational. 
The Brjuno function $\brjuno(\cdot)$ is comparable to the original Brjuno sum $\Brjuno(\cdot)$ from (\ref{eqn:brjuno-original}):
\[
|\brjuno(\cdot) - \Brjuno(\cdot)| = O(1).
\]
This follows from elementary properties of continued fractions; 
see \cite[Chapter 1 \S1.5]{Yoc95} for more details.

The next lemma states that we can use $\brjuno_{m}(\theta_n)$ to estimate the set $A^m_n$. 
Set
\[
    \Kbold_1 := 2( \Kbold_0 +{\boldsymbol{\epsilon}}) \qquad \text{and} \qquad \Kbold_2 := \max\{ \boldsymbol{C} ,\Kbold_1\}.
\]

\begin{lemma}
\label{lem:main-estimate}
    For all integers $m,n\geq 0$, if $\theta_k \neq 0$ for every integer $k$ with $n \leq k \leq n+m$, we have
    \begin{align*}
        \left\{ \zeta \: : \: 
        \imag(\zeta) \geq \frac{1}{2\pi}\brjuno_m(\theta_n) + \Kbold_2 \right\}
        \subset A^m_n \subset 
        \left\{ \zeta \: : \: 
        \imag(\zeta) \geq \frac{1}{2\pi}\brjuno_m(\theta_n) - \Kbold_1 \right\}.
    \end{align*}
\end{lemma}

\begin{proof}
    We will prove this for $n=0$; the general case is analogous.
    Assume that $\theta_k \neq 0$ for $k \in \{0,1,\ldots,m\}$.

    Let us prove the first inclusion. 
    Pick a point $\zeta$ such that $\imag(\zeta) \geq \frac{1}{2 \pi} \brjuno_m(\theta_0) + \Kbold_2$. 
    We claim that for all $k \in \{0,1,\ldots,m\}$, 
    \begin{enumerate}
        \item[(a:$k$)] the $k$\textsuperscript{th} landing point $\zeta_k$ of $\zeta$ is a well-defined point in $\hat{\mathcal{Z}}_k^{[0]}$;
        \item[(b:$k$)] $\imag(\zeta_k) \geq \frac{1}{2\pi} \brjuno_{m-k}(\theta_k) + \Kbold_2$.
    \end{enumerate}
    The case $k=0$ comes from the assumption on $\zeta$.
    Suppose (a:$k$) and (b:$k$) hold for some $k$.
    The inequality (b:$k$) implies $\imag(\zeta_k; \theta_k) \geq C$.
    In particular, $\zeta_k$ has to lie in the set $\hat{\mathcal{Z}}_k^{[0]}$ and so the first landing point $\zeta_{k+1}$ associated to $\zeta_k$ is well-defined.
    By \ref{item-2} and \ref{item-4}, we have
    \begin{equation}
    \label{eqn:zeta-induction}
        \left| \imag ( \zeta_k ) - |\theta_k| \, \imag(\zeta_{k+1}) - \frac{1}{2\pi} \log \frac{1}{|\theta_k|} \right| \leq \frac{\Kbold_1}{2}.
    \end{equation}
    Together with (b:$k$) and (\ref{eqn:yoccoz-relation}), we have
    \begin{align*}
        \imag(\zeta_{k+1}) 
        & \geq \frac{1}{|\theta_k|} \imag (\zeta_k) - \frac{1}{2\pi |\theta_k|} \log\frac{1}{|\theta_k|} - \frac{\Kbold_1}{2 |\theta_k|} \\
        & \geq \frac{1}{2\pi} \left(
        \frac{1}{|\theta_k|} \brjuno_{m-k}(\theta_k) - \frac{1}{|\theta_k|} \log\frac{1}{|\theta_k|} 
        \right) + \frac{\Kbold_2}{|\theta_k|} - \frac{\Kbold_1}{2 |\theta_k|} \\
        & \geq \frac{1}{2\pi} \brjuno_{m-k-1}(\theta_{k+1}) + \Kbold_2.
    \end{align*}
    This implies (a:$k+1$) and (b:$k+1$) and so our claim follows from induction.
    The statements (a:$k$) for $k \in \{0,1,\ldots,m\}$ imply that $\zeta \in \hat{\mathcal{Z}}^{[m]}_0$, whereas (b:$k$) for $k \in \{0,1,\ldots,m\}$ imply that $\imag(\zeta_k ; \theta_k) \geq C$.
    Hence, $\zeta$ is indeed contained in $A^m_0$.
    
    Next, let us assume $\zeta$ is in $A^m_0$, which means that the landing points $\zeta_0,\ldots,\zeta_{m}$ are well-defined and $\imag(\zeta_k; \theta_k) \geq C$ for all $k \in \{0,\ldots,m\}$. 
    In particular, the inequality (\ref{eqn:zeta-induction}) holds.
    Consider the inequality
    \begin{equation}
        \imag(\zeta_k) \geq \frac{1}{2\pi} \brjuno_{m-k}(\theta_k) - \Kbold_1 \tag{$\spadesuit_k$}
    \end{equation}
    for some $k \in \{0,\ldots,m\}$. If ($\spadesuit_k$) holds for some $1\leq k \leq m$, then by (\ref{eqn:zeta-induction}) and (\ref{eqn:yoccoz-relation}),
    \begin{align*}
        \imag(\zeta_{k-1}) 
        & \geq |\theta_{k-1}| \, \imag(\zeta_{k}) + \frac{1}{2\pi} \log \frac{1}{|\theta_{k-1}|} - \frac{\Kbold_1}{2}\\
        & \geq \frac{1}{2\pi} \left( |\theta_{k-1}| \brjuno_{m-k}(\theta_k) + \log \frac{1}{|\theta_{k-1}|} \right) - \Kbold_1 |\theta_{k-1}| - \frac{\Kbold_1}{2} \\
        & \geq \frac{1}{2\pi} \brjuno_{m-k+1}(\theta_{k-1}) - \Kbold_1,
    \end{align*}
    which implies ($\spadesuit_{k-1}$). 
    Observe that ($\spadesuit_m$) is trivially true, so by induction, ($\spadesuit_0$) holds. This gives us the second desired inclusion.
\end{proof}

\begin{corollary}
\label{cor:non-brjuno}
    If $\fbold$ is not eventually Brjuno, then for every $\zeta \in \mathcal{H}_0$, there exist infinitely many $n \in \N$ such that the $n$\textsuperscript{th} landing point $\zeta_n$ of $\zeta$ satisfies
    \[
        \imag(\zeta_n; \theta_n) < \boldsymbol{C}.
    \]
\end{corollary}

\begin{proof}
    If there are infinitely many $n$'s such that $\theta_n = 0$, then for each of such $n$, we have $\imag(\zeta_n; \theta_n) < \boldsymbol{C}$ trivially.
    Else, there exists some $N \geq 0$ such that $\theta_N$ is a non-Brjuno irrational.
    In this case, for every $n\geq N$, Lemma \ref{lem:main-estimate} implies that $A^\infty_n$ is always empty. 
    This implies the desired claim.
\end{proof}

\subsection{The size of the Siegel disk}

\begin{theorem}
\label{thm:size-siegel}
    If $\theta_0$ is a Brjuno irrational, then $f_0$ admits a Siegel disk $Z(f_0)$ with conformal radius $\asymp e^{-\brjuno(\theta_0)}$ about the fixed point $0$.
\end{theorem}

The constants encoded in ``$\asymp$'' are uniform, i.e. independent of $f_0 \in \overline{\orbsec}$.
This theorem settles part (3) of Theorem \ref{side-theorem}.
Note that at this stage, we have yet determined that the interior of the Mother Hedgehog is equal to the Siegel disk.
This will be resolved at the end of the paper.

\begin{proof}
    By Lemma \ref{lem:main-estimate}, the interior of the set $A^\infty_0$ contains the half plane
    \[
    \left\{ \zeta \: : \: \imag(\zeta) > \frac{1}{2\pi} \brjuno(\theta_0) + \Kbold_2 \right\}.
    \]
    By taking the image under $\exp$, we deduce that the Mother Hedgehog $H_0(\fbold)$ has an interior component $Z$ that contains the round disk centered at $0$ with radius $e^{-\brjuno(\theta_0)-2\pi \Kbold_2}$.
    Since $f_0: (Z,0) \to (Z,0)$ is a self-homeomorphism, then by Denjoy-Wolff Theorem, $Z$ must be a Siegel disk of $f_0$ centered at $0$.
    
    Next, let us prove the upper bound.
    The idea is to find a point outside of $Z$ that has distance $\asymp e^{-\brjuno(\theta_0)}$ from the origin.
    
    By \ref{item-4}, for any $n \in \N$ and any point $\zeta$ in $\Sigma_{n+1}$,
    \begin{equation}
    \label{eqn:upper-bound-main}
        \imag(\zeta) \geq \frac{ \boldsymbol{C} + \Kbold_0}{|\theta_n|} 
        \quad \xRightarrow{\qquad} \quad
        \left| 
        \imag (\phi_n^{-1}(\zeta)) - |\theta_n| \imag(\zeta) - \frac{1}{2\pi} \log\frac{1}{|\theta_n|} 
        \right| \leq \boldsymbol{\epsilon} + \Kbold_0. 
    \end{equation}
    For every integer $n \in \N$, 
    we will pick a point $\xi_n$ on the boundary of $\hat{\mathcal{Z}}^{[0]}_n$ with the highest imaginary value. 
    By Theorem \ref{thm:uniform-fjords}, we have
    \begin{equation}
    \label{eqn:height-estimate-start}
        \left| \imag(\xi_n ; \theta_n) \right| \leq \Kbold_3,
    \end{equation}
    where $\Kbold_3>0$ is some universal constant.
    For $n \geq 1$, we will also assume without loss of generality that each of $\xi_n$ is contained in $\Sigma_n$.
    Set
    \[
        \Kbold_4 = \max\{ \boldsymbol{C} + \Kbold_0 , \Kbold_3\}.
    \]

    Pick a large enough integer $N \geq 1$ such that 
    \begin{equation}
        \label{eqn:Y-close}
        \brjuno(\theta_0) - \brjuno_N(\theta_0) \leq 1.
    \end{equation}
    Below, we will construct a sequence of points $\omega_N, \omega_{N-1}, \ldots, \omega_0$ such that
    \begin{equation}
        \left| \imag(\omega_n) - \frac{1}{2\pi} \brjuno_{N-n}(\theta_n) \right| \leq 4 \Kbold_4
    \tag{$\clubsuit_n$}
    \end{equation}
    for all $n \in \{0,\ldots,N\}$.
    
    To start off, we set $\omega_N := \xi_N$.
    The estimate ($\clubsuit_N$) follows from (\ref{eqn:height-estimate-start}).
    Inductively, for $1 \leq n \leq N$, we define $\omega_{n-1}$ out of $\omega_n$ and verify ($\clubsuit_{n-1}$) as follows. 
    \vspace{0.1in}

    \noindent \underline{Case 1:} $\imag(\omega_n) \geq \Kbold_4/|\theta_{n-1}|$.\\
    Set $\omega_{n-1} = \phi_{n-1}^{-1}(\omega_n)$. 
    Then, by (\ref{eqn:upper-bound-main}) and ($\clubsuit_n$),
        \begin{align*}
            &\left| 
            \imag(\omega_{n-1}) - \frac{1}{2\pi} \brjuno_{N-n+1}(\theta_{n-1})
            \right| \\
            =&
            \left| 
            \imag(\omega_{n-1}) - \frac{1}{2\pi} \log\frac{1}{|\theta_{n-1}|} - \frac{|\theta_{n-1}|}{2\pi} \brjuno_{N-n}(\theta_{n})
            \right| \\
            \leq & \left|
            \imag(\omega_{n-1}) - |\theta_{n-1}| \imag(\omega_n) - \frac{1}{2\pi} \log \frac{1}{|\theta_{n-1}|} \right| +
            |\theta_{n-1}| \left|
                \imag(\omega_n) - \frac{1}{2\pi} \brjuno_{N-n}(\theta_{n})
            \right| \\
            \leq & \Kbold_4 + 4 |\theta_{n-1}| \Kbold_4 \\
            <& 4 \Kbold_4.
        \end{align*}
    On the last line, we have used the inequality $|\theta_{n-1}| \leq 1/2$.
    \vspace{0.1in}

    \noindent \underline{Case 2:} $\imag(\omega_n) < \Kbold_4/|\theta_{n-1}|$. \\
    We will instead pick $\omega_{n-1}$ to be the point $\xi_{n-1}$.
    Then, by (\ref{eqn:height-estimate-start}) and ($\clubsuit_n$),
        \begin{align*}
            &\left| 
            \imag(\omega_{n-1}) - \frac{1}{2\pi} \brjuno_{N-n+1}(\theta_{n-1})
            \right| \\
            \leq &
            \left| 
            \imag(\xi_{n-1});\theta_{n-1})
            \right|
            + \frac{1}{2\pi} \left| \log\frac{1}{|\theta_{n-1}|} - \brjuno_{N-n+1}(\theta_{n-1})
            \right|  \\
            \leq & 
            \Kbold_4 + \frac{|\theta_{n-1}|}{2\pi} \brjuno_{N-n} (\theta_n) \\
            \leq &
            \Kbold_4 + |\theta_{n-1}| \left( 
            \imag(\omega_n) + 4 \Kbold_4
            \right) \\
            \leq &
            2 \Kbold_4 + 4 |\theta_{n-1}| \Kbold_4 \\
            \leq & 4 \Kbold_4.
        \end{align*}

    Now, let $n_* \in \{0,1,\ldots,N\}$ be the smallest integer such that $\omega_{n_*} = \xi_{n_*}$.
    Since $\xi_{n_*}$ is on the boundary of $\hat{ \mathcal{Z} }^{[0]}_{n_*}$, 
    the lift $w_0$ is on the boundary of $\hat{\mathcal{Z}}^{[n_*]}_0$.
    Consider the point $w= \exp(\omega'_0)$; since it is on the boundary of $\hat{Z}^{[n_*]}_0$, it lies outside of the Siegel disk $Z(f_0)$ of $f_0$.
    By ($\clubsuit_0$) and (\ref{eqn:Y-close}), we have 
    \[
    |w| \leq e^{8\pi \Kbold_4 +1} e^{-\brjuno(\theta_0)}.
    \]
    This inequality implies the desired upper bound for the conformal radius of $Z(f_0)$ at $0$.
\end{proof}

\subsection{Typical orbits outside of the Siegel set}

From now on,  when $\theta_n$ is a Brjuno irrational for some $n \geq 0$, we will denote by $\mathcal{Z}_n(\fbold)$ the lift of the Siegel disk $Z(f_n)$ under $\exp$.
Below, we introduce a generalization of Siegel disks.

\begin{definition}
    Suppose $\fbold$ is eventually Brjuno. 
    Let $N \geq 0$ is the smallest integer such that $\theta_N$ is a Brjuno irrational. 
    For all $n \geq 0$ with $n < N$, we define the $n$\textsuperscript{th} \emph{Siegel set} of $\fbold$ \emph{in logarithmic coordinates} to be
    \[
        \mathcal{Z}_n = \mathcal{Z}_n(\fbold) := \{ \zeta \in \mathcal{H}_n \: : \: \zeta_{N-n} \in \mathcal{Z}_N(\fbold) \},  
    \]
    where $\zeta_{N-n}$ refers to the $(N-n)$\textsuperscript{th} landing point of $\zeta$.
    If $\fbold$ is not eventually Brjuno, we simply define $\mathcal{Z}_n$ to be the empty set for all $n \geq 0$.
\end{definition}

The set $\mathcal{Z}_n$ consists of points that project into the Siegel disk of a sufficiently high renormalization. 
In the usual coordinates, the \emph{Siegel set} of $\fbold$ is
\[
    Z_n = Z_n(\fbold) := \exp( \mathcal{Z}_n ).
\]

The next proposition will be the key towards proving zero area.

\begin{proposition}
\label{prop:brjuno}
    For almost every point $\zeta$ in $\mathcal{H}_0 \backslash \mathcal{Z}_0$, there exist infinitely many $n \in \N$ such that
    \[
        \imag(\zeta_n; \theta_n) <  C.
    \]
\end{proposition}

\begin{proof}
    Based on Corollary \ref{cor:non-brjuno}, we can assume that we are in the eventually Brjuno situation.
    Let $N \geq 0$ be the smallest integer such that $\theta_N$ is a Brjuno irrational.
    Consider the cylinder $\C/\Z$ equipped with the Lebesgue measure $\lambda$. 
    For $n\geq N$, consider the quotient 
    \[
        U_n := \left (A^\infty_n \backslash \mathcal{Z}(\ftilde_n) \right)/\Z,
    \]
    which lives in $\C /\Z$.
    To prove the proposition, it suffices to show that $U_n$ has zero area (with respect to $\lambda$) for all $n \geq N$.
    
    We gather from Lemma \ref{lem:main-estimate} and Theorem \ref{thm:size-siegel} that
    \[
        - \Kbold_1 \leq\imag(\zeta) - \frac{1}{2\pi}\brjuno(\theta_n) \leq \Kbold_2 \qquad \text{ for all } n \geq N \text{ and } \zeta \in U_n.
    \]
    Hence,
    \begin{equation}
        \label{eqn:bdd-height}
        \textnormal{Area}(U_n) \leq \Kbold_1 + \Kbold_2 \qquad \text{ for all } n \geq N.
    \end{equation}
    
    Let us fix $n \geq N$. 
    Pull back $U_{n+1}$ via $\phi_{n}$ and obtain a subset $U'_n$ of $U_n \cap \mathcal{S}_n$. 
    By \ref{item-3}, we have
    \[
        \textnormal{Area} \left( U_{n+1} \right) 
        = \int_{U'_{n}} \left| (\phi_{n})'(\zeta)\right|^2 d\lambda(\zeta) 
        \geq \left( \frac{1-{\boldsymbol{\epsilon}}}{|\theta_n|} \right)^2 \textnormal{Area} \left( U'_n \right). 
    \]
    Observe that from the definition of the set $A^\infty_n$, we have
    \[
        U_n \subset \bigcup_{j=0}^{\bar{a}_{n+1}} \ftilde_n^{-j} (U'_n).
    \]
    Then, by \ref{item-1}, we have
    \begin{align*}
        \textnormal{Area} \left(U_n \right) 
        &\leq \sum_{j=0}^{\bar{a}_{n+1}} \textnormal{Area} \left( \ftilde_n^{-j} (U'_n) \right) \\
        &\leq \sum_{j=0}^{\bar{a}_{n+1}} \int_{U'_n} |(\ftilde_n^{-j})'(w)|^2 d\lambda(w) 
        \\
        &\leq (\bar{a}_{n+1}+1) (1+{\boldsymbol{\epsilon}} |\theta_n|)^{2\bar{a}_{n+1}}  \textnormal{Area}(U'_n) \\
        &\leq (\bar{a}_{n+1}+1) e^{2{\boldsymbol{\epsilon}}} \textnormal{Area}(U'_n).
    \end{align*}
    Note that the last inequality above follows from the inequality $\bar{a}_{n+1} |\theta_n| \leq 1$.
    Hence,
    \[
        \textnormal{Area} \left( U_n \right) 
        \leq (\bar{a}_{n+1}+1) e^{2{\boldsymbol{\epsilon}}} \left( \frac{|\theta_n|}{1-{\boldsymbol{\epsilon}}}\right)^2 
        \cdot \textnormal{Area} \left( U_{n+1} \right). 
    \]
    The inequality $\bar{a}_{n+1} |\theta_n| \leq 1$ implies $(\bar{a}_{n+1}+1) |\theta_n|^2 \leq 3/4$. 
    Then, by our choice of ${\boldsymbol{\epsilon}}$ in (\ref{eqn:epsilon}), we end up with
    \[
        \textnormal{Area} \left( U_n \right) \leq 0.9 \cdot \textnormal{Area} \left( U_{n+1} \right) \quad \text{ for all } n\geq N.
    \]
    Together with (\ref{eqn:bdd-height}), we conclude that $U_n$ must have zero area for all $n \geq N$.
\end{proof}

\subsection{Definite holes}

Consider the pseudo-Siegel disks $\hat{\mathcal{Z}}^{[m]}_n$ of $\ftilde_n$ in logarithmic coordinates, and the constant $\boldsymbol{C}>0$ used in the previous subsection. 
Again, we denote the unit translation by $T(\zeta)=\zeta+1$.

\begin{lemma}
\label{lem:porosity-lemma}
    There exists a universal constant $\tau > 1$ such that for every $n \geq 1$ and every point $\zeta$ in $\mathcal{H}_n$ satisfying 
    \[
        \imag(\zeta;\theta_n) < \boldsymbol{C},
    \]
    there exists a pair of open disks $B_\zeta$ and $D_\zeta$ in $\UHP$ with the following properties.
    \begin{enumerate}
        \item $B_\zeta$ is a round disk $B_\zeta = \D(b_\zeta, r_\zeta)$ in the complement of $\hat{\mathcal{Z}}^{[0]}_n$ and it satisfies
        \[
            \tau^{-1} r_\zeta \leq |b_\zeta - \zeta| \leq \tau r_\zeta.
        \]
        \item $B_\zeta \cup \{\zeta\}$ is compactly contained in $D_\zeta$ and, with respect to the hyperbolic metric of $D_\zeta$, the diameter of $B_\zeta \cup \{\zeta\}$ is bounded above by $\tau$.
        \item There exists an open disk $D'_\zeta \subset \Dom(\ftilde_n)$ such that the map 
        \begin{equation}
        \label{eqn:branched-covering-map}
            \ftilde_n : (D'_\zeta, \zeta') \to (D_\zeta, \zeta)
        \end{equation}
        is a branched covering map sending some point $\zeta'$ in $\mathcal{H}_n \cap D'_\zeta$ to $\zeta$. 
        \item If $\zeta$ is a critical value of $\ftilde_n$, i.e. in $\mathbf{c}_1(\ftilde_n)+ \Z$, then the critical point $\zeta'$ is the unique branched point of the map \textnormal{(\ref{eqn:branched-covering-map})} and $D_\zeta$ is contained in $T^j(\Sigma_n)$ for some $n \in \Z$.
        Otherwise, $D_\zeta$ avoids $\mathbf{c}_1(\ftilde_n)+ \Z$, the map \textnormal{(\ref{eqn:branched-covering-map})} has degree one, and there is some $j \in \Z$ such that $D_\zeta$ is contained in the interior of the closure of $T^j(\Sigma_n) \cup T^{j+1}(\Sigma_n)$.
    \end{enumerate} 
\end{lemma}

\begin{proof}
    Suppose $\zeta$ is a critical value of $\ftilde_n$, say $\mathbf{c}_1 = \mathbf{c}_1(\ftilde_n)$. This case is straightforward because it lies on the boundary of the top pseudo-Siegel disk $\hat{\mathcal{Z}}^{[-1]}_n$ of $\ftilde_n$. By Theorem \ref{thm:renorm-limits} (4), there exists a uniform constant $\nu>0$ and a uniformly quasiconformal map $u_n: \UHP \to \UHP$ that commutes with $T$ and sends $\hat{\mathcal{Z}}^{[-1]}_n$ onto the upper half plane $\{\imag z > \nu \}$.
    Pick $D_\zeta$ to be any sufficiently small quasidisk in $\Sigma_n$ such that $u_n(D_\zeta)$ is a round disk centered at $u_n(\zeta)$.
    One half of $D_\zeta$ is outside of $\hat{\mathcal{Z}}^{[-1]}_n$ and it contains a disk $B_\zeta$ with the desired properties.
    Properties (1)--(4) are immediate.

    Now, consider the general case when $\zeta$ is not a critical value. By Theorem \ref{thm:renorm-limits} (6)(d), there exists a uniformly quasiconformal map $\Xi : \UHP \to \UHP$ that commutes with $T$ and sends the strip $\Sigma_n$ to the straight half-strip $\{0 < \real(w) < 1, \imag(w) > 0\}$. 
    From now on, objects in $\Xi$-coordinate will be denoted with a dagger, e.g. $\zeta^\dagger = \Xi(\zeta)$, $\mathbf{c}_1^{\dagger} = \Xi(\mathbf{c}_1)$, and $\hat{\mathcal{Z}}_n^{[-1],\dagger} = \Xi(\hat{\mathcal{Z}}_n^{[-1]})$.
    By composing with some integer translation, we will assume without loss of generality that
    \[
        \frac{1}{2} \leq \real(\zeta^\dagger) < \frac{3}{2}.
    \]
    
    Let $\threshold$ be the combinatorial threshold for the pseudo-Siegel disks. 
    Recall from Theorem \ref{thm:renorm-limits}(4)(c) that there exists a universal constant $R>0$ such that
    \[
        \hat{\mathcal{Z}}^{[-1]}_n \supset \left\{ w \in \C \: : \: \imag (w) > \frac{1}{2\pi} \log\frac{1}{R} \right\}.
    \]
    We will fix a sufficiently large constant $\lambda > 0$ such that 
    \begin{equation}
        \label{eqn:lambda}
        0 \leq \lambda - \max\left\{ \frac{1}{2\pi} \log\frac{1}{R}, \frac{1}{2\pi} \log \threshold + \boldsymbol{C} \right\} = O(1).
    \end{equation}
    The construction of the disks $B_\zeta$ and $D_\zeta$ is split into two cases.
    The estimates that appear below may depend on $\lambda$, but this is alright.
    \vspace{0.05in}
    
    \noindent \underline{Case 1:} $\imag(\zeta) \leq \lambda$.

    Let $E_*^\dagger$ be the largest open round disk that is centered at $\zeta^\dagger$, has radius $\leq \frac{1}{4}$, and is disjoint from $\mathbf{c}_1^\dagger + \Z$.
    Consider the union of disks 
    \[
        E^\dagger := \bigcup_{j \in \Z} T^j(E_*^\dagger).
    \]
    Let $\alpha_1^\dagger$ be the unique vertical line segment (possibly degenerate) in $\hat{\mathcal{Z}}^{[-1],\dagger}_n \backslash E^\dagger$ with upper endpoint $\zeta^\dagger$ and lower endpoint being a point on the boundary of $ \hat{\mathcal{Z}}^{[-1],\dagger}_n \backslash E^\dagger $. 
    Let $\alpha_2^\dagger$ be a circular subarc of $\partial E^\dagger \cap \hat{\mathcal{Z}}^{[-1],\dagger}_n$ (possibly degenerate) that starts from the endpoint of $\alpha_1^\dagger$ and ends at a point $y^\dagger$ on the boundary of $\hat{\mathcal{Z}}^{[-1],\dagger}_n$. 
    The arcs  $\alpha^\dagger := \alpha_1^\dagger \cup \alpha_2^\dagger$ and $\alpha := \Xi^{-1}(\alpha^\dagger)$
    have the following properties.
    \begin{itemize}
        \item $\alpha$ is contained in $\hat{\mathcal{Z}}^{[-1]}_n$ and it starts from $\zeta$ and ends at a point $y=\Xi^{-1}(y^\dagger)$ on the boundary of $\hat{\mathcal{Z}}^{[-1]}_n$.
        \item $\alpha^\dagger$ is contained in the strip $\{ \frac{1}{4} \leq \real (w) \leq \frac{7}{4}\}$.
        \item There exists a uniform constant $K > 1$ such that
        \[
            K^{-1} \dist(\alpha, \mathbf{c}_1 + \Z) \leq \diam(\alpha) \leq K.
        \]
    \end{itemize}
    
    Consequently, there exists an open disk neighborhood $D_\zeta$ of $\alpha$ such that $D_\zeta$ is contained in $\Sigma_n \cup T(\Sigma_n)$, avoids $\mathbf{c}_1 + \Z$, and $D_\zeta \backslash \alpha$ is an annulus with definite modulus. 
    We can further arrange such that the intersection $D_\zeta \cap \partial \hat{\mathcal{Z}}^{[-1]}_n$ is an interval neighborhood of $y$.
    With the definiteness of $D_\zeta \backslash \alpha$, we can also find an open round disk $B_\zeta$ near $y$ inside of $D_\zeta \backslash \hat{\mathcal{Z}}_n^{[-1]}$ that satisfies items (1) and (2).
    Assuming $D_\zeta$ is not too large, we can also assume that $D_\zeta$ has the lifting property described in item (3). 
    Lastly, item (4) follows from the construction.    
    \vspace{0.05in}

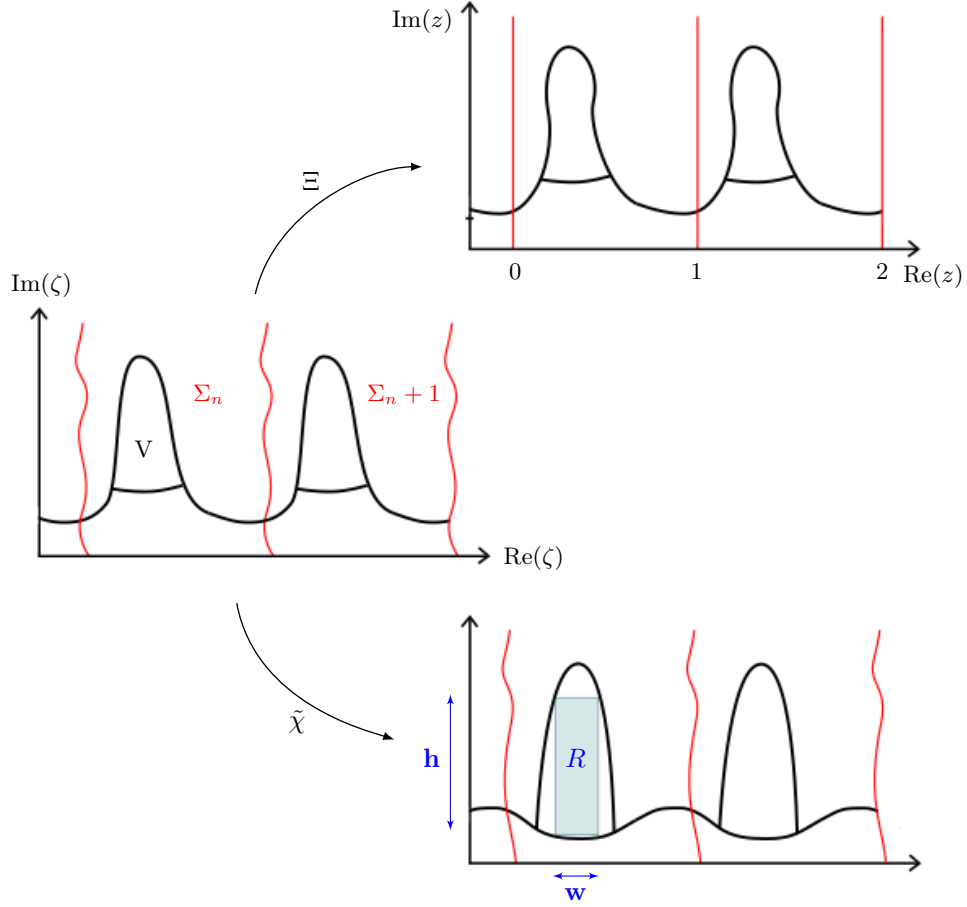
\begin{figure}
    \centering
    
    \begin{tikzpicture}
    \node[anchor=south west, inner sep=0] (image) at (0,0) {\includegraphics[width=0.97\linewidth]{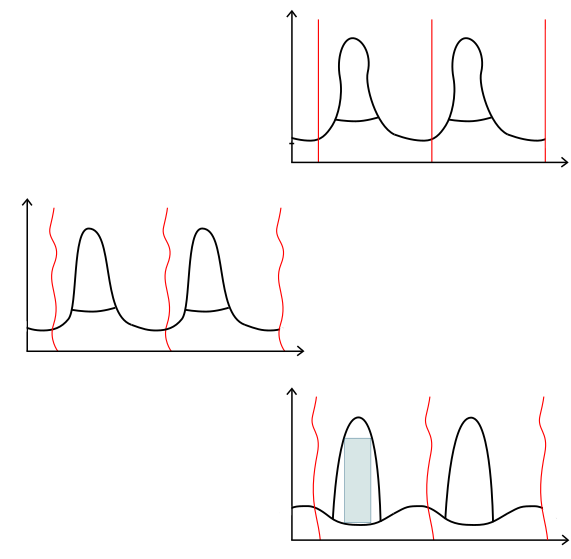}};
    \begin{scope}[
        x={(image.south east)},
        y={(image.north west)}
    ]   
        \node [blue, font=\bfseries] at (0.625,0.15) {$R$};
        \node [black, font=\bfseries] at (0.56,0.685) {\small $0$};
        \node [black, font=\bfseries] at (0.755,0.685) {\small $1$};
        \node [black, font=\bfseries] at (0.955,0.685) {\small $2$};
        \node [red, font=\bfseries] at (0.23,0.55) {\small $\Sigma_n$};
        \node [red, font=\bfseries] at (0.44,0.55) {\small $\Sigma_n+1$};
        \node [black, font=\bfseries] at (0.16,0.49) {\small $\textnormal{V}$};
        
        \draw[black,line width=0.5pt,-latex] (0.28,0.66) .. controls (0.3,0.75) and (0.4,0.80) .. (0.46,0.8);        
        \node [black, font=\bfseries] at (0.34,0.785) {$\Xi$};
        \draw[black,line width=0.5pt,-latex] (0.26,0.32) .. controls (0.28,0.21) and (0.4,0.18) .. (0.43,0.17);        
        \node [black, font=\bfseries] at (0.325,0.19) {$\tilde{\chi}$};
        
        \draw [blue, latex'-latex'] (0.49,0.07) -- (0.49,0.22);
        \node [blue, font=\bfseries] at (0.47,0.15) {$\mathbf{h}$};
        \draw [blue, latex'-latex'] (0.60,0.02) -- (0.65,0.02);
        \node [blue, font=\bfseries] at (0.625,0.00) {$\mathbf{w}$};

        \node [black, font=\bfseries] at (0.05,0.67) {\small $\imag(\zeta)$};
        \node [black, font=\bfseries] at (0.58,0.37) {\small $\real(\zeta)$};
        \node [black, font=\bfseries] at (0.46,0.96) {\small $\imag(z)$};
        \node [black, font=\bfseries] at (1.01,0.68) {\small $\real(z)$};
    \end{scope}
\end{tikzpicture}
    
    \caption{The maps $\Xi$ and $\tilde{\chi}$}
    \label{fig:xi-and-rectangle}
\end{figure}

    \noindent \underline{Case 2:} $\lambda < \imag(\zeta) \leq 
    \frac{1}{2\pi} \log\frac{1}{|\theta_n|} + \boldsymbol{C}$.

    By (\ref{eqn:lambda}), $|\theta_n| < \frac{1}{\threshold}$ and so $f_n$ has to be near-parabolic. 
    Consider the quasiconformal uniformization $\chi : (\hat{Z}^{[-1]}_n, 0) \to (\overline{\D}, i\theta_n)$ associated to $f_n$ as described in Theorem \ref{thm:uniform-fjords}. It lifts to a uniformly quasiconformal map $\tilde{\chi}$ on $\hat{\mathcal{Z}}^{[-1]}_n$ such that
    \[
        \chi \circ \exp = \exp \circ \tilde{\chi} + i\theta_n.
    \]
    
    Consider the connected component $\textnormal{V}$ of the closure of $\hat{\mathcal{Z}}^{[-1]}_n \backslash \hat{\mathcal{Z}}^{[0]}_n$ that is contained in $\Sigma_n \cup T(\Sigma_n)$ and is the closest to $\zeta$ in $\Xi$-coordinates. 
    Inside of $\tilde{\chi}(\textnormal{V})$, there is a straight rectangle $R$ of horizontal width $\mathbf{w} \asymp 1$ and vertical height $\mathbf{h}=\frac{1}{2\pi} \log \frac{1}{|\theta_n|} - O(1)$ such that the two bottom vertices are on the boundary of $\tilde{\chi}(\hat{\mathcal{Z}}^{[-1]}_n)$. See Figure \ref{fig:xi-and-rectangle}.
    There exists a point $y \in \tilde{\chi}^{-1}(R)$ such that
    \[
        0 \leq \imag(\zeta^\dagger)-\imag(y^\dagger) = O(1) \quad \text{ and } \quad \dist(\tilde{\chi}(y), \partial R) \asymp 1,
    \]
    where $y^\dagger = \Xi(y)$.
    Denote 
    \[
    \xi^\dagger := \real(\zeta^\dagger) + i \, \imag(y^\dagger).
    \]
    Let $\alpha_1^\dagger$ be the vertical line segment connecting $\zeta^\dagger$ and $\xi^\dagger$, and let $\alpha_2^\dagger$ the horizontal line segment connecting $\xi^\dagger$ and $y^\dagger$. 
    Then, the arc $\alpha := \Xi^{-1}(\alpha_1^\dagger \cup \alpha_2^\dagger)$ joins $\zeta$ and $y$ and has definite diameter $\asymp 1$.
    We set $D_\zeta$ to be the $\varepsilon$-neighborhood of $\alpha$, where $\varepsilon>0$ is a some definite constant such that $D_\zeta$ is still contained in both $\Sigma_n \cup T(\Sigma_n)$ and $\hat{\mathcal{Z}}^{[-1]}_n$.
    By design, $D_\zeta$ contains a round disk $B_\zeta$ centered at $y$ satisfying items (1) and (2).
    Lastly, items (3) and (4) follow from the construction.
\end{proof}

\subsection{Porosity}

We are now ready to prove the main results in this paper.

\begin{definition}
    Given a closed subset $X$ of the plane, we say that $X$ is \emph{(non-uniformly) porous} at a point $\zeta$ in $X$ if there exists $\varepsilon>0$ and a sequence of positive numbers $r_n$ decreasing to zero such that for every $n \geq 1$, the round disk $\D(\zeta,r_n)$ contains another round disk in $\C \backslash X$ of radius $\varepsilon r_n$.
\end{definition}
 
Every point of porosity of a set $X$ is not a Lebesgue density point of $X$. 
So if almost every point in $X$ is non-uniformly porous, then $X$ must have zero Lebesgue measure.

\begin{theorem}
\label{thm:almost-main-thm}
    The Mother Hedgehog $\mathcal{H}_0$ is porous at almost every point not in the Siegel set $\mathcal{Z}_0$.
\end{theorem}

This theorem implies Theorems \ref{main-thm-01} and \ref{main-thm-02}. It also implies that the Siegel set $Z_0(\fbold)$
is indeed equal to the interior of the Mother Hedgehog $H_0(\fbold)$; in particular, there are no wandering domains.

\begin{proof}
Let $\boldsymbol{C}>0$ be the constant from \S\ref{ss:setup} and pick a typical point $\zeta$ in $\mathcal{H}_0 \backslash \mathcal{Z}_0$.
Denote
\[
    \mathfrak{I}(\zeta) := \{ n \geq 1 \: : \: \imag(\zeta_n;\theta_n) < \boldsymbol{C} \},
\]
where $\zeta_n \in \mathcal{H}_n$, $n \geq 0$ is the sequence of landing points of $\zeta$. 
According to Proposition \ref{prop:brjuno}, typically $\mathfrak{I}(\zeta)$ is an infinite set. 

Recall the sectors $\mathcal{S}_n^m$ from \S\ref{ss:log-coordinates}.
For every $n \in \mathfrak{I}(\zeta)$, there exists a unique smallest element $s_n \in \Time^0 \cup \{0\}$ such that $\omega_n := \ftilde_0^{s_n}(\zeta)$ is contained in $T^{j}(\mathcal{S}_0^n)$ for some integer $j$. 
Let $\Upsilon_{n,j}: \Sigma_{n} \to T^j(\mathcal{S}_0^n)$ be the inverse of the composition
\begin{center}
    \begin{tikzcd}[column sep=large, row sep=large]
	T^j(\mathcal{S}_0^n) & \mathcal{S}_1^{n-1} & \ldots & \mathcal{S}_{n-1}^1 & \Sigma_{n}. 
	\arrow["\phi_0", from=1-1, to=1-2]
	\arrow["\phi_1", from=1-2, to=1-3]
	\arrow["\phi_{n-2}", from=1-3, to=1-4]
	\arrow["\phi_{n-1}", from=1-4, to=1-5]
\end{tikzcd}
\end{center}
Observe that $\Upsilon_{n,j}(\omega_n) = \zeta_n$.

For $n \in \mathfrak{I}(\zeta)$, consider the disks $B_{\zeta_n}$ and $D_{\zeta_n}$ from Lemma \ref{lem:porosity-lemma} associated to $\zeta_n$.
By items (3)--(4) of the lemma, the map $\Upsilon_{n,j}$ uniquely extends to a univalent map on $D_{\zeta_n} \cup \Sigma_n$ and it sends the pair $(D_{\zeta_n}, B_{\zeta_n})$ onto a pair of nested disks $(D'_n, B'_n)$ around $\omega_n$.
Let $D_n$ be the lift of $D'_n$ under $\ftilde_0^{s_n}$ containing $\zeta$.
By Lemma \ref{lem:porosity-lemma} (3)--(4), the map $\ftilde_0^{s_n}: D_n \to D'_n$ is either univalent or a double covering map branched at $\zeta$. 
By Lemma \ref{lem:porosity-lemma} (1)--(2) together with Koebe distortion control, the corresponding lift of $B'_n$ under this map contains a round disk $B_n$ in the complement of $\mathcal{H}_0$ with radius comparable to its distance to $\zeta$. 
As $n \to \infty$ in $\mathfrak{I}(\zeta)$, the diameter of $B_n$ has to shrink to zero. This implies that $\mathcal{H}_0$ is indeed porous at $\zeta$.
\end{proof}

As an immediate consequence, we have:

\begin{corollary}
\label{cor:interior-mother}
    The interior of $\mathcal{H}_0$ is equal to the Siegel set $\mathcal{Z}_0$.
\end{corollary}

Recall that $\mathcal{Z}_0$ is defined to be the empty set when $\fbold$ is not eventually Brjuno. 
As such, this corollary implies the optimality of the Brjuno condition.
Together with Theorem \ref{thm:size-siegel}, this also completes the proof of Theorem \ref{side-theorem}.


\bibliographystyle{alpha}
 
{\small \bibliography{bibliography}}

@article {AO93,
    AUTHOR = {Aarts, Jan M. and Oversteegen, Lex G.},
     TITLE = {The geometry of {J}ulia sets},
   JOURNAL = {Trans. Amer. Math. Soc.},
    VOLUME = {338},
      YEAR = {1993},
    NUMBER = {2},
     PAGES = {897--918}
}

@incollection {AM05,
    AUTHOR = {Avila, Artur and Moreira, Carlos Gustavo},
     TITLE = {Phase-parameter relation and sharp statistical properties for general families of unimodal maps},
 BOOKTITLE = {Geometry and dynamics},
    SERIES = {Contemp. Math.},
    VOLUME = {389},
     PAGES = {1--42},
 PUBLISHER = {Amer. Math. Soc., Providence, RI},
      YEAR = {2005},
      ISBN = {0-8218-3851-2},
   MRCLASS = {37E05 (37E20 37F25 82C05)},
  MRNUMBER = {2181956},
MRREVIEWER = {Masato\ Tsujii},
       DOI = {10.1090/conm/389/07270},
       URL = {https://doi-org.brown.idm.oclc.org/10.1090/conm/389/07270},
}

@article{Bi16,
title={Positive area and inaccessible fixed points for hedgehogs}, 
volume={36}, 
number={6}, 
journal={Ergod. Th. Dynam. Sys.}, 
author={Biswas, Kingshook}, 
year={2016}, 
pages={1839--1850}
}

@article{Br71,
    author = {Brjuno, A. D.},
    title = {Analytical form of differential equations},
    journal = {Trans. Moscow Math. Soc.},
    year = {1971},
    volume = {25},
    pages = {131--288}
}

@article{BC04,
author = {Buff, Xavier and Ch\'eritat, Arnaud},
year = {2004},
pages = {1--24},
title = {Upper bound for the size of quadratic {Siegel} disks},
volume = {156},
journal = {Invent. Math.}
}

@article{BC12,
author = {Buff, Xavier and Ch\'eritat, Arnaud},
year = {2012},
pages = {673--746},
title = {Quadratic {Julia} sets with positive area},
volume = {176},
journal = {Ann. Math.}
}

@article{Che13,
author = {Cheraghi, Davoud},
year = {2013},
pages = {999--1035},
title = {Typical orbits of quadratic polynomials with a neutral fixed point: {Brjuno} type},
volume = {332},
number = {3},
journal = {Commun. Math. Phys.}
}

@article{Che19,
author = {Cheraghi, Davoud},
year = {2019},
pages = {59--138},
title = {Typical orbits of quadratic polynomials with a neutral fixed point: {Non-Brjuno} type},
volume = {52},
number = {1},
journal = {Ann. Sci. \'{E}c. Norm. Sup.}
}

@article{Che23,
    AUTHOR = {Cheraghi, Davoud},
     TITLE = {Arithmetic geometric model for the renormalisation of irrationally indifferent attractors},
   JOURNAL = {Nonlinearity},
    VOLUME = {36},
      YEAR = {2023},
    NUMBER = {12},
     PAGES = {6403--6475}
}

@article{Che25,
author = {Cheraghi, Davoud},
year = {2025},
pages = {1321--1383},
title = {Topology of irrationally indifferent attractors},
volume = {6},
number = {58},
journal = {Ann. Sci. \'{E}c. Norm. Sup.}
}

@misc{CS15,
author = {Cheraghi, Davoud and Shishikura, Mitsuhiro},
year = {2015},
title = {Satellite renormalisation of quadratic polynomials},
publisher = {arXiv},
eprint = {1509.07843},
doi = {10.48550/arXiv.1509.07843},
note = {\href{https://arxiv.org/abs/1509.07843}{arXiv.1509.07843}}
}

@misc{CDY,
author = {Cheraghi, Davoud and de Zotti, Alexandre and Yang, Fei},
publisher = {arXiv},
year = {2020},
copyright = {arXiv.org perpetual, non-exclusive license},
title = {Dimension paradox of irrationally indifferent attractors},
eprint = {2003.12340},
doi = {10.48550/arXiv.2003.12340},
note = {\href{https://arxiv.org/abs/2003.12340}{arXiv.2003.12340}}
}

@incollection{D83,
author = {Douady, Adrien},
title = {System\`es dynamiques holomorphes},
booktitle = {S\'eminaire Bourbaki : volume 1982/83, expos\'es 597-614},
series = {Ast\'erisque},
pages = {39--63},
year = {1983},
publisher = {Soci\'et\'e math\'ematique de France},
number = {105-106},
}

@incollection{D87,
author = {Douady, Adrien},
title = {Disques de {S}iegel et anneaux de {H}erman},
booktitle = {S\'eminaire Bourbaki : volume 1986/87, expos\'es 669-685},
journal = {Ast\'{e}risque},
series = {Ast\'erisque},
publisher = {Soci\'et\'e math\'ematique de France},
number = {152--153},
year = {1987},
pages = {4, 151--172},
issn = {0303-1179,2492-5926}
}

@article{DLS,
author = {Dudko, Dzmitry and Lyubich, Mikhail and Selinger, Nikita},
year = {2020},
pages = {653--733},
title = {Pacman renormalization and self-similarity of the {Mandelbrot} set near {Siegel} parameters},
volume = {33},
number = {3},
journal = {J. Amer. Math. Soc.}
}

@article{DL23,
author = {Dudko, Dzmitry and Lyubich, Mikhail},
year = {2023},
title = {{Local connectivity of the Mandelbrot set at some satellite parameters of bounded type}},
journal = {Geom. Funct. Anal.},
volume = {33},
pages={912--1047}
}

@misc{DL22,
author = {Dudko, Dzmitry and Lyubich, Mikhail},
publisher = {arXiv},
year = {2022},
copyright = {arXiv.org perpetual, non-exclusive license},
title = {Uniform a priori bounds for neutral renormalization},
eprint = {2210.09280},
doi = {10.48550/arXiv.2210.09280},
note = {\href{https://arxiv.org/abs/2210.09280}{arXiv.2210.09280}}
}

@article{DL26a,
author = {Dudko, Dzmitry and Lyubich, Mikhail},
journal={Publ. Math. Inst. Hautes {\'E}tudes Sci.},
year = {2026},
title = {{MLC at Feigenbaum points}},
volume = {143},
pages = {97--141}
}

@misc{DL26b,
author = {Dudko, Dzmitry and Lyubich, Mikhail},
year = {2026},
title = {Uniform a priori bounds for neutral renormalization. {Variation I:} Sector Renormalization},
note = {Manuscript}
}

@misc{DLL,
author = {Dudko, Dzmitry and Lim, Willie Rush and Lyubich, Mikhail},
title = {Rigidity of the attractor of neutral renormalization},
note = {Manuscript},
year = {2026}
}

@article{G84,
author = {Ghys, {\'E}tienne},
year = {1984},
pages = {385--388},
title = {Transformations Holomorphes au Voisinage d'Une Courbe de {Jordan}},
volume = {298},
number = {16},
journal = {C. R. Acad. Sci. Paris S\'er. I Math.}
}

@misc{IS,
author = {Inou, Hiroyuki and Shishikura, Mitsuhiro},
publisher = {arXiv},
year = {2006},
title = {The renormalization for parabolic fixed points and their perturbation},
note = {\href{https://www.math.kyoto-u.ac.jp/~mitsu/pararenorm/ParabolicRenormalization.pdf}{Preprint}}
}

@misc{K06,
author = {Kahn, Jeremy},
publisher = {arXiv},
year = {2006},
title = {A Priori Bounds for Some Infinitely Renormalizable Quadratics: I. Bounded Primitive Combinatorics},
note = {\href{https://arxiv.org/abs/math/0609045}{arXiv:math/0609045}},
eprint = {math/0609045}
}

@misc{Lim26,
author = {Lim, Willie Rush},
year = {2026},
title = {The combinatorics of sector renormalization},
publisher = {arXiv},
copyright = {arXiv.org perpetual, non-exclusive license},
eprint = {2607.11408},
doi = {10.48550/arXiv.2607.11408},
note = {\href{https://arxiv.org/abs/2607.11408}{arXiv:2607.11408}}
}

@misc{Lim,
author = {Lim, Willie Rush},
title = {The topology of {Mother Hedgehogs}},
note = {Manuscript in preparpation}
}

@misc{Lyu91,
author = {Lyubich, Mikhail},
year = {1991},
title = {On the {Lebesgue} measure of the {Julia} set of a quadratic polynomial},
doi = {10.48550/arXiv.math/9201285},
note = {Preprint IMS at Stony Brook, 1991/10, \href{https://arxiv.org/abs/math/9201285}{arXiv:math/9201285}}
}

@article{L99,
author = {Lyubich, Mikhail},
year = {1999},
pages = {319--420},
title = {{Feigenbaum-Coullet-Tresser} Universality and {Milnor}'s Hairiness Conjecture},
volume = {149},
number = {2},
journal = {Ann. Math. (2)}
}

@article{L02,
    AUTHOR = {Lyubich, Mikhail},
     TITLE = {Almost every real quadratic map is either regular or stochastic},
   JOURNAL = {Ann. Math. (2)},
    VOLUME = {156},
      YEAR = {2002},
    NUMBER = {1},
     PAGES = {1--78}
}

@book{McM94, 
    author = {McMullen, Curtis},
    place = {Princeton, NJ}, 
    publisher = {Princeton University Press},
    title = {Complex Dynamics and Renormalization},
    year = {1994},
    isbn = {9780691029818}
}

@book{McM96, 
    author = {McMullen, Curtis},
    place = {Princeton, NJ}, 
    publisher = {Princeton University Press},
    title = {Renormalization and 3-Manifolds Which Fiber over the Circle},
    year = {1996},
    isbn = {9780691011530}
}

@article{McM98,
author = {McMullen, Curtis},
year = {1998},
pages = {247--292},
title = {Self-similarity of {Siegel} disks and {Hausdorff} dimension of {Julia} sets},
volume = {180},
number = {2},
journal = {Acta Math.}
}

@incollection{Pal00,
  author    = {Palis, Jacob},
  title     = {A global view of dynamics and a conjecture on the denseness of finitude of attractors},
  booktitle = {G{\'e}om{\'e}trie complexe et syst{\`e}mes dynamiques - Colloque en l'honneur d'Adrien Douady Orsay, 1995},
  year      = {2000},
  publisher = {Soci{\'e}t{\'e} math{\'e}matique de France},
  pages     = {335--347}
}

@article{Pe96,
  title={Local connectivity of some {Julia} sets containing a circle with an irrational rotation},
  author={Petersen, Carsten Lunde},
  journal={Acta Math.},
  year={1996},
  volume={177},
  pages={163--224}
}

@article{PZ04,
  title={On the {Julia} set of a typical quadratic polynomial with a {Siegel} disk},
  author={Petersen, Carsten Lunde and Zakeri, Saeed},
  journal={Ann. Math.},
  year={2004},
  volume={159},
  pages={1--52}
}

@article{Shi98,
author = {Shishikura, Mitsuhiro},
year = {1998},
pages = {225--267},
title = {The {Hausdorff} dimension of the boundary of the {Mandelbrot} Set and {Julia} Sets},
volume = {147},
number={2},
journal = {Ann. Math. (2)}
}

@inproceedings{ShiICM,
  author    = {Shishikura, Mitsuhiro},
  title     = {Topological, complex and arithmetic dynamics on trees},
  booktitle = {Proceedings of the International Congress of Mathematicians, Vol. 1, 2 (Z\"urich, 1994)},
  pages     = {886--895},
  publisher = {Birkh\"auser},
  address   = {Basel},
  year      = {1995},
  editor    = {Chatterji, S. D.}
}

@article{SY24,
    author = {Shishikura, Mitsuhiro and Yang, Fei},
    title = {The high type quadratic {Siegel} disks are {Jordan} domains},
    journal = {J. Eur. Math. Soc.},
    year = {2024},
    volume = {27},
    number = {11},
    pages = {4501--4562}
}

@incollection{Yoc95,
     author = {Yoccoz, Jean-Christophe},
     title = {Th\'eor\`eme de {Siegel,} nombres de {Bruno} et polyn\^omes quadratiques},
     booktitle = {Petits diviseurs en dimension 1},
     series = {Ast\'erisque},
     pages = {1--88},
     publisher = {Soci\'et\'e math\'ematique de France},
     number = {231},
     year = {1995}
}

\end{document}